\documentclass[11pt, twoside, leqno]{article}
\usepackage{enumerate}
\usepackage{graphics}
\usepackage{graphicx}
\usepackage{amssymb}
\usepackage{amsmath}
\usepackage{amsthm}
\usepackage{color}
\usepackage{mathrsfs}

\usepackage{indentfirst}

\usepackage{txfonts}

\allowdisplaybreaks

\newtheorem{theorem}{Theorem}[section]
\newtheorem{lemma}[theorem]{Lemma}

\newtheorem{assumption}[theorem]{Assumption}
\newtheorem{proposition}[theorem]{Proposition}

\theoremstyle{definition}
\newtheorem{remark}[theorem]{Remark}
\newtheorem{definition}[theorem]{Definition}
\renewcommand{\appendix}{\par
	\setcounter{section}{0}%
	\setcounter{subsection}{0}%
	\setcounter{subsubsection}{0}%
	\gdef\thesection{\@Alph\c@section}%
	\gdef\thesubsection{\@Alph\c@section.\@arabic\c@subsection}%
	\gdef\theHsection{\@Alph\c@section.}%
	\gdef\theHsubsection{\@Alph\c@section.\@arabic\c@subsection}%
	\csname appendixmore\endcsname
}

\numberwithin{equation}{section}

\begin{document}
\title{\bf\Large
Hardy Spaces Associated with Non-Negative Self-Adjoint Operators
and Ball Quasi-Banach Function Spaces on Doubling Metric Measure Spaces and Their Applications
\footnotetext{\hspace{-0.35cm} 2020 {\it
Mathematics Subject Classification}.
Primary 42B30; Secondary 42B25,
35K08, 42B35,  35J30.
\endgraf {\it Key words and phrases.}
Hardy space, ball quasi-Banach function space, non-negative self-adjoint operator, atom, molecule,
Schr\"older group, spectral multiplier, Littlewood--Paley function.
\endgraf This project is partially supported by the National Key Research and Development Program of China (Grant No. 2020YFA0712900), the National Natural Science Foundation of China (Grant Nos.
11971058, 12071197, 12122102 and 12071431), the Fundamental Research Funds for the
Central Universities (Grant No. lzujbky-2021-ey18) and the Innovative Groups of
Basic Research in Gansu Province (Grant No. 22JR5RA391).}}
\author{Xiaosheng Lin, Dachun Yang\footnote{Corresponding
author, E-mail: \texttt{dcyang@bnu.edu.cn}/{\color{red} April 26, 2023}/Final version.},
\ Sibei Yang and Wen Yuan}
\date{}
\maketitle

\vspace{-0.7cm}
	
\begin{center}
\begin{minipage}{13cm}
{\small {\bf Abstract.}\quad
Let $(\mathcal{X},d,\mu)$ be a doubling metric measure space 
in the sense of R. R. Coifman and G. Weiss,
$L$ a non-negative self-adjoint operator on $L^2(\mathcal{X})$
satisfying the Davies--Gaffney estimate,
and $X(\mathcal{X})$ a ball quasi-Banach function space
on $\mathcal{X}$ satisfying some mild assumptions.
In this article, the authors introduce the Hardy type space
$H_{X,\,L}(\mathcal{X})$ by the Lusin area function
associated with $L$ and establish the atomic and the molecular
characterizations of $H_{X,\,L}(\mathcal{X}).$
As an application of these characterizations of $H_{X,\,L}(\mathcal{X})$,
the authors obtain the boundedness of
spectral multiplies on $H_{X,\,L}(\mathcal{X})$. Moreover, when $L$
satisfies the Gaussian upper bound estimate,
the authors further characterize $H_{X,\,L}(\mathcal{X})$ in 
terms of the Littlewood--Paley functions $g_L$ and
$g_{\lambda,\,L}^\ast$ and establish the boundedness estimate of Schr\"odinger groups on $H_{X,\,L}(\mathcal{X})$.
Specific spaces $X(\mathcal{X})$ to which these results can be 
applied include Lebesgue spaces, Orlicz spaces,
weighted Lebesgue spaces, and variable Lebesgue spaces. 
This shows that the results obtained in the article
have extensive generality.}
\end{minipage}
\end{center}

\vspace{0.2cm}



\section{Introduction}
As a major landmark in the development of harmonic analysis and partial differential equations,
the real-variable theory of the classical Hardy space on the Euclidean space $\mathbb{R}^n$ with $p\in(0,1]$
was creatively initiated by Stein and Weiss \cite{sw60} and then further seminally developed by Fefferman
and Stein  \cite{fs72}. Moreover, Calder\'on et al. \cite{c77,ccfjr78} used weighted Hardy
spaces to investigate the Cauchy integral on Lipschitz curves and Kenig \cite{k80} studied the (weighted)
Hardy spaces on Lipschitz domains. Observe that the classical Hardy space is built on the Lebesgue space.
It is also well known that, due to the need from various applications, there appear a lot
of more exquisite function spaces than the Lebesgue space, such as Orlicz spaces, Morrey spaces, weighted Lebesgue spaces,
mixed-norm Lebesgue spaces, and variable Lebesgue spaces.
Among with these function spaces,
the (quasi-)Banach function space is an important concept, which includes many function spaces such as
Lebesgue spaces, variable Lebesgue spaces, and Orlicz spaces as special cases (see, for instance, \cite{bs88}). However,
there  also exist many function spaces which are not necessarily (quasi-)Banach function spaces; for examples,
 weighted Lebesgue space, Morrey spaces, and Herz spaces (see, for instance, \cite{shyy17}).
Thus, the concept of (quasi-)Banach function spaces is restricted. In order to include all  these function spaces in
a unified framework, Sawano et al. \cite{shyy17} introduced the ball quasi-Banach function space $X$ on
$\mathbb{R}^n$ and studied the real-variable theory of the Hardy type space $H_{X}(\mathbb{R}^n)$ associated
with $X$. Compared with (quasi-)Banach function spaces, ball (quasi-)Banach function spaces contain more
function spaces and hence are more general. For more studies on the ball (quasi-)Banach function space $X$
and the Hardy type space $H_{X}(\mathbb{R}^n)$, we refer the reader to \cite{Syy21,syy21,wyy20,wyyz21,
yhyy22b,zwyy20}.

Although the theory of classical Hardy spaces has been very successful and
fruitful in the past decades,
there exist some important situations in which the classical real-variable theory
of Hardy spaces is not applicable.
For example, let $L:=-{\rm div}(A\nabla)$ be a second-order divergence elliptic operator with complex
bounded measurable coefficients on $\mathbb{R}^n$.
Then the Riesz transform $\nabla L^{-1/2}$ associated
with $L$ may not be bounded from the classical Hardy space
$H^1(\mathbb{R}^n)$ to $L^1(\mathbb{R}^n)$, but
is bounded from $H^1_{L}(\mathbb{R}^n)$ to $L^1(\mathbb{R}^n)$ (see, for instance, \cite{hmm11}), where
$H^1_{L}(\mathbb{R}^n)$ denotes the Hardy space associated with the operator $L$. Thus, it is necessary
to consider   Hardy spaces and also other function spaces
associated with different differential operators.
In particular, when the operator $L$ satisfies a Poisson upper bound estimate, Auscher et al. \cite{adm05}
initially introduced the Hardy space $H^1_{L}(\mathbb{R}^n)$ associated with $L$ and established its
molecular characterization. Furthermore, Duong and Yan \cite{dy051,dy05} showed that the dual space of
the Hardy space $H^1_L(\mathbb{R}^n)$ is just the BMO-type space $\mathrm{BMO}_{L^\ast}(\mathbb{R}^n)$,
where $L^\ast$ denotes the adjoint operator of $L$. Later on, when $p\in(0,1]$ is closed to $1$, these
results were generalized to $H^p_{L}(\mathbb{R}^n)$ by Yan \cite{y08}.

The real-variable theory of Hardy type spaces has also been developed on the more general spaces of homogeneous type.
In what follows, let
$$(\mathcal{X},d,\mu) \text{ or simply } \mathcal{X}$$
be a space of homogeneous type in the sense of Coifman and Weiss \cite{cw77}. Hofmann et al. \cite{hlmmy11} established several characterizations
of the Hardy space $H^1_{L}(\mathcal{X})$,
respectively, in terms of atoms, molecules, and Littlewood--Paley square
functions, when the operator $L$ satisfies the so-called Davies--Gaffney estimate which is weaker than
the Gaussian upper bound estimate. Later, in \cite{jy11,yy14}, the results of Hofmann et al. \cite{hlmmy11}
was further extended to the case of (Musielak--)Orlicz--Hardy spaces
associated with $L$ which contain Hardy spaces $H^p_L(\mathcal{X})$,
where $p\in(0,1]$, as special cases. Meanwhile, the boundedness of spectral multipliers on
$H^p_{L}(\mathcal{X})$ with $p\in(0,1]$ was also considered by Duong and Yan \cite{dy11}. Recently,
Chen et al. \cite{cdly20,cdly21arxiv} gave the boundedness estimates of Schr\"odinger groups generated by $L$ on
the Lebesgue space $L^p(\mathcal{X})$ with $p\in(1,\infty)$ and the Hardy space $H^1_{L}(\mathcal{X})$, which
are generalized to the case of $H^p_{L}(\mathcal{X})$ with $p\in(0,1]$ by Bui and Ly \cite{bl22} when $L$ further
satisfies the Gaussian upper bound estimate; see also \cite{bddm19,bdn20} for more boundedness results
on Schr\"odinger groups generated by $L$. The real-variable theory of Hardy spaces associated with operators
and various function spaces were developed rapidly and vigorously; see, for instance, \cite{yz16,yzz18} for variable Hardy spaces associated with operators, \cite{bckyy13,bl11,jy10,jy11,yy14}
for (Musielak--)Orlicz--Hardy spaces associated with operators, \cite{b14,bckyy131,ls13,sy10} for
weighted Hardy spaces associated with operators, and, particularly, Georgiadis et al. \cite{gkkp19,gkkp17,gk20,gk23,gn18,gn17}
for other function spaces associated with non-negative
self-adjoint operators and their applications.

Let $\mathcal{X}$ be a space of homogeneous type, $L$ a non-negative self-adjoint operator on $L^2(\mathcal{X})$,
and $X(\mathcal{X})$ a ball quasi-Banach function space on $\mathcal{X}$. For the purpose of having a deeper
understanding of Hardy spaces associated with the operator $L$ and putting various Hardy spaces into a general
framework, in this article, we introduce and study the Hardy type space $H_{X,\,L}(\mathcal{X})$ associated with
both $L$ and $X(\mathcal{X})$. Precisely, we introduce the Hardy type space $H_{X,\,L}(\mathcal{X})$ via the Lusin
area function associated with $L$ and establish the atomic and the molecular characterizations of
$H_{X,\,L}(\mathcal{X})$ with the help of the atomic decomposition of the tent space $T_{X}(\mathcal{X}^+)$
on $\mathcal{X}^+:=\mathcal{X}\times(0,\infty)$. As an application of these characterizations of $H_{X,\,L}(\mathcal{X})$,
we obtain the boundedness of the  H\"ormander type spectral multiplier on $H_{X,\,L}(\mathcal{X})$.
Moreover, assume further that $L$ satisfies the Gaussian upper bound estimate. Applying a pointwise estimate
for the Peetre type maximal function associated with $L$ (see Lemma \ref{point} below) and the atomic
characterization of $H_{X,\,L}(\mathcal{X})$, we further characterize $H_{X,\,L}(\mathcal{X})$
by means of the Littlewood--Paley $g$-function $g_L$ and the Littlewood--Paley $g_\lambda^\ast$-function
$g^\ast_{\lambda,\,L}$. Furthermore, we also show that the operator $(I+L)^{-s}e^{itL}$ is bounded
on $H_{X,\,L}(\mathcal{X})$.

It is worth mentioning that the main results given in this article can be applied to many specific
Hardy type spaces associated with operators, including
Hardy spaces $H^p_L(\mathcal{X})$, Orlicz--Hardy spaces $H^\Phi_L(\mathcal{X})$, weighted Hardy spaces
$H^p_{\omega,\,L}(\mathcal{X})$, and variable Hardy spaces $H^{p(\cdot)}_L(\mathcal{X})$. Moreover, to the best
of our knowledge, even in the special cases of Orlicz--Hardy spaces, weighted Hardy spaces, and variable
Hardy spaces, the boundedness of the operator $(I+L)^{-s}e^{itL}$ on $H_{X,\,L}(\mathcal{X})$ obtained in
this article is completely new.

Comparing with the classical Hardy space, one main difficulty to deal with Hardy type spaces built on
$X(\mathcal{X})$ is that the quasi-norm $\|\cdot\|_{X(\mathcal{X})}$ has no explicit
expression. The main novelty of this article is that we make full use of the
vector-valued maximal inequality in $X(\mathcal{X})$,
the boundedness the Hardy--Littlewood maximal
operator $\mathcal{M}$ on the associate space of the convexification of $X(\mathcal{X})$ (see Assumptions \ref{vector1} and \ref{vector2} below),
and some ideas from  the extrapolation theorem and its proof
to overcome the difficulties caused by the deficiency of the explicit
expression of the quasi-norm $\|\cdot\|_{X(\mathcal{X})}$. Indeed, the boundedness of the Hardy--Littlewood
maximal operator $\mathcal{M}$ on the associate space of the convexification of $X(\mathcal{X})$
allows us to reconstruct the atomic or the molecular Hardy space (quasi-)norm and the vector-valued
maximal inequality of $X(\mathcal{X})$ plays an essential role in the $\|\cdot\|_{X(\mathcal{X})}$ estimate of functions under consideration .
Moreover, to obtain the Littlewood--Paley characterization of $H_{X,\,L}(\mathcal{X})$,
we borrow some ideas from the extrapolation theorem and its proof and
translate the problem into the weighted case.

The remainder of this article is organized as follows.

In Section \ref{section2}, we recall some concepts and well-known results on spaces of homogeneous type
$(\mathcal{X},d,\mu)$, the ball quasi-Banach function $X(\mathcal{X})$, and the Hardy type space
$H_{X,\,L}(\mathcal{X})$.

In Section \ref{section3}, we establish the atomic and the molecular characterizations of
$H_{X,\,L}(\mathcal{X})$ (see Theorem \ref{thm-mc} below). By making full use of the geometrical properties
of $\mathcal{X}$, we first gave the atomic decomposition of the tent space $T_X(\mathcal{X}^+)$ in Subsection
\ref{sec3.1} (see Theorem \ref{thm-ad-tent} below). Then, applying the atomic decomposition of $T_X(\mathcal{X}^+)$
and the functional calculus associated with the operator $L$, we obtain the atomic and the molecular
characterizations of $H_{X,\,L}(\mathcal{X})$.

Section \ref{section4} is  divided into three subsections to give some applications of the atomic and the
molecular characterizations of $H_{X,\,L}(\mathcal{X})$. First, we establish the boundedness of the  H\"ormander
type spectral multiplier on $H_{X,\,L}(\mathcal{X})$ in Subsection \ref{sec4.3} (see Theorem \ref{thm-spec} below).
Assume further that the operator $L$ satisfies the Gaussian upper bound estimate. In Subsection \ref{sec4.1},
we obtain the Littlewood--Paley characterization of $H_{X,\,L}(\mathcal{X})$ (see Theorem \ref{thm-g} below)
via using the atomic characterization of $H_{X,\,L}(\mathcal{X})$ and a pointwise estimate for the Peetre type
maximal function associated with $L$ given in \cite{h17}. In Subsection \ref{sec4.2}, we give the boundedness
of the operator $(I+L)^{-s}e^{itL}$ on $H_{X,\,L}(\mathcal{X})$ (see Theorem \ref{sch} below).

In Section \ref{section5}, we apply the results obtained in Sections \ref{section3} and \ref{section4} to
some specific spaces including  Orlicz--Hardy spaces, weighted Hardy spaces, and
variable Hardy spaces associated with operators.

Finally, we make some convention on symbols. Let $\mathbb{N}:=\{1,2,3,\ldots\}$, $\mathbb{Z}_+:=\mathbb{N}\cup\{0\}$,
and $\mathbb{Z}$ be the set of all integers. For any $p\in[1,\infty]$, the \emph{symbol} $p'$ denotes its
\emph{conjugate index}, that is, $1/p'+1/p=1$. The \emph{symbol} $C$ always denotes a positive constant independent
of main parameters involved, but may vary from line to line. We also use the \emph{symbol} $C_{(\alpha,\beta,\ldots)}$ or
$c_{(\alpha,\beta,\ldots)}$ to denote a positive constant depending on the indicated parameters $\alpha,\beta,\ldots.$
If $f\le Cg$, we then write $f\lesssim g$ or $g\gtrsim f$; if $f\lesssim g\lesssim f$, we then write $f\sim g$.
Moreover, if $f\lesssim g$ and $g=h$, or $f\lesssim g$ and $g\leq h$, we then write $f\lesssim g= h$ or
$f\lesssim g\leq h$. For any subset $E$ of
$\mathcal{X}$, the \emph{symbol} $\mathbf{1}_E$ denotes its \emph{characteristic function}
and $E^\complement$ its \emph{complement}. For any $x\in\mathcal{X}$ 
and $E_1,E_2\subset \mathcal{X}$, let
$$
d(x,E_1):=\inf_{y\in E_1}d(x,y)\ \mathrm{and}\ d(E_1,E_2):=\inf_{y\in E_1,z\in E_2}d(y,z).
$$
For any $j\in\mathbb{N}$ and any ball $B\subset\mathcal{X}$, let $U_j(B):=(2^{j}B)\setminus(2^{j-1}B)$
and $U_0(B):=B.$
The \emph{symbol} $\mathscr{M}(\mathcal{X})$ denotes the set of all $\mu$-measurable
functions on $\mathcal{X}$. For any $\mu$-measurable set $E\subset \mathcal{X}$ and any $p\in(0,\infty)$,
the \emph{Lebesgue space} $L^p(E)$ is defined by setting
$$
L^p(E):=\left\{f\in\mathscr{M}(\mathcal{X}):\ \|f\|_{L^p(E)}:
=\left[\int_E|f(z)|^p\,d\mu(z)\right]^{\frac{1}{p}}<\infty\right\},
$$
and the \emph{Lebesgue space} $L^{\infty}(E)$ is defined by setting
$L^{\infty}(E):=\{f\in\mathscr{M}( \mathcal{X}):\ \|f\|_{L^\infty(E)}<\infty\}$,
where $\|f\|_{L^\infty(E)}$ denotes the \emph{essential supremum} of $f$ on $E$.
When we prove a theorem or the like, we always use the same symbols
in the wanted proved theorem or the like.
\section{Preliminary}\label{section2}

In this section, we recall some concepts and well-known results on any space of homogeneous type,
$(\mathcal{X},d,\mu)$, in the sense of  Coifman and  Weiss \cite{cw77}, as well as the ball quasi-Banach function
$X(\mathcal{X})$  and the Hardy type space $H_{X,\,L}(\mathcal{X})$.

\subsection{Spaces of Homogeneous Type}

In this subsection, we recall the concept of spaces of homogeneous type initially introduced by
 Coifman and  Weiss \cite{cw77}.
\begin{definition}
A \emph{quasi-metric space} $(\mathcal{X},d)$ is a non-empty set $\mathcal{X}$ equipped with a
\emph{quasi-metric} $d$, namely a non-negative function defined on $\mathcal{X}\times\mathcal{X}$
satisfying that, for any $x,y,z\in\mathcal{X}$,
\begin{enumerate}
\item[{\rm(i)}] $d(x,y)=0$ if and only if $x=y$;
\item[{\rm(ii)}] $d(x,y)=d(y,x)$;
\item[{\rm(iii)}] there exists a constant $A_0\in[1,\infty)$, independent of $x,y$, and $z$, such that
\begin{align*}
d(x,z)\leq A_0[d(x,y)+d(y,z)].
\end{align*}
\end{enumerate}
\end{definition}

\begin{definition}
Let $(\mathcal{X},d)$ be a quasi-metric space and $\mu$ a non-negative measure on $\mathcal{X}$.
The triple $(\mathcal{X},d,\mu)$ is called a \emph{space of homogeneous type} if $\mu$ satisfies the
following \emph{doubling condition}: there exists a constant $C_{(\mu)}\in[1,\infty)$ such that, for any
ball $B\subset\mathcal{X}$,
\begin{equation*}
\mu(2B)\leq C_{(\mu)}\mu(B).
\end{equation*}
If $A_0:=1$, then $(\mathcal{X},d,\mu)$ is called a \emph{metric measure space of homogeneous type} or,
simply, a \emph{doubling metric measure space}.
\end{definition}
In what follows, we \emph{always} assume that
$(\mathcal{X},d,\mu)$ is a  doubling metric measure space. Then
the above doubling condition further implies that, for any ball $B\subset\mathcal{X}$ and any
$\lambda\in[1,\infty)$,
\begin{equation}\label{eqoz}
\mu(\lambda B)\leq C_{(\mu)}\lambda^{n}\mu(B),
\end{equation}
where $n:= \log_2 C_{(\mu)}$ is called the \emph{upper dimension} of $\mathcal{X}$. Moreover,
for any $x,y\in\mathcal{X}$ and $r\in(0,\infty)$,
\begin{align}\label{eqoz3}
V(y,r)\leq C_{(\mu)}\left[1+\frac{d(x,y)}{r}\right]^{n}V(x,r).
\end{align}

Throughout this article, according to \cite[pp.\,587-588]{cw77}, we \emph{always} make the following assumptions
on $(\mathcal{X},d,\mu)$:
\begin{enumerate}
\item[\rm(i)] for any point $x\in\mathcal{X}$, the balls $\{B(x,r)\}_{r\in(0,\infty)}$ form a basis of
open neighborhoods of $x$;
\item[{\rm(ii)}] $\mu$ is \emph{Borel regular} which means that all open sets are $\mu$-measurable
and every set $A\subset\mathcal{X}$ is contained in a Borel set $E$ such that $\mu(A)=\mu(E)$;
\item[{\rm(iii)}] for any $x\in\mathcal{X}$ and $r\in(0,\infty)$, $\mu(B(x,r))\in(0,\infty)$;
\item[{\rm(iv)}] $\mathrm{diam}\, \mathcal{X}=\infty$ and $(\mathcal{X},d,\mu)$ is \emph{non-atomic}
which means $\mu(\{x\})=0$ for any $x\in\mathcal{X}$. Here and thereafter, $\mathrm{diam}\, \mathcal{X}
:=\sup\{d(x,y):\ x,y\in\mathcal{X}\}$.
\end{enumerate}
Note that $\mathrm{diam}\, \mathcal{X}=\infty$ implies that $\mu(\mathcal{X})=\infty$ (see \cite[p.\,284]{ah13}
or \cite[Lemma 5.1]{ny97}). From this, it follows that, under the above assumptions,
$\mu(\mathcal{X})=\infty$ if and only if $\mathrm{diam}\, \mathcal{X}=\infty$.

\subsection{Ball Quasi-Banach Function Spaces\label{sbqbs}}

In this subsection, we recall the concept of ball (quasi-)Banach function spaces on $\mathcal{X}$.
The ball (quasi-)Banach function space on $\mathbb{R}^n$ was originally introduced in
\cite[Definition 2.2 and (2.3)]{shyy17}.

\begin{definition}\label{debb}
A quasi-normed linear space $X(\mathcal{X})\subset\mathscr{M}(\mathcal{X})$, equipped with a \emph{quasi-norm}
$\|\cdot\|_{X(\mathcal{X})}$ which makes sense for all functions in $\mathscr{M}(\mathcal{X})$, is called
a \emph{ball quasi-Banach function space} (for short, BQBF \emph{space}) on $\mathcal{X}$ if it satisfies
the following conditions:
\begin{enumerate}
\item [(i)] for any $f\in\mathscr{M}(\mathcal{X})$, $\|f\|_{X( \mathcal{X})}=0$ if and only if
$f=0$ $\mu$-almost everywhere;
\item [(ii)] for any $f,g\in\mathscr{M}(\mathcal{X})$, $|g|\le|f|$ $\mu$-almost everywhere
implies that $\|g\|_{X( \mathcal{X})}\le \|f\|_{X( \mathcal{X})}$;
\item [(iii)] for any $\{f_k\}_{n\in\mathbb{N}}\subset\mathscr{M}(\mathcal{X})$ and
$f\in\mathscr{M}(\mathcal{X})$, $0\le f_k\uparrow f$ $\mu$-almost everywhere as $k\to\infty$
implies that $\|f_k\|_{X( \mathcal{X})}\uparrow \|f\|_{X( \mathcal{X})}$ as $k\to\infty$;
\item [(iv)] for any ball $B\subset\mathcal{X}$, $\mathbf{1}_B\in X(\mathcal{X})$.
\end{enumerate}
A normed linear space $X(\mathcal{X})\subset\mathscr{M}(\mathcal{X})$ is called a
\emph{ball Banach function space} (for short, BBF \emph{space}) on $\mathcal{X}$ if
it satisfies (i)-(iv) and
\begin{enumerate}
\item [(v)] for any ball $B$ of $\mathcal{X}$, there exists a positive constant $C$, depending
only on $B$, such that, for any $f\in X( \mathcal{X})$,
$$\int_{B}|f(x)|\,d\mu(x)\leq C\|f\|_{X( \mathcal{X})}.$$
\end{enumerate}
\end{definition}

\begin{remark}
\begin{itemize}
\item[\textup{(i)}] Let $X(\mathcal{X})$ be a BQBF space. By \cite[Theorem 2]{dfmn21}, we find that both
(ii) and (iii) of Definition \ref{debb} imply that $X(\mathcal{X})$ is complete.

\item[\textup{(ii)}] In Definition \ref{debb}, if any ball $B$ is replaced by any bounded $\mu$-measurable
set $E\subset \mathcal{X}$, then we obtain an equivalent definition of the ball (quasi-)Banach function space
on $\mathcal{X}$. In Definition \ref{debb}, if any ball $B$ is replaced by any $\mu$-measurable set
$E\subset\mathcal{X}$ with $\mu(E)<\infty$, then we obtain the definition of the (quasi-)Banach function
space on $\mathcal{X}$; see, for instance, \cite[p.\,3, Definition 1.3]{bs88} and \cite[Definition 2.4]{yhyy21a}.
It is easy to show that a (quasi-)Banach function space on $\mathcal{X}$ is a ball (quasi-)Banach function space
on $\mathcal{X}$.
\end{itemize}
\end{remark}
The associate space $X'(\mathcal{X})$ of a $\mathrm{BBF}$ space $X(\mathcal{X})$ is defined by setting
\begin{equation*}
X'(\mathcal{X}):=\left\{f\in\mathscr{M}(\mathcal{X}):\ \|f\|_{X'(\mathcal{X})}<\infty\right\},
\end{equation*}
where, for any $f\in \mathscr{M}(\mathcal{X})$,
$$\|f\|_{X'(\mathcal{X})}:=\sup\left\{\|fg\|_{L^1(\mathcal{X})}: \ g\in X(\mathcal{X}),\
\|g\|_{X(\mathcal{X})}=1\right\}$$
(see, for instance, \cite[p.\,8, Definition 2.1]{bs88}).
The following conclusions are just \cite[Lemmas 2.18-2.20]{yhyy21a}.
\begin{lemma}\label{2023221}
Let $X(\mathcal{X})$ be a $\mathrm{BBF}$ space and $X'(\mathcal{X})$ the associate space of $X(\mathcal{X})$. Then the following conclusions hold true:
\begin{itemize}
\item [$\mathrm{(i)}$]
 $X'(\mathcal{X})$ is a $\mathrm{BBF}$ space;
\item [$\mathrm{(ii)}$]
for any $f\in X(\mathcal{X})$ and $g\in X'(\mathcal{X})$,
\begin{align*}
\|fg\|_{L^1(\mathcal{X})}\leq \|f\|_{X(\mathcal{X})}\|g\|_{X'(\mathcal{X})};
\end{align*}
\item [$\mathrm{(iii)}$] a function $f\in X(\mathcal{X})$
if and only if $f\in X''(\mathcal{X});$ moreover, for any $f\in X(\mathcal{X})$ or $f\in X''(\mathcal{X})$,
\begin{align*}
\|f\|_{X(\mathcal{X})}=\|f\|_{X''(\mathcal{X})}.
\end{align*}
\end{itemize}
\end{lemma}

Let us recall the concept of the convexification
of a BQBF space as follows (see, for instance, \cite[p.\,53]{lt79} and \cite[Definition 2.6]{shyy17}).
\begin{definition}\label{decon}
Let $p\in(0,\infty)$ and $X(\mathcal{X})$ be a BQBF space. The \emph{p-convexification} $X^p(\mathcal{X})$
of $X( \mathcal{X})$ is defined by setting
\begin{equation*}
X^p( \mathcal{X}):=\left\{f\in\mathscr{M}( \mathcal{X}):\ \|f\|_{X^p(\mathcal{X})}:=
\left\||f|^p\right\|_{X(\mathcal{X})}^{\frac{1}{p}}<\infty\right\}.
\end{equation*}
\end{definition}

Recall that the \emph{Hardy--Littlewood maximal operator} $\mathcal{M}$ on $\mathcal{X}$ is defined by setting,
for any $f\in\mathscr{M}(\mathcal{X})$ and $x\in\mathcal{X}$,
\begin{align*}
\mathcal{M}(f)(x):=\sup_{B\ni x}\frac{1}{\mu(B)}\int_B|f(z)|\,d\mu(z),
\end{align*}
where the supremum is taken over all the balls $B$ containing $x$.

Moreover, we also need the following two key assumptions on the BQBF space.

\begin{assumption}\label{vector1}
Let $X(\mathcal{X})$ be a $\mathrm{BQBF}$ space. Assume that there exists a positive constant $p$ such that,
for any given $t\in (0,p)$ and $u\in(1,\infty)$, there exists a positive constant $C$ such that, for
any $\{f_j\}_{j\in\mathbb{N}}\subset \mathscr{M}(\mathcal{X})$,
\begin{equation*}
\left\|\left\{\sum_{j\in\mathbb{N}}\left[\mathcal{M}\left(f_j\right)\right]^u\right\}
^{\frac{1}{u}}\right\|_{X^{\frac{1}{t}}(\mathcal{X})}
\leq C\left\|\left(\sum_{j\in\mathbb{N}}\left|f_j\right|^u\right)
^{\frac{1}{u}}\right\|_{X^{\frac{1}{t}}(\mathcal{X})}.
\end{equation*}
\end{assumption}
\begin{assumption}\label{vector2}
Let $X(\mathcal{X})$ be a $\mathrm{BQBF}$ space. Assume that there exist constants $s_0\in(0,\infty)$
and $q_0\in(s_0,\infty)$ such that $X^{\frac{1}{s_0}}(\mathcal{X})$ is a $\mathrm{BBF}$ space and the
Hardy--Littlewood maximal operator $\mathcal{M}$ is bounded on the $\frac{1}{(q_0/s_0)'}$-convexification
of the associate space $(X^{1/s_0})'(\mathcal{X})$, where $\frac{1}{(q_0/s_0)'}+\frac{1}{q_0/s_0}=1$.
\end{assumption}

The following two propositions are just, respectively,  \cite[Proposition 4.8(i)]{syy21} and
\cite[Proposition 2.14]{Syy21}.

\begin{proposition}\label{prfs}
Let $X(\mathcal{X})$ be a $\mathrm{BQBF}$ space satisfying Assumption \ref{vector1} for some $p\in(0,\infty)$.
Then, for any given $t\in(0,\infty)$ and $s\in(\max\{1,t/p\},\infty)$,
there exists a positive constant $C$ such that, for any $\tau\in[1,\infty)$, any $\{\lambda_j\}_{j\in\mathbb{N}}\in
[0,\infty)$, and any sequence $\{B_j\}_{j\in\mathbb{N}}$ of balls,
\begin{align*}
\left\|\sum_{j\in\mathbb{N}}\lambda_j\mathbf{1}_{\tau B_j}
\right\|_{X^{\frac{1}{t}}(\mathcal{X})}\le C\tau^{sn }\left\|\sum_{j\in\mathbb{N}}
\lambda_j\mathbf{1}_{B_j}\right\|_{X^{\frac{1}{t}}(\mathcal{X})},
\end{align*}
where $n$ is the same as in \eqref{eqoz}.
\end{proposition}
\begin{proposition}\label{pras}
Let $X(\mathcal{X})$ be a $\mathrm{BQBF}$ space satisfying Assumption \ref{vector2} for some $s_0\in(0,\infty)$
and $q_0\in(s_0,\infty)$. Assume that $q\in[q_0,\infty)$. Then there exists a positive constant $C$
such that, for any $\{\lambda_j\}_{j\in\mathbb{N}}\subset (0,\infty)$ and
$\{a_j\}_{j\in\mathbb{N}}\subset L^q(\mathcal{X})$ satisfying both  $\|a_j\|_{L^q(\mathcal{X})}\le
\lambda_j[\mu(B_j)]^{\frac{1}{q}}$ and $\mathrm{supp}\,(a_j)\subset \theta B_j$ for any $j\in\mathbb{N}$,
\begin{equation*}
\left\|\sum_{j\in\mathbb{N}}\left|a_j\right|^{s_0}\right\|_{X^{\frac{1}{s_0}}(\mathcal{X})}
\le C\theta^{(1-\frac{s_0}{q})n}\left\|\sum_{j\in\mathbb{N}}\left|\lambda_j\right|^{s_0}
\mathbf{1}_{B_j}\right\|_{X^{\frac{1}{s_0}}(\mathcal{X})}.
\end{equation*}
\end{proposition}

\subsection{Hardy Spaces $H_{X,\,L}(\mathcal{X})$}

In this subsection, we give the assumption on the operator $L$ considered in this article
and the definition of the Hardy space $H_{X,\,L}(\mathcal{X})$.

Let $L$ be a non-negative self-adjoint operator on $L^2(\mathcal{X})$.
In this article, we \emph{always} make the following assumption on the operator $L$.

\begin{assumption}\label{ls2}
The semigroup $\{e^{-tL}\}_{t\in(0,\infty)}$, generated by $-L$, satisfies the \emph{Davies--Gaffney estimate},
that is, there exist positive constants $C_{(L)}$ and $c_{(L)}$ such that, for any closed sets
$E,F\subset\mathcal{X}$ and any $f\in L^2(\mathcal{X})$ with $\mathrm{supp}\,(f)\subset E$,
\begin{align}\label{dg}
\left\|e^{-tL}(f)\right\|_{L^2(F)}\leq C_{(L)}\exp\left\{-c_{(L)}\frac{[\mathrm{dist}\,(E,F)]^2}{t}\right\}
\|f\|_{L^2(E)}.
\end{align}
\end{assumption}

Examples of operators satisfying Assumption \ref{ls2} include second-order elliptic self-adjoint operators of
divergence form on $\mathbb{R}^n$, Schr\"odinger operators with real potential and magnetic field on
$\mathbb{R}^n$, and Laplace--Beltrami operators on all complete Riemannian manifolds (see, for instance,
\cite{da92,hlmmy11}).

\begin{remark}\label{1515}
Let $L$ be a non-negative self-adjoint operator on $L^2(\mathcal{X})$ whose related semigroup  satisfies the
Davies--Gaffney estimate \eqref{dg}. By \cite[Proposition 3.1]{hlmmy11}, we find that, for any given
$k\in\mathbb{Z}_+$, the family $\{(tL)^ke^{-tL}\}_{t\in(0,\infty)}$ of operators also satisfies the
Davies--Gaffney estimate.
\end{remark}

For any given  $\alpha\in(0,\infty)$, the \emph{square function}
$S_{L}^{(\alpha)}$, associated with $L$, is defined by setting, for any $f\in L^2(\mathcal{X})$ and
$x\in\mathcal{X},$
\begin{align*}
S_{L}^{(\alpha)}(f)(x):=\left[\int_0^\infty\int_{B(x,\alpha t)}\left|t^2L e^{-t^2L}
(f)(y)\right|^2\,\frac{d\mu(y)\,dt}{V(x,t)t}\right]^{\frac{1}{2}}.
\end{align*}
In particular, when $\alpha:=1$,  let $S_L(f):=S_L^{(1)}(f)$.
Let $R(L)$ be the range of $L$ and $\overline{R(L)}$ its \emph{closure} in $L^2(\mathcal{X})$. Then
$L^2(\mathcal{X})$ is the orthogonal sum of $\overline{R(L)}$ and the null space $N(L)$, that is,
$L^2(\mathcal{X}):=\overline{R(L)}\bigoplus N(L)$.
\begin{definition}\label{defi-area}
Let $X(\mathcal{X})$ be a $\mathrm{BQBF}$ space and $L$ a non-negative self-adjoint operator on $L^2(\mathcal{X})$
satisfying the Davies--Gaffney estimate \eqref{dg}. Then the \emph{Hardy space} $H_{X,\,L}(\mathcal{X})$,
associated with $X$ and $L$, is defined as the completion of the set
\begin{align*}
\widetilde{H}_{X,\,L}(\mathcal{X}):=\left\{f\in\overline{R(L)}:\ \|S_{L}(f)\|_{X(\mathcal{X})}<\infty\right\}
\end{align*}
with respect to the \emph{quasi-norm}
$\|f\|_{H_{X,\,L}(\mathcal{X})}:=\|S_{L}(f)\|_{X(\mathcal{X})}.$
\end{definition}
\begin{remark}
From \cite[Section 2.6]{hlmmy11}, it follows that, for any $f\in L^2(\mathcal{X})$ with $f\not\equiv0$,
$\|S_{L}(f)\|_{X(\mathcal{X})}=0$  if and only if $f\in N(L)$. Thus, in Definition \ref{defi-area},
it is necessary to use $\overline{R(L)}$ rather than $L^2(\mathcal{X})$ to guarantee $\|\cdot\|_{X(\mathcal{X})}$
to be a quasi-norm.
\end{remark}

\section{Atomic and Molecular Characterizations of $H_{X,\,L}(\mathcal{X})$}\label{section3}

In this section, we establish the atomic and the molecular characterizations of $H_{X,\,L}(\mathcal{X})$.
We first introduce the molecular Hardy space $H^{M,\,\epsilon}_{X,\,L,\,\rm{mol}}(\mathcal{X})$ and
the atomic Hardy space $H^{M}_{X,\,L,\,\rm{at}}(\mathcal{X})$ as follows.

\begin{definition}\label{defi-mol}
Let $M\in\mathbb{N}$, $\epsilon\in(0,\infty)$, and $X(\mathcal{X})$ be a $\mathrm{BQBF}$ space satisfying Assumption
\ref{vector2} for some $s_0\in(0,\infty)$ and $q_0\in(s_0,\infty)$. Assume that $L$ is a non-negative
self-adjoint operator on $L^2(\mathcal{X})$ satisfying the Davies--Gaffney estimate \eqref{dg}.
Denote by $\mathcal{D}(L^M)$ the domain of $L^M$.
\begin{itemize}
\item[\rm(i)] A function $\alpha\in L^2(\mathcal{X})$ is called an \emph{$(X,M,\epsilon)$-molecule}
associated with the ball $B:=B(x_B,r_B)$ of $\mathcal{X}$ for some $x_B\in\mathcal{X}$ and $r_B\in(0,\infty)$
if there exists a function $b\in\mathcal{D}(L^M)$ such that $\alpha=L^M(b)$ and
\begin{align}\label{chicun}
\left\|\left(r_B^2L\right)^k(b)\right\|_{L^2(U_j(B))}
\leq2^{-j\epsilon}r_B^{2M}\left[\mu(2^jB)\right]^{\frac{1}{2}}\|\mathbf{1}_{B}\|_{X(\mathcal{X})}^{-1}
\end{align}
for any $k\in\{0,\ldots,M\}$ and $j\in\mathbb{Z}_+$.

\item[\rm(ii)] For any $f\in L^2(\mathcal{X})$, $f=\sum_{j=1}^\infty \lambda_j\alpha_j$ is called a
\emph{molecular $(X,M,\epsilon)$-representation} of $f$ if, for any $j\in\mathbb{N}$, $\alpha_j$ is
an $(X,M,\epsilon)$-molecule associated with the ball $B_j\subset\mathcal{X}$, the summation
converges in $L^2(\mathcal{X})$, and $\{\lambda_j\}_{j\in\mathbb{N}}\subset[0,\infty)$ satisfies
\begin{align*}
\Lambda\left(\left\{\lambda_j\alpha_j\right\}_{j\in\mathbb{N}}\right)
:=\left\|\left\{\sum_{j=1}^\infty
\left[\frac{\lambda_j}{\|\mathbf{1}_{B_j}\|_{X(\mathcal{X})}}\right]^{s_0}
\mathbf{1}_{B_j}\right\}^{\frac{1}{s_0}}\right\|_{X(\mathcal{X})}<\infty.
\end{align*}
Let
\begin{align*}
\widetilde{H}^{M,\,\epsilon}_{X,\,L,\,\rm{mol}}(\mathcal{X}):=\left\{f\in L^2(\mathcal{X}):\ f \
\text{has a molecular}\ (X,M,\epsilon)\text{-representation}\right\}
\end{align*}
equipped with the \emph{quasi-norm} $\|\cdot\|_{H^{M,\,\epsilon}_{X,\,L,\,\rm{mol}}(\mathcal{X})}$
given by setting, for any $f\in\widetilde{H}^{M,\,\epsilon}_{X,\,L,\,\rm{mol}}(\mathcal{X})$,
\begin{align*}
\|f\|_{H^{M,\,\epsilon}_{X,\,L,\,\rm{mol}}(\mathcal{X})}
&:=\inf\Bigg\{\Lambda\left(\left\{\lambda_j\alpha_j\right\}_{j\in\mathbb{N}}\right):\
f=\sum_{j=1}^\infty \lambda_j\alpha_j \\
&\quad\quad\quad\quad \ \text{is a molecular}\ (X,M,\epsilon)\text{-representation}\Bigg\},
\end{align*}
where the infimum is taken over all the molecular $(X,M,\epsilon)$-representations of $f$ as above.

The \emph{molecular Hardy  space} $H^{M,\,\epsilon}_{X,\,L,\,\rm{mol}}(\mathcal{X})$ is
defined as the completion of $\widetilde{H}^{M,\,\epsilon}_{X,\,L,\,\rm{mol}}(\mathcal{X})$ with
respect to the quasi-norm $\|\cdot\|_{H^{M,\,\epsilon}_{X,\,L,\,\rm{mol}}(\mathcal{X})}$.
\end{itemize}
\end{definition}

\begin{definition}
Let $M\in\mathbb{N}$ and $X(\mathcal{X})$ be a $\mathrm{BQBF}$ space satisfying Assumption \ref{vector2} for some
$s_0\in(0,\infty)$ and $q_0\in(s_0,\infty)$. Assume that $L$ is a non-negative self-adjoint operator on
$L^2(\mathcal{X})$ satisfying the Davies--Gaffney estimate \eqref{dg}. A function $\alpha\in L^2(\mathcal{X})$
is called an \emph{$(X,M)$-atom} associated with the ball $B:=B(x_B,r_B)\subset \mathcal{X}$ with some
$x_B\in\mathcal{X}$ and $r_B\in(0,\infty)$ if there exists a function $b\in\mathcal{D}(L^M)$ such that
$\alpha=L^M(b)$, $\mathrm{supp}\,(L^k(b))\subset B$, and
\begin{align}\label{chichun-at}
\left\|\left(r_B^{2}L\right)^k(b)\right\|_{L^2(\mathcal{X})}
\leq r_B^{2M}[\mu(B)]^{\frac{1}{2}}\|\mathbf{1}_{B}\|_{X(\mathcal{X})}^{-1}
\end{align}
for any $k\in\{0,\ldots,M\}.$ In a similar way, the set $\widetilde{H}^{M}_{X,\,L,\,\rm{at}}(\mathcal{X})$
and the atomic Hardy space $H^{M}_{X,\,L,\,\rm{at}}(\mathcal{X})$ are defined
in the same way, respectively, as $\widetilde{H}^{M,\,\epsilon}_{X,\,L,\,\rm{mol}}(\mathcal{X})$ and
$H^{M,\,\epsilon}_{X,\,L,\,\rm{mol}}(\mathcal{X})$ with $(X,M,\epsilon)$-molecules replaced by $(X,M)$-atoms.
\end{definition}

\begin{remark}\label{buceng}
Let $M\in\mathbb{N}$ and $\epsilon\in(0,\infty)$. Obviously, any  $(X,M)$-atom is also
an $(X,M,\epsilon)$-molecule. This implies that $\widetilde{H}^{M}_{X,\,L,\,\rm{at}}
(\mathcal{X})\subset \widetilde{H}^{M,\,\epsilon}_{X,\,L,\,\rm{mol}}(\mathcal{X})$ and hence
$H^{M}_{X,\,L,\,\rm{at}}(\mathcal{X})\subset H^{M,\,\epsilon}_{X,\,L,\,\rm{mol}}(\mathcal{X})$.
\end{remark}

The following conclusion is the main result of this section.

\begin{theorem}\label{thm-mc}
Let $X(\mathcal{X})$ be a $\mathrm{BQBF}$ space satisfying both Assumptions \ref{vector1} and \ref{vector2}
for some $p\in(0,\infty)$, $s_0\in(0,\min\{p,1\}]$, and $q_0\in(s_0,2]$. Assume that $L$ is a non-negative
self-adjoint operator on $L^2(\mathcal{X})$ satisfying the Davies--Gaffney estimate \eqref{dg}.
Let $M\in(\frac{n}{2}[\frac{1}{s_0}-\frac{1}{2}],\infty)\cap\mathbb{N}$ and $\epsilon\in(\frac{n}{s_0},\infty)$.
Then the spaces $H_{X,\,L}(\mathcal{X})$, $H^{M,\,\epsilon}_{X,\,L,\,\rm{mol}}(\mathcal{X})$,
and $H^{M}_{X,\,L,\,\rm{at}}(\mathcal{X})$ coincide with equivalent quasi-norms.
\end{theorem}
\begin{remark}
Since we only assume that the operator $L$ satisfies the Davies--Gaffney estimate \eqref{dg},
the condition $q_0\leq 2$ is nature, which further deduce the condition $r\in(0,2)$
in Theorem \ref{ls10} below.
\end{remark}
The proof of  Theorem \ref{thm-mc} is given in Subsection \ref{sec3.2}. In Subsection \ref{sec3.1}, we
first establish the atomic decomposition of the tent space $T_X(\mathcal{X}^+)$ on $\mathcal{X}^+:=
\mathcal{X}\times(0,\infty)$, which is an essential tool for the proof of Theorem \ref{thm-mc}.

Let $r\in(0,\infty)$. Then $L^r(\mathcal{X})$ is a $\mathrm{BQBF}$ space. In the case when $X(\mathcal{X}):=L^{r}(\mathcal{X})$,
we denote $H_{X,\,L}(\mathcal{X})$, $H_{X,\,L,\,\rm{mol}}^{M,\,\epsilon}(\mathcal{X})$, and $H_{X,\,L,\,\rm{at}}^{M}(\mathcal{X})$, respectively, by
$H_{L}^{r}(\mathcal{X})$, $H_{L,\,\rm{mol}}^{r,\,M,\,\epsilon}(\mathcal{X})$, and $H_{L,\,\rm{at}}^{r,\,M}(\mathcal{X})$. Applying Theorem \ref{thm-mc}
with $X(\mathcal{X}):=L^{r}(\mathcal{X})$, we obtain the following conclusion.

\begin{theorem}\label{ls10}
Let $L$ be a non-negative self-adjoint operator on $L^2(\mathcal{X})$ satisfying
the Davies--Gaffney estimate \eqref{dg}. Assume that $r\in(0,2)$, $M\in(\frac{n}{2}
[\frac{1}{\min\{1,r\}}-\frac{1}{2}],\infty)\cap\mathbb{N}$, and $\epsilon\in(\frac{n}{\min\{1,r\}},\infty)$.
Then the spaces $H_{L}^{r}(\mathcal{X})$, $H_{L,\,\rm{mol}}^{r,\,M,\,\epsilon}(\mathcal{X})$, and $H_{L,\,\rm{at}}^{r,\,M}
(\mathcal{X})$ coincide with equivalent quasi-norms.
\end{theorem}

\begin{proof}
Let $p:=r,$ $q_0:=2,$ and $s_0\in (0,\min\{1,r\})$ satisfying both $M>\frac{n}{2}(\frac{1}{s_0}-\frac{1}{2})$
and $\epsilon>\frac{n}{s_0}.$ By \cite[Theorem 1.2]{gly09}, we find that $X(\mathcal{X}):=L^r(\mathcal{X})$
satisfies Assumption \ref{vector1} for the aforementioned $p.$ Using \cite[p.\,10, Theorem 2.5]{bs88} and
the boundedness of $\mathcal{M}$ on $L^q(\mathcal{X})$ for any given $q\in(1,\infty)$ (see, for instance,
\cite[Theorem 2.2]{h01}), we conclude that $X(\mathcal{X}):=L^r(\mathcal{X})$ satisfies Assumption
\ref{vector2} for the aforementioned $s_0$ and $q_0.$ Thus, all the assumption of Theorem \ref{thm-mc}
are satisfied with $X(\mathcal{X}):=L^r(\mathcal{X})$, which further implies the desired conclusions of
the present theorem. This finishes the proof of Theorem \ref{ls10}.
\end{proof}

\begin{remark}
We point out that, when $r\in(0,1]$, the conclusion of Theorem \ref{ls10} is well known (see, for instance,
\cite[Theorem 2.5]{hlmmy11} and \cite[Theorem 5.5]{jy11}).
\end{remark}

\subsection{Atomic Decomposition of $T_X(\mathcal{X}^+)$}\label{sec3.1}
Let us first recall the concept of the tent space. In what follows, we \emph{always} let $\mathcal{X}^+:
=\mathcal{X}\times(0,\infty).$ For any $\alpha\in(0,\infty)$ and $x\in\mathcal{X}$, let
\begin{align*}
\Gamma_\alpha(x):=\left\{(y,t)\in\mathcal{X}^+:\ d(x,y)<\alpha t\right\}.
\end{align*}
Moreover, for any closed set $E\subset\mathcal{X}$, let
\begin{align*}
R_\alpha(E):=\bigcup_{x\in E}\Gamma_\alpha(x).
\end{align*}
For any open set $O\subset\mathcal{X}$ and any $\alpha\in(0,\infty)$, define the \emph{tent $T_\alpha(O)$}
over $O$ with aperture $\alpha$ by setting
\begin{align*}
T_\alpha(O):=\left\{(x,t) \in \mathcal{X}^+:\ d(x,O^\complement)\geq \alpha t\right\}.
\end{align*}
In what follows, when $\alpha:=1$, we remove the superscript $\alpha$ for simplicity.
It is easy to show that
\begin{align*}
T_\alpha(O)=\left[R_\alpha\left(O^\complement\right)\right]^\complement.
\end{align*}

Coifman et al. \cite{cms85} first introduced the tent space $T^p(\mathbb{R}^{n+1}_+)$ for any given
$p\in(0,\infty)$, where $\mathbb{R}^{n+1}_+:=\mathbb{R}^{n}\times(0,\infty)$.  Moreover, Russ \cite{r07}
studied the tent space $T^p(\mathcal{X}^+)$ on $\mathcal{X}^+$. Recall that a measurable function $f$ is
said to belong to the \emph{tent space} $T^p(\mathcal{X}^+)$ with $p\in(0,\infty)$ if $\|f\|_{T^p(\mathcal{X}^+)}:=\|\mathcal{A}(f)\|_{L^p(\mathcal{X})}<\infty$, where, for any $x\in\mathcal{X},$
\begin{align*}
\mathcal{A}(f)(x):=\left[\int_0^\infty\int_{B(x,t)}|f(y,t)|^2
\,\frac{d\mu(y)\,dt}{V(x,t)t}\right]^{\frac{1}{2}}.
\end{align*}
For any given $\mathrm{BQBF}$ space $X(\mathcal{X})$, the \emph{$X$-tent space} $T_X(\mathcal{X}^+)$ is defined to be the
set of all the  measurable functions $f:\ \mathcal{X}^+ \to {\mathbb C}$ with the finite \emph{quasi-norm}
$$\|f\|_{T_X(\mathcal{X}^+)}:=\left\|{\mathcal A}(f)\right\|_{X(\mathcal{X})}.$$

\begin{definition}
Let $X(\mathcal{X})$ be a $\mathrm{BQBF}$ space and $p\in(1,\infty)$. A measurable function $a:\ \mathcal{X}^+ \to
{\mathbb C}$ is said to be a \emph{$(T_X,p)$-atom} if there exists a ball $B\subset\mathcal{X}$ such that
\begin{enumerate}
\item[(i)] ${\rm supp}\,(a):=\{(x,t)\in\mathcal{X}^+:\ a(x,t)\neq0\} \subset T(B)$,
\item[(ii)] $\|a\|_{T^p(\mathcal{X}^+)}\le[\mu(B)]^{\frac{1}{p}}\|\mathbf{1}_B\|_{X(\mathcal{X})}^{-1}$.
\end{enumerate}
Furthermore, if $a$ is a $(T_X,p)$-atom for any $p\in(1,\infty)$, then $a$ is called a \emph{$(T_X,\infty)$-atom}.
\end{definition}

Then we have the following atomic decomposition of the tent space $T_X(\mathcal{X}^+)$.

\begin{theorem}\label{thm-ad-tent}
Assume that $X(\mathcal{X})$ is a  $\mathrm{BQBF}$ space satisfying Assumption \ref{vector1} for some
$p\in(0,\infty)$. Let $f\in T_X(\mathcal{X}^+)$ and $s_0\in(0,p]$. Then there exists a sequence
$\{\lambda_j\}_{j=1}^\infty\subset[0,\infty)$ and a sequence $\{a_j\}_{j=1}^\infty$ of
$(T_X,\infty)$-atoms associated, respectively,  with the balls $\{B_j\}_{j=1}^\infty$ such that,
for almost every $(x,t)\in\mathcal{X}^+$,
\begin{equation}\label{ad-tent-ae}
f(x,t)=\sum_{j\in\mathbb{N}} \lambda_j a_j(x,t)
\end{equation}
and
\begin{align*}
\left\|\left\{\sum_{j=1}^\infty
\left[\frac{\lambda_j}{\|\mathbf{1}_{B_j}\|_{X(\mathcal{X})}}\right]^{s_0}\mathbf{1}_{B_j}
\right\}^{\frac{1}{s_0}}\right\|_X\lesssim\|f\|_{T_X(\mathcal{X}^+)},
\end{align*}
where the implicit positive constant is independent of $f$. Moreover,
if $f\in T^2(\mathcal{X}^+)\cap T_X(\mathcal{X}^+),$ then \eqref{ad-tent-ae} holds true in $T^2(\mathcal{X}^+).$
\end{theorem}

\begin{remark}
Let $X(\mathcal{X}):=L^r(\mathcal{X})$ and $p=s_0:=r\in(0,1]$. In this case,
Theorem \ref{thm-ad-tent} coincides with \cite[Theorem 1.1]{r07}. However, in the case that
$r\in(1,\infty)$, Theorem \ref{thm-ad-tent} is new.
\end{remark}

To prove Theorem \ref{thm-ad-tent}, we need several lemmas. The following lemma is just \cite[Lemma 2.1]{r07}.

\begin{lemma}\label{tent-1}
Let $\eta\in(0,1)$. Then there exists a  $\gamma\in(0,1)$ and a $C_{(\gamma,\eta)}\in(0,\infty)$ such
that, for any closed subset $F\subset\mathcal{X}$ and any non-negative measurable function $H$ on
$\mathcal{X}^+$,
\begin{align*}
\iint_{R_{1-\eta}(F^\ast_\gamma)}H(y,t)V(y,t)\,d\mu(y)\,dt\leq
C_{(\gamma,\eta)}\int_F\iint_{\Gamma(x)}H(y,t)\,d\mu(y)\,dt\,d\mu(x),
\end{align*}
where
\begin{align*}
F^\ast_\gamma:=\left\{x\in\mathcal{X}:\ \mathcal{M}\left(\mathbf{1}_{F^\complement}\right)(x)
\leq 1-\gamma\right\}.
\end{align*}
\end{lemma}

\begin{remark}
We should point out that Lemma \ref{tent-1} was obtained in \cite[Lemma 2.7]{r07} under the additional
assumption $\mu(F^\complement)<\infty$. However, by checking the proof of \cite[Lemma 2.7]{r07} very
carefully, we find that the additional assumption $\mu(F^\complement)<\infty$
in \cite[Lemma 2.7]{r07}
is superfluous.
\end{remark}

The following lemma plays a key role in the proof of Theorem \ref{thm-ad-tent} (see, for instance,
\cite[Lemmas 5.6 and 5.6]{syy21}).

\begin{lemma}\label{tent-ge}
Let $\Omega$ be a proper open subset of $\mathcal{X}$. Then there exists a sequence
$\{x_j\}_{j\in\mathbb{N}}\subset \Omega$ such that
\begin{enumerate}
\item  [$\mathrm{(i)}$]
$$\Omega=\bigcup_{j\in\mathbb{N}}B(x_j,r_j),$$
where $r_j:=2^{-5}\mathrm{dist}\,(x_j,\Omega^\complement)$ for any $j\in\mathbb{N};$
\item [$\mathrm{(ii)}$] $\{B(x_j,5^{-1}r_j)\}_{j\in\mathbb{N}}$ is pairwise disjoint;
\item  [$\mathrm{(iii)}$] for any $j\in\mathbb{N}$, $\sharp\{i\in\mathbb{N}:\ B(x_i,2^4r_i)\cap
B(x_j,2^4r_j)\not=\emptyset\}\leq M$, where $M$ is a positive integer independent of $j$ and $\Omega$ and
$\sharp E$ denotes the cardinality of the set $E.$
\end{enumerate}
Moreover, there exists a family $\{\phi_j\}_{j\in\mathbb{N}}$ of non-negative functions on $\mathcal{X}$
such that
\begin{enumerate}
\item [$\mathrm{(iv)}$] for any $j\in\mathbb{N}$, $\mathrm{supp}\,(\phi_j)\subset B(x_j,2r_j)$;
\item [$\mathrm{(v)}$] for any $j\in\mathbb{N}$ and $x\in B(x_j,r_j),$ $\phi_j(x)\geq M^{-1}$;
\item [$\mathrm{(vi)}$] $\sum_{j\in\mathbb{N}}\phi_j=\mathbf{1}_\Omega.$
\end{enumerate}
\end{lemma}

Now, we show Theorem \ref{thm-ad-tent} by using Lemmas \ref{tent-1} and \ref{tent-ge}.

\begin{proof}[Proof of Theorem \ref{thm-ad-tent}]
Let $\eta\in(0,1)$ be fixed and $\gamma\in(0,1)$ be the same as in Lemma \ref{tent-1}.
Let $f\in T_X(\mathcal{X}^+)$. For any $k\in\mathbb{Z}$, let
\begin{align*}
O_k:=\left\{x\in\mathcal{X}:\ \mathcal{A}(f)(x)>2^k\right\}\ \
\text{and}\ \ O_{k,\gamma}^\ast:=\left\{x\in\mathcal{X}:\
\mathcal{M}\left(\mathbf{1}_{O_k}\right)(x)>1-\gamma\right\}.
\end{align*}
In the remainder of this proof, for any $k\in\mathbb{Z}$, we denote $O^\ast_{k,\gamma}$ simply
by $O^\ast_{k}$. If $O^\ast_k=\mathcal{X}$ for some $k\in\mathbb{Z}$, then, by Assumption \ref{vector1}
and $f\in T_X(\mathcal{X}^+)$, we find that
\begin{align*}
\|1-\gamma\|_{X^{\frac{1}{t}}(\mathcal{X})}
\leq\|\mathcal{M}(\mathbf{1}_{O_k})\|_{X^{\frac{1}{t}}(\mathcal{X})}
\lesssim\|\mathbf{1}_{O_k}\|_{X(\mathcal{X})}^t
\lesssim\left\|\frac{\mathcal{A}(f)}{2^k}\right\|_{X(\mathcal{X})}^t<\infty,
\end{align*}
where $t\in(0,p)$ is a fixed constant. This contradicts  $\|1\|_{X(\mathcal{X})}=\infty$
(see, for instance, \cite[Proposition 4.8(ii)]{syy21}). Thus, for any $k\in\mathbb{Z}$,
$O^\ast_k$ is a proper subset of $\mathcal{X}$. Using this and applying Lemma \ref{tent-ge} with $\Omega$
replaced by $O^\ast_k$, we obtain a family $\{B_{k,j}\}_{j\in\mathbb{N}}:=\{B(x_{k,j},r_{k,j})\}_{j\in\mathbb{N}}$
of balls and a family $\{\phi_{k,j}\}_{j\in\mathbb{N}}$ of non-negative functions having the properties stated
in Lemma \ref{tent-ge}. By Lemma \ref{tent-ge}(vi), we conclude that, for any $k\in\mathbb{Z}$ and
$(x,t)\in\mathcal{X}^+,$
\begin{align}\label{qingniao}
\mathbf{1}_{T_{1-\eta}(O^\ast_k)\setminus T_{1-\eta}(O^\ast_{k+1})}(x,t)
&=\mathbf{1}_{O^\ast_k}(x)\mathbf{1}_{T_{1-\eta}(O^\ast_k)\setminus T_{1-\eta}(O^\ast_{k+1})}(x,t)\\ \notag
&=\sum_{j\in\mathbb{N}}\phi_{k,j}(x)\mathbf{1}_{T_{1-\eta}(O^\ast_k)\setminus T_{1-\eta}(O^\ast_{k+1})}(x,t).
\end{align}
On the other hand, from \cite[p.\,513]{yy14}, we deduce that
\begin{align*}
\mathrm{supp}\,(f)\subset \bigcup_{k\in\mathbb{Z}}
T_{1-\eta}(O^\ast_{k})\cup E,
\end{align*}
where the set $E\subset \mathcal{X}^+$ satisfies that $\int_{E}\,d\mu(y)\,dt=0.$ This implies that, for almost every $(x,t)
\in\mathcal{X}^+,$
\begin{align*}
f(x,t)=f(x,t)\mathbf{1}_{\cup_{k\in\mathbb{Z}}T_{1-\eta}(O^\ast_k)}(x,t)
=\sum_{k\in\mathbb{Z}}f(x,t)\mathbf{1}_{T_{1-\eta}(O^\ast_k)\setminus T_{1-\eta}(O^\ast_{k+1})}(x,t),
\end{align*}
which, combined with \eqref{qingniao}, further implies that
\begin{align*}
f(x,t)=\sum_{k\in\mathbb{Z}}\sum_{j\in\mathbb{N}}f(x,t)\phi_{k,j}(x)
\mathbf{1}_{T_{1-\eta}(O^\ast_k)\setminus T_{1-\eta}(O^\ast_{k+1})}(x,t).
\end{align*}
For any $k\in\mathbb{Z}$, $j\in\mathbb{N}$, and $(x,t)\in\mathcal{X}^+$, let
\begin{align*}
a_{k,j}(x,t):=2^{-k}\left\|\mathbf{1}_{\widetilde{B}_{k,j}}\right\|_{X(\mathcal{X})}^{-1}f(x,t)
\phi_{k,j}(x)\mathbf{1}_{T_{1-\eta}(O^\ast_k)\setminus T_{1-\eta}(O^\ast_{k+1})}(x,t)
\end{align*}
and
\begin{align*}
\lambda_{k,j}:=2^{k}\left\|\mathbf{1}_{\widetilde{B}_{k,j}}\right\|_{X(\mathcal{X})},
\end{align*}
where $\widetilde{B}_{k,j}:=[2+34(1-\eta)^{-1}]B_{k,j}$.

Next, we prove that, for any given $k\in\mathbb{Z}$ and $j\in\mathbb{N}$, $a_{k,j}$ is a
$(T_X,\infty)$-atom, associated with the ball $\widetilde{B}_{k,j}$, up to a harmless
positive constant multiple. We first show that
\begin{align}\label{zhiji3}
\mathrm{supp}\,(a_{k,j})\subset T\left(\widetilde{B}_{k,j}\right).
\end{align}
By the definition of $a_{k,j}$ and Lemma \ref{tent-ge}(iv), we find that, for any $(y,t)\in
\mathrm{supp}\,(a_{k,j})$,
\begin{align}\label{xixixi}
(y,t)\in T_{1-\eta}(O^\ast_k)\ \ \text{and}\ \ y\in2B_{k,j}.
\end{align}
From this and the definition of $\widetilde{B}_{k,j}$, it follows that
\begin{align}\label{2015}
d\left(y,(\widetilde{B}_{k,j})^\complement\right)\geq
d\left(x_{k,j},(\widetilde{B}_{k,j})^\complement\right)-d(y,x_{k,j})
\geq34(1-\eta)^{-1}r_{k,j}.
\end{align}
By the definition of $r_{k,j}$, we conclude that $d(x_{k,j},(O^\ast_k)^\complement)=32 r_{k,j}$.
Thus, for any $\epsilon\in(0,\infty)$, there exists a $u_{k,j}\in (O^\ast_k)^\complement$ such that
$d(x_{k,j},u_{k,j})<32r_{k,j}+\epsilon.$ Using this and \eqref{xixixi}, we find that, for any $(y,t)\in
\mathrm{supp}\,(a_{k,j})$,
\begin{align*}
(1-\eta)t\leq d\left(y,(O^\ast_k)^\complement\right)\leq d(y,u_{k,j})\leq d(y,x_{k,j})+d(x_{k,j},u_{k,j})
<34r_{k,j}+\epsilon.
\end{align*}
Letting $\epsilon\to0$ and using \eqref{2015}, we obtain,  for any $(y,t)\in \mathrm{supp}\,(a_{k,j})$,
\begin{align*}
d\left(y,(\widetilde{B}_{k,j})^\complement\right)\geq\frac{2^5+2}{1-\eta}r_{k,j}\geq t
\end{align*}
and hence $(y,t)\in T(\widetilde{B}_{k,j})$.
This proves \eqref{zhiji3}.

Furthermore, let $r\in(1,\infty)$ and $r'\in(1,\infty)$ be such that $\frac{1}{r'}+\frac{1}{r}=1.$
From  the fact that
$$\mathrm{supp}\,(a_{k,j})\subset\left[T_{1-\eta}(O^\ast_{k+1})\right]^\complement
=R_{1-\eta}\left([O_{k+1}^\ast]^\complement\right),$$
the H\"older inequality, and \eqref{zhiji3},
we deduce that, for any $h\in T^{r'}(\mathcal{X}^+)$,
\begin{align}\label{huimian}
&\left|\iint_{\mathcal{X}^+}a_{k,j}(y,t)\overline{h(y,t)}\,\frac{d\mu(y)\,dt}{t}\right|\\ \notag
&\quad\leq
\iint_{R_{1-\eta}([O_{k_+1}^\ast]^\complement)}\left|a_{k,j}(y,t)h(y,t)\right|\,\frac{d\mu(y)\,dt}{t}\\ \notag
&\quad\lesssim\int_{[O_{k+1}]^\complement}\iint_{\Gamma(x)}
\left|a_{k,j}(y,t)h(y,t)\right|\,\frac{d\mu(y)\,dt}{V(y,t)t}\,d\mu(x)\\ \notag
&\quad\lesssim\int_{[O_{k+1}]^\complement}\left[\iint_{\Gamma(x)\cap T(\widetilde{B}_{k,j})}
\left|a_{k,j}(y,t)\right|^2\,\frac{d\mu(y)\,dt}{V(y,t)t}\right]^{\frac{1}{2}}\mathcal{A}(h)(x)\,d\mu(x),
\end{align}
where we applied Lemma
\ref{tent-1} with $H(y,t):=|a_{k,j}(y,t)h(y,t)|[V(y,t)]^{-1}t^{-1}$ and $F:=[O_{k+1}]^\complement$ in the second step.
Notice that, for any $x\in\mathcal{X},$ if $(y,t)\in\Gamma(x)\cap T(\widetilde{B}_{k,j})$, then
$x\in \widetilde{B}_{k,j}$. By this, \eqref{huimian}, the H\"older inequality,
and the definitions of both $a_{k,j}$ and $O_{k+1}$, we conclude that, for any $h\in T^{r'}(\mathcal{X}^+)$,
\begin{align*}
&\left|\iint_{\mathcal{X}^+}a_{k,j}(y,t)\overline{h(y,t)}\,\frac{d\mu(y)\,dt}{t}\right|\\\notag
&\quad\lesssim
2^{-k}\left\|\mathbf{1}_{\widetilde{B}_{k,j}}\right\|_{X(\mathcal{X})}^{-1}
\int_{[O_{k+1}]^\complement\cap
\widetilde{B}_{k,j}}\left[\iint_{\Gamma(x)\cap T(\widetilde{B}_{k,j})}
|f(y,t)|^2\,\frac{d\mu(y)\,dt}{V(y,t)t}\right]^{\frac{1}{2}}\mathcal{A}(h)(x)\,d\mu(x)\\\notag
&\quad\lesssim2^{-k}\left\|\mathbf{1}_{\widetilde{B}_{k,j}}\right\|_{X(\mathcal{X})}^{-1}
\left\|\mathcal{A}(f)\right\|_{L^r([O_{k+1}]^\complement\cap
\widetilde{B}_{k,j})}\|h\|_{T^{r'}(\mathcal{X}^+)}\\\notag
&\quad\lesssim\left[\mu\left(\widetilde{B}_{k,j}\right)\right]^{\frac{1}{r}}
\left\|\mathbf{1}_{\widetilde{B}_{k,j}}\right\|_{X(\mathcal{X})}^{-1}
\|h\|_{T^{r'}(\mathcal{X}^+)}.
\end{align*}
This, together with the dual result on $T^r(\mathcal{X}^+)$ (see, for instance, \cite[Proposition 3.10]{a14}), further
implies that
$$\left\|a_{k,j}\right\|_{T^r(\mathcal{X}^+)}\lesssim\left[\mu\left(\widetilde{B}_{k,j}\right)\right]^{\frac{1}{r}}
\left\|\mathbf{1}_{\widetilde{B}_{k,j}}\right\|_{X(\mathcal{X})}^{-1}.$$
Using this and \eqref{zhiji3}, we find that $a_{k,j}$ is a $(T_X,\infty)$-atom, associated with the ball
$\widetilde{B}_{k,j}$, up to a harmless positive constant multiple.

Furthermore, from Proposition \ref{prfs} with $t:=s_0$, both (i) and (ii) of Lemma \ref{tent-ge},
the estimate that $(1-\gamma)\mathbf{1}_{O^\ast_k}\leq\mathcal{M}(\mathbf{1}_{O_k}),$ and Assumption
\ref{vector1} with $t:=\frac{s_0}{2}$ and $u:=2$, it follows that
\begin{align*}
&\left\|\left\{\sum_{k\in\mathbb{Z}_+}\sum_{j\in\mathbb{N}}
\left(\frac{\lambda_{k,j}}{\|\mathbf{1}_{\widetilde{B}_{k,j}}\|_{X}}\right)^{s_0}\mathbf{1}_{\widetilde{B}_{k,j}}
\right\}^{\frac{1}{s_0}}\right\|_{X(\mathcal{X})}\\
&\quad=\left\|\left\{\sum_{k\in\mathbb{Z}_+}\sum_{j\in\mathbb{N}}
2^{ks_0}\mathbf{1}_{\widetilde{B}_{k,j}}\right\}^{\frac{1}{s_0}}\right\|_{X(\mathcal{X})}
\lesssim\left\|\left\{\sum_{k\in\mathbb{Z}_+}\sum_{j\in\mathbb{N}}
2^{ks_0}\mathbf{1}_{5^{-1}B_{k,j}}\right\}^{\frac{1}{s_0}}\right\|_{X(\mathcal{X})}\\
&\quad\leq\left\|\left\{\sum_{k\in\mathbb{Z}_+}
\left[2^{\frac{ks_0}{2}}\mathbf{1}_{O^\ast_k}\right]^{2}\right\}^{\frac{1}{s_0}}\right\|_{X(\mathcal{X})}
\lesssim\left\|\left\{\sum_{k\in\mathbb{Z}_+}
\left[\mathcal{M}\left(2^{\frac{ks_0}{2}}\mathbf{1}_{O_k}\right)\right]^{2}\right\}^{\frac{1}
{s_0}}\right\|_{X(\mathcal{X})}\\\notag
&\quad\lesssim\left\|\left\{\sum_{k\in\mathbb{Z}_+}
2^{ks_0}\mathbf{1}_{O_k}\right\}^{\frac{1}{s_0}}\right\|_{X(\mathcal{X})}
\sim\left\|\left\{\sum_{k\in\mathbb{Z}_+}
2^{ks_0}\mathbf{1}_{O_k\setminus O_{k+1}}\right\}^{\frac{1}{s_0}}\right\|_{X(\mathcal{X})}\\\notag
&\quad\sim\|\mathcal{A}(f)\|_{X(\mathcal{X})}= \|f\|_{T_X(\mathcal{X}^+)}.
\end{align*}

Finally, assuming further that $f\in T^2(\mathcal{X}^+)\cap T_{X}(\mathcal{X}^+)$, we show that
\begin{align}\label{227}
f=\sum_{k\in\mathbb{Z}_+}\sum_{j\in\mathbb{N}}\lambda_{k,j}a_{k,j}
\end{align}
in $T^2(\mathcal{X}^+).$ By \eqref{eqoz3}, we conclude that, for any $(y,t)\in\mathcal{X}^+$,
\begin{align}\label{jiewei}
\int_{B(y,t)}\,\frac{d\mu(x)}{V(x,t)}
\sim\int_{B(y,t)}\,\frac{d\mu(x)}{V(y,t)}=1.
\end{align}
From this, the Fubini theorem, and the definitions of both $a_{k,j}$ and $\lambda_{k,j}$,
we deduce that
\begin{align}\label{nawei}
&\left\|\mathcal{A}\left(f-\sum_{|k|+|j|\leq N}\lambda_{k,j}a_{k,j}\right)\right\|_{L^2(\mathcal{X})}\\\notag
&\quad=\left[\int_0^\infty\int_\mathcal{X} \int_{B(y,t)}
\left|\sum_{|k|+|j|> N}\lambda_{k,j}a_{k,j}(y,t)\right|^2\,\frac{d\mu(x)}{V(x,t)}\,
\frac{d\mu(y)\,dt}{t}\right]^{\frac{1}{2}}\\\notag
&\quad\lesssim\left[\int_0^\infty\int_\mathcal{X}
\sum_{|k|+|j|> N}|\lambda_{k,j}a_{k,j}(y,t)|^2\,\frac{d\mu(y)\,dt}{t}\right]^{\frac{1}{2}}.
\end{align}
On the other hand, by the definitions of both $a_{k,j}$ and $\lambda_{k,j}$, we have, for almost every
$(x,t)\in\mathcal{X}^+,$
\begin{align*}
\sum_{k\in\mathbb{Z}_+}\sum_{j\in\mathbb{N}}\left|\lambda_{k,j}a_{k,j}\right|^2\leq|f(x,t)|^2.
\end{align*}
From this, the Lebesgue dominated convergence theorem, and \eqref{nawei}, it follows that \eqref{227} holds true
in $T^2(\mathcal{X}^+).$ This finishes the proof of Theorem \ref{thm-ad-tent}.
\end{proof}

\subsection{Proof of Theorem \ref{thm-mc}}\label{sec3.2}

In this subsection, we show Theorem \ref{thm-mc}, which can be easily deduced from Propositions \ref{thm-mc-re}
and  \ref{thm-at-dec} below.

For any given $\delta\in(0,\infty)$, let $\varphi$ be a measurable function from $\mathbb{C}$ to $\mathbb{C}$
such that there exists a positive constant $C_{(\delta)}$ satisfying that, for any $z\in\mathbb{C}$,
$|\varphi(z)|\leq C_{(\delta)}\frac{|z|^\delta}{1+|z|^{2\delta}}$. Then $\int_0^\infty|\varphi(t)|^2\,\frac{dt}{t}<\infty$.
By the functional calculi associated with $L$, we conclude that, for any $f\in L^2(\mathcal{X})$,
\begin{align}\label{bound3}
\int_0^\infty\int_{\mathcal{X}}\left|\varphi\left(t\sqrt{L}\right)(f)(x)\right|^2\,\frac{d\mu(x)\,dt}{t}
\leq\|f\|_{L^2(\mathcal{X})}^2\int_0^\infty|\varphi(t)|^2\,\frac{dt}{t}
\end{align}
(see, for instance, \cite[(3.14)]{hlmmy11}).
Taking $\varphi(z):=|z|^2e^{-|z|^2}$ for any $z\in \mathbb{C}$.
Then, using \eqref{bound3} and the Fubini theorem, we conclude that, for any $f\in L^2(\mathcal{X})$,
\begin{align}\label{bound}
\|S_L(f)\|_{L^2(\mathcal{X})}\lesssim \|f\|_{L^2(\mathcal{X})}.
\end{align}
We first prove the following conclusion.
\begin{proposition}\label{thm-mc-re}
Let $L$ be a non-negative self-adjoint operator on $L^2(\mathcal{X})$ satisfying the Davies--Gaffney estimate \eqref{dg}.
Assume that $X(\mathcal{X})$ is a $\mathrm{BQBF}$ space satisfying  Assumption \ref{vector2} for some $s_0\in(0,1]$ and
$q_0\in(s_0,2]$. Let $M\in(\frac{n}{2}[\frac{1}{s_0}-\frac{1}{2}],\infty)\cap\mathbb{N}$ and $\epsilon\in
(\frac{n}{s_0},\infty)$. Then there exists a positive constant $C$ such that, for any $f\in
\widetilde{H}^{M,\,\epsilon}_{X,\,L,\,\rm{mol}}(\mathcal{X})$,
\begin{align}\label{tabuhui1}
\|f\|_{H_{X,\,L}(\mathcal{X})}\leq C\|f\|_{H^{M,\,\epsilon}_{X,\,L,\,\rm{mol}}(\mathcal{X})},
\end{align}	
where $\widetilde{H}^{M,\,\epsilon}_{X,\,L,\,\rm{mol}}(\mathcal{X})$ is the same as in Definition \ref{defi-mol}.
\end{proposition}
\begin{proof}
Let $f\in \widetilde{H}^{M,\,\epsilon}_{X,\,L,\,\rm{mol}}(\mathcal{X})$. Then there exists a sequence
$\{\lambda_{j}\}_{j\in\mathbb{N}}\subset[0,\infty)$ and a sequence $\{\alpha_j\}_{j\in\mathbb{N}}$ of
$(X,M,\epsilon)$-molecules associated, respectively, with the balls $\{B_j\}_{j\in\mathbb{N}}$ satisfying
\begin{equation}\label{converge}
f=\sum_{j\in\mathbb{N}} \lambda_{j}\alpha_j
\end{equation}
in $L^2(\mathcal{X})$ and
\begin{align}\label{guji3}
\left\|\left\{\sum_{j\in\mathbb{N}}\left[\frac{\lambda_j}
{\|\mathbf{1}_{B_j}\|_{X(\mathcal{X})}}\right
]^{s_0}\mathbf{1}_{B_j}\right\}^{\frac{1}{s_0}}\right\|_X
\lesssim\|f\|_{H^{M,\,\epsilon}_{X,\,L,\,\rm{mol}}(\mathcal{X})}.
\end{align}
Recall that, for any $s\in(0,1]$ and $\{a_j\}_{j\in\mathbb{Z}_+}\subset [0,\infty)$,
\begin{align}\label{jiben}
\left(\sum_{j\in\mathbb{Z}_+}a_j\right)^s\leq\sum_{j\in\mathbb{Z}_+}a_j^{s}.
\end{align}
From this,  \eqref{converge}, and \eqref{bound}, we deduce that, for almost every $x\in\mathcal{X}$,
\begin{align*}
S_L(f)(x)&\leq\sum_{j\in\mathbb{N}} \lambda_jS_L(\alpha_j)(x)=\sum_{j\in\mathbb{N}}\sum_{i\in\mathbb{Z}_+}
\lambda_j S_L(\alpha_j)(x)\mathbf{1}_{U_i(B_j)}(x)\\ \notag
&\leq\left\{\sum_{i\in\mathbb{Z}_+}\sum_{j\in\mathbb{N}}\left[\lambda_j S_L(\alpha_j)(x)\mathbf{1}_{
U_i(B_j)}(x)\right]^{s_0}\right\}^{\frac{1}{s_0}}.
\end{align*}
By this and  the assumption that $X^{\frac{1}{s_0}}(\mathcal{X})$ is a $\mathrm{BBF}$ space,  we conclude that
\begin{align}\label{silu}
\|f\|_{H_{X,\,L}(\mathcal{X})}^{s_0}&=\left\|[S_L(f)]^{s_0}\right\|_{X^{\frac{1}{s_0}}(\mathcal{X})}
\leq\left\|\sum_{i\in\mathbb{Z}_+}\sum_{j\in\mathbb{N}}\left[\lambda_j S_L(\alpha_j)\mathbf{1}_{U_i(B_j)}\right
]^{s_0}\right\|_{X^{\frac{1}{s_0}}(\mathcal{X})}\\ \notag
&\leq\sum_{i\in\mathbb{Z}_+}\left\|\sum_{j\in\mathbb{N}}
\left[\lambda_jS_L(\alpha_j)\mathbf{1}_{U_i(B_j)}\right
]^{s_0}\right\|_{X^{\frac{1}{s_0}}(\mathcal{X})}.
\end{align}

We claim that,
for any $(X,M,\epsilon)$-molecule $\alpha$, associated with the ball $B:=B(x_B,r_B)$ for some $x_B\in\mathcal{X}$ and
$r_B\in(0,\infty)$, and any $i\in\mathbb{Z}_+$,
\begin{align}\label{guji2}
\|S_L(\alpha)\|_{L^2(U_i(B))}\lesssim 2^{-i\eta}[\mu(B)]^{\frac{1}{2}}\left\|\mathbf{1}_B\right\|_{X(\mathcal{X})}^{-1}
\end{align}
for some constant $\eta\in(n[\frac{1}{s_0}-\frac{1}{2}],\infty)$.
For the moment, we take this claim for granted and proceed with the proof of the present proposition. From \eqref{guji2},
it follows that, for any $i\in\mathbb{Z}_+$ and $j\in\mathbb{N}$,
\begin{align*}
\left\|S_L(\alpha_j)\right\|_{L^2(U_i(B_j))}\lesssim2^{-i\eta}[\mu(B_j)]^{\frac{1}{2}}
\left\|\mathbf{1}_{B_j}\right\|_{X(\mathcal{X})}^{-1}.
\end{align*}
Using this and applying Proposition \ref{pras} with $q:=2$ and $\theta:=2^i$, we find that, for any $i\in\mathbb{Z}_+,$
\begin{align*}
&\left\|\sum_{j\in\mathbb{N}}\left[\lambda_jS_L(\alpha_j)\mathbf{1}_{U_i(B_j)}\right]^{s_0}
\right\|_{X^{\frac{1}{s_0}}(\mathcal{X})}\\ \notag
&\quad\lesssim2^{-is_0[\eta-n(\frac{1}{s_0}-\frac{1}{2})]}\left\|\sum_{j\in\mathbb{N}}
\left[\frac{\lambda_j}{\|\mathbf{1}_{B_j}\|_{X(\mathcal{X})}}\right]^{s_0}\mathbf{1}_{B_j}
\right\|_{X^{\frac{1}{s_0}}(\mathcal{X})},
\end{align*}
which, together with \eqref{silu} and \eqref{guji3},  further implies that
\begin{align*}
\|f\|_{H_{X,\,L}(\mathcal{X})}^{s_0}&\leq\sum_{i\in\mathbb{Z}_+}\left\|\sum_{j\in\mathbb{N}}
\left[\lambda_jS_L(\alpha_j)\mathbf{1}_{U_i(B_j)}\right]^{s_0}\right\|_{X^{\frac{1}{s_0}}(\mathcal{X})}\\
&\lesssim\sum_{i\in\mathbb{Z}_+}2^{-is_0[\eta-n(\frac{1}{s_0}-\frac{1}{2})]}\left\|
\sum_{j\in\mathbb{N}}\left[\frac{\lambda_j}{\|\mathbf{1}_{B_j}\|_{X(\mathcal{X})}}\right]^{s_0}
\mathbf{1}_{B_j}\right\|_{X^{\frac{1}{s_0}}(\mathcal{X})}\\
&\lesssim\|f\|_{H^{M,\,\epsilon}_{X,\,L,\,\rm{mol}}(\mathcal{X})}^{s_0}.
\end{align*}
This proves \eqref{tabuhui1}.

Finally, we show \eqref{guji2}. If $i\in\{0,1,2,3\}$, then, from \eqref{bound}, \eqref{chicun}, \eqref{eqoz}, and
the assumption that $\epsilon>\frac{n}{s_0}>\frac{n}{2}$, we deduce that
\begin{align*}
\left\|S_L(\alpha)\right\|_{L^2(U_i(B))}^2&\lesssim\|\alpha\|_{L^2(\mathcal{X})}^2=
\sum_{k\in\mathbb{Z}_+}\|\alpha\|_{L^2(U_k(B))}^2\\ \notag
&\leq\sum_{k\in\mathbb{Z}_+} 2^{-2k\epsilon}\mu(2^kB)\|\mathbf{1}_B\|_{X(\mathcal{X})}^{-2}\\ \notag
&\lesssim\sum_{k\in\mathbb{Z}_+}2^{-2k\epsilon+kn}
\mu(B)\|\mathbf{1}_B\|_{X(\mathcal{X})}^{-2}\sim\mu(B)\|\mathbf{1}_B\|_{X(\mathcal{X})}^{-2},
\end{align*}
which is the desired estimate of this case.
If $i\in\mathbb{N}\cap[4,\infty)$, let $\delta\in(0,1)$ be determined later. Then we find that
\begin{align}\label{tabuhui5}
&\|S_L(\alpha)\|_{L^2(U_i(B))}^2\\ \notag
&\quad=\int_{U_i(B)}\int_0^{2^{i\delta}r_B}\int_{B(x,t)}
\left|t^2Le^{-t^2L}(\alpha)(y)\right|^2\,\frac{d\mu(y)\,dt}{V(x,t)t}\,d\mu(x)\\ \notag
&\quad\quad+\int_{U_i(B)}\int_{2^{i\delta}r_B}^\infty\int_{B(x,t)}\cdots\\\notag
&\quad=:\mathrm{I}+\mathrm{II}.
\end{align}
Since $\alpha$ is an $(X,M,\epsilon)$-molecule, it follows that there exists a $b\in \mathcal{D}(L^M)$ such that
$\alpha=L^M(b)$ and \eqref{chicun} holds true. By this, \eqref{bound}, \eqref{eqoz}, and the assumption that
$\epsilon>\frac{n}{s_0}>\frac{n}{2}$, we conclude that
\begin{align}\label{tabuhui7}
\mathrm{II}&=\int_{U_i(B)}\int_{2^{i\delta}r_B}^{\infty}\int_{B(x,t)}
t^{-4M}\left|\left(t^2L\right)^{M+1}e^{-t^2L}(b)\right|^2\,\frac{d\mu(y)\,dt}{V(x,t)t}\,d\mu(x)\\ \notag
&\leq\left(2^{i\delta}r_B\right)^{-4M}\left\|S_{L,\,M+1}(b)\right\|_{L^2(U_i(B))}^2
\lesssim\left(2^{i\delta}r_B\right)^{-4M}\|b\|_{L^2(\mathcal{X})}^2\\ \notag
&\sim\left(2^{i\delta}r_B\right)^{-4M}\sum_{k\in\mathbb{Z}_+}\|b\|_{L^2(U_k(B))}^2
\lesssim\sum_{k\in\mathbb{Z}_+} 2^{-2k\epsilon-4i\delta M}\mu\left(2^kB\right)
\|\mathbf{1}_B\|_{X(\mathcal{X})}^{-2}\\ \notag
&\lesssim\sum_{k\in\mathbb{Z}_+}2^{-2k\epsilon+kn-4i\delta M}\mu(B)\|\mathbf{1}_B\|_{X(\mathcal{X})}^{-2}
\lesssim2^{-4i\delta M}\mu(B)\|\mathbf{1}_B\|_{X(\mathcal{X})}^{-2}.
\end{align}
To estimate $\rm{I},$ for any $i\in\mathbb{N}\cap [4,\infty),$ let
\begin{align*}
S_i(B):=\left(2^{i+1}B\right)\setminus\left(2^{i-2}B\right)\ \text{and}\ \widetilde{S}_i(B):
=\left(2^{i+2}B\right)\setminus\left(2^{i-3}B\right).
\end{align*}
Observe that, if $t\in(0,2^{i\delta}r_B)$ and $x\in U_i(B)$, then $B(x,t)\subset S_i(B).$ From this, we deduce that
\begin{align}\label{zuozu1}
\mathrm{I}&\lesssim\int_{U_i(B)}\int_0^{2^{i\delta}r_B}\int_{S_i(B)}
\left|t^2Le^{-t^2L}\left(\alpha\mathbf{1}_{[\widetilde{S}_i(B)]^\complement}\right)(y)
\right|^2\,\frac{d\mu(y)\,dt}{V(x,t)t}\,d\mu(x)\\ \notag
&\quad+\int_{U_i(B)}\int_0^{2^{i\delta}r_B}\int_{B(x,t)}\left|t^2Le^{-t^2L}\left(\alpha\mathbf{1}_{
\widetilde{S}_i(B)}\right)(y)\right|^2\,\frac{d\mu(y)\,dt}{V(x,t)t}\,d\mu(x)\\ \notag
&=:\rm{I}_1+\rm{I}_2.
\end{align}
Using \eqref{bound}, \eqref{chicun}, and \eqref{eqoz},  we find that
\begin{align}\label{zuozu3}
\mathrm{I}_2&\leq\left\|S_L\left(\alpha\mathbf{1}
_{\widetilde{S}_i(B)}\right)\right\|_{L^2(\mathcal{X})}^2
\lesssim\|\alpha\|_{L^2(\widetilde{S}_i(B))}^2
=\sum_{k=i-2}^{i+2}\|\alpha\|_{L^2(U_k(B))}^2\\\notag
&\leq\sum_{k=i-2}^{i+2} 2^{-2k\epsilon}\mu\left(2^kB\right)\|\mathbf{1}_B\|_{X(\mathcal{X})}^{-2}
\lesssim\sum_{k=i-2}^{i+2}2^{-2k\epsilon+kn}\mu(B)\|\mathbf{1}_B\|_{X(\mathcal{X})}^{-2}\\ \notag
&\sim2^{-2i\epsilon+in}\mu(B)\|\mathbf{1}_B\|_{X(\mathcal{X})}^{-2}.
\end{align}
Moreover, by Remark \ref{1515}, we conclude that $\{tLe^{-tL}\}_{t\in(0,\infty)}$ satisfies the Davies--Gaffney estimate.
This, together with the estimate $\mathrm{dist}\,(S_i(B),[\widetilde{S}_i(B)]^\complement)\sim 2^ir_B$, further
implies that
\begin{align}\label{zuozu5}
\mathrm{I}_1&\lesssim\int_{U_i(B)}\int_0^{2^{i\delta}r_B}
e^{-c(2^ir_B/t)^2}\|\alpha\|_{L^2(\mathcal{X})}^2\,\frac{dt\,d\mu(x)}{V(x,t)t}\\ \notag
&\lesssim\|\alpha\|_{L^2(\mathcal{X})}^2\int_{U_i(B)}\int_0^{2^{i\delta}r_B}
\left(\frac{t}{2^ir_B}\right)^N\,\frac{dt\,d\mu(x)}{V(x,t)t},
\end{align}
where $c$ is a positive constant depending only on $L$ and $N\in\mathbb{N}$ is determined later. From \eqref{eqoz}
and \eqref{eqoz3}, it follows that, for any $x\in U_i(B)$,
\begin{align*}
&\int_0^{2^{i\delta}r_B}\left(\frac{t}{2^ir_B}\right)^N\,\frac{dt}{V(x,t)t}\\ \notag
&\quad=(2^ir_B)^{-N}\sum_{k=-\infty}^i\int_{2^{(k-1)\delta}r_B}^{2^{k\delta} r_B}\,\frac{dt}{V(x,t)t^{1-N}}\\ \notag
&\quad\sim(2^ir_B)^{-N}\sum_{k=-\infty}^i (2^{k\delta}r_B)^N\left[V\left(x,2^{k\delta}r_B\right)\right]^{-1}\\ \notag
&\quad\lesssim(2^ir_B)^{-N}\sum_{k=-\infty}^i (2^{k\delta}r_B)^N\left[1+\frac{d(x,x_B)}{2^{k\delta}r_B}\right]^{n}
\left[V\left(x_B,2^{k\delta}r_B\right)\right]^{-1}\\ \notag
&\quad\sim2^{-iN}\sum_{k=-\infty}^i2^{k\delta N+
(i-k\delta)n}\left[V\left(x_B,2^{k\delta}r_B\right)\right]^{-1}\\ \notag
&\quad\lesssim2^{-iN}\left[V\left(x_B,2^{i}r_B\right)\right]^{-1}\sum_{k=-\infty}^i2^{k\delta N+
2(i-k\delta)n}\\ \notag
&\quad\sim2^{-i(N-2n)(1-\delta)}\left[V\left(x_B,2^{i}r_B\right)\right]^{-1}.
\end{align*}
By this, \eqref{zuozu5},  \eqref{chicun}, \eqref{eqoz},
and $\epsilon\in(\frac{n}{s_0},\infty)$, we have
\begin{align*}
\mathrm{I}_1 &\lesssim2^{-i(N-2n)(1-\delta)}\|\alpha\|_{L^2(\mathcal{X})}^2
\sim2^{-i(N-2n)(1-\delta)}\sum_{k\in\mathbb{Z}_+}\|\alpha\|_{L^2(U_k(B))}^2\\ \notag
&\lesssim2^{-i(N-2n)(1-\delta)}\sum_{k\in\mathbb{Z}_+}
2^{-2k\epsilon+kn}\mu(B)\|\mathbf{1}_B\|_{X(\mathcal{X})}^{-2}\\
&\sim2^{-i(N-2n)(1-\delta)}\mu(B)\|\mathbf{1}_B\|_{X(\mathcal{X})}^{-2}.
\end{align*}
From this, \eqref{zuozu3}, \eqref{zuozu1}, \eqref{tabuhui5}, and \eqref{tabuhui7}, we deduce that
\begin{align*}
\|S_L(\alpha)\|_{L^2(U_i(B))}\lesssim2^{-i\eta}[\mu(B)]^{\frac{1}{2}}\|\mathbf{1}_B\|_{X(\mathcal{X})}^{-1},
\end{align*}
where
\begin{align*}
\eta:=\min\left\{\frac{1}{2}(N-2n)(1-\delta),\epsilon-\frac{n}{2},2M\delta\right\}.
\end{align*}
By the assumptions that $\epsilon\in(\frac{n}{s_0},\infty)$ and $M\in\mathbb{N}\cap(\frac{n}{2}[\frac{1}{s_0}
-\frac{1}{2}],\infty)$, we can choose a $\delta\in(0,1)$ and  an $N\in\mathbb{N}$ large enough such that
$\eta\in (n[\frac{1}{s_0}-\frac{1}{2}],\infty)$. This proves \eqref{guji2}, which completes the proof of
Proposition \ref{thm-mc-re}.
\end{proof}

In what follows, for any operator $T$, we denote its integral kernel by $K_T$. Let $L$ be a non-negative
self-adjoint operator on $L^2(\mathcal{X})$ satisfying the Davies--Gaffney estimate \eqref{dg}.
By \cite[Theorem 3.14]{CS08}, we find that there exists a positive constant $C$ such that, for any $t\in(0,\infty)$,
the kernel $K_T$ of the operator $T:=\cos(t\sqrt{L})$ satisfies
\begin{align}\label{2039}
\mathrm{supp}\,(K_T)\subset D_t:=\left\{(x,y)\in\mathcal{X}\times\mathcal{X}:\ d(x,y)\leq C t\right\}.
\end{align}
The following lemma is just \cite[Lemma 3.5]{hlmmy11}.

\begin{lemma}\label{guanjian}
Let $\psi\in C^\infty_{\rm{c}}(\mathbb{R})$ be even, $\psi\not\equiv0$, and $\mathrm{supp}\,(\psi)\subset(-C,C)$,
where $C$ is the same as in \eqref{2039}. Denote by $\Psi$ the Fourier transform of $\psi.$ Assume that $L$ is
a non-negative self-adjoint operator on $L^2(\mathcal{X})$ satisfying the Davies--Gaffney estimate \eqref{dg}.
Then, for any $k\in\mathbb{N}$ and $t\in(0,\infty)$, the kernel $K_{(t^2L)^k\Psi(t\sqrt{L})}$ of $(t^2L)^k\Psi(t\sqrt{L})$
satisfies
\begin{align*}
\mathrm{supp}\,\left(K_{(t^2L)^k\Psi(t\sqrt{L})}\right)\subset\left\{(x,y)\in\mathcal{X}\times\mathcal{X}:\
d(x,y)\leq t\right\}.
\end{align*}
\end{lemma}

\begin{proposition}\label{thm-at-dec}
Let $L$ be a non-negative self-adjoint operator on $L^2(\mathcal{X})$ satisfying the Davies--Gaffney estimate \eqref{dg}.
Assume that $X(\mathcal{X})$ is a $\mathrm{BQBF}$ space satisfying  Assumption \ref{vector1} for some $p\in(0,\infty)$.
Let $M\in\mathbb{N}$ and  $s_0\in(0,p].$ Then, for any $f\in H_{X,\,L}(\mathcal{X})\cap \overline{R(L)}$,
there exists a sequence $\{\lambda_j\}_{j\in\mathbb{N}}\subset[0,\infty)$ and a sequence $\{\alpha_j\}_{j\in\mathbb{N}}$
of $(X,M)$-atoms associated, respectively, with the balls $\{B_j\}_{j\in\mathbb{N}}\subset \mathcal{X}$ such that
\begin{align*}
f=\sum_{j\in\mathbb{N}}\lambda_j\alpha_j
\end{align*}
in $L^2(\mathcal{X})$ and
\begin{align*}
\left\|\left\{\sum_{j=1}^\infty
\left(\frac{\lambda_j}{\|\mathbf{1}_{B_j}\|_{X(\mathcal{X})}}\right)^{s_0}\mathbf{1}_{B_j}
\right\}^{\frac{1}{s_0}}\right\|_X\lesssim \|f\|_{H_{X,\,L}(\mathcal{X})},
\end{align*}
where the implicit positive constant is independent of $f$.
\end{proposition}
\begin{proof}
Let $\Psi$ be the same as in Lemma \ref{guanjian} and $\Phi(t):=t^{2M+2}\Psi(t)$ for any $t\in(0,\infty)$.
For any $g\in T^2(\mathcal{X}^+)$ and $x\in\mathcal{X}$, let
\begin{align*}
\pi_{\Phi,\,L}(g)(x):=C_{(\Phi)}\int_0^\infty \Phi\left(t\sqrt{L}\right)(g(\cdot,t))(x)\,\frac{dt}{t},
\end{align*}
where $C_{(\Phi)}$ is a positive constant satisfying
\begin{align*}
C_{(\Phi)}\int_0^\infty\Phi(t)t^2e^{-t^2}\,\frac{dt}{t}=1.
\end{align*}
Let $f\in H_{X,\,L}(\mathcal{X})\cap \overline{R(L)}$. By the functional calculi for $L$, we find that
\begin{align}\label{ige}
f=C_{(\Phi)}\int_0^\infty \Phi\left(t\sqrt{L}\right)t^2Le^{-t^2L}f\,\frac{dt}{t}=\pi_{\Phi,\,L}
\left(t^2Le^{-t^2L}(f)\right)
\end{align}
in $L^2(\mathcal{X}).$ From \eqref{bound} and the assumption that $f\in L^2(\mathcal{X})$, we deduce that
$t^2Le^{-t^2L}(f)\in T_X(\mathcal{X}^+)\cap T^2(\mathcal{X}^+).$ By this and Theorem \ref{thm-ad-tent}, we
conclude that there exists a sequence $\{\lambda_j\}_{j=1}^\infty\subset[0,\infty)$ and a sequence $\{a_j\}_{j=1}^\infty$
of $(T_X,\infty)$-atoms associated, respectively,  with the balls $\{B_j\}_{j=1}^\infty\subset \mathcal{X}$ such that
\begin{align*}
t^2Le^{-t^2L}(f)=\sum_{j\in\mathbb{N}}
\lambda_ja_j
\end{align*}
in $T^2(\mathcal{X}^+)$ and
\begin{align*}
\left\|\left\{\sum_{j=1}^\infty\left[\frac{\lambda_j}{\|\mathbf{1}_{B_j}\|_{X(\mathcal{X})}}\right]^{s_0}
\mathbf{1}_{B_j}\right\}^{\frac{1}{s_0}}\right\|_{X(\mathcal{X})}
\lesssim\left\|t^2Le^{-t^2L}(f)\right\|_{T_X(\mathcal{X}^+)}= \|f\|_{H_{X,\,L}(\mathcal{X})}.
\end{align*}
This, combined with \eqref{ige} and the fact that $\pi_{\Phi,\,L}$ is bounded from $T^2(\mathcal{X}^+)$ to
$L^2(\mathcal{X})$ (see, for instance, \cite[Proposition 4.1(i)]{jy11}), further implies that
\begin{align*}
f=\pi_{\Phi,\,L}\left(t^2Le^{-t^2L}(f)\right)=\sum_{j\in\mathbb{N}}\lambda_j\pi_{\Phi,\,L}(a_j)
=:\sum_{j\in\mathbb{N}}\lambda_j\alpha_j
\end{align*}
in $L^2(\mathcal{X})$.

For any $j\in\mathbb{N}$ and $x\in\mathcal{X}$, let
\begin{align*}
b_j(x):=\int_0^\infty t^{2M+2}L\Psi\left(t\sqrt{L}\right)\left(\alpha_j(\cdot,t)\right)(x)\,\frac{dt}{t}.
\end{align*}
Then $\alpha_j=L^M(b_j)$ for any $j\in\mathbb{N}.$ Moreover, from Lemma \ref{guanjian}, it follows that,
for any $k\in\{0,\ldots,M\}$, $\mathrm{supp}\,(L^{k}(b_j))\subset B_j$.
By the Fubini theorem and the assumption that $L$ is self-adjoint, we find that, for any given $g\in L^2(\mathcal{X})$,
\begin{align*}
&\int_{\mathcal{X}}\left(r_{B_j}^{2}L\right)^k(b_j)(x) \overline{g(x)}\,d\mu(x)\\ \notag
&\quad=r_{B_j}^{2k}\int_0^\infty\int_{\mathcal{X}}t^{2M-2k}\left(t^2L\right)^{k+1}
\Psi\left(t\sqrt{L}\right)(a_j(\cdot,t))(x)\overline{g(x)}\,d\mu(x)\,\frac{dt}{t}\\ \notag
&\quad=r_{B_j}^{2k}\int_0^\infty\int_{\mathcal{X}}t^{2M-2k}a_j(x,t)\overline{(t^2L)^{k+1}
\Psi\left(t\sqrt{L}\right)(g)(x)}\,d\mu(x)\,\frac{dt}{t}.
\end{align*}
From this, the assumption that $a$ is a $(T_X,\infty)$-atom, the H\"older inequality, and \eqref{bound3}, we deduce
that, for any $g\in L^2(\mathcal{X}),$
\begin{align*}
&\left|\int_{\mathcal{X}}\left(r_{B_j}^{2}L\right)^kb_j(x)\overline{g(x)}\,d\mu(x)\right|\\ \notag
&\quad\leq r_{B_j}^{2M}\iint_{T(B)}\left|a_j(x,t)\right|\left|(t^2L)^{k+1}\Psi\left(t\sqrt{L}\right)(g)(x)\right|\,
\frac{d\mu(x)\,dt}{t}\\ \notag
&\quad\leq r_{B_j}^{2M}\left\|a_j\right\|_{T^2(\mathcal{X}^+)}\left\|(t^2L)^{M+1-k}
\Psi\left(t\sqrt{L}\right)(g)\right\|_{T^2(\mathcal{X}^+)}\\ \notag
&\quad\lesssim r_{B_j}^{2M}\left\|a_j\right\|_{T^2(\mathcal{X}^+)} \|g\|_{L^2(\mathcal{X})}
\lesssim r_{B_j}^{2M}\left[\mu\left(B_j\right)\right]^{\frac{1}{2}}
\left\|\mathbf{1}_{B_j}\right\|_{X(\mathcal{X})}^{-1}
\|g\|_{L^2(\mathcal{X})},
\end{align*}
which further implies that, for any $j\in\mathbb{Z}_+$ and $k\in\{0,\ldots,M\}$,
\begin{align*}
\left\|\left(r_{B_j}^{2}L\right)^k(b_j)\right\|_{L^2(\mathcal{X})}
\lesssim r_{B_j}^{2M}\left[\mu\left(B_j\right)\right]^{\frac{1}{2}}
\left\|\mathbf{1}_{B_j}\right\|_{X(\mathcal{X})}^{-1}.
\end{align*}
By this, we conclude that, for any $j\in\mathbb{N},$ $\alpha_j$ is an $(X,M)$-atom, associated with $B_j$,
up to a harmless constant. This finishes the proof of Proposition \ref{thm-at-dec}.
\end{proof}

Finally, we prove Theorem \ref{thm-mc} by using Propositions \ref{thm-mc-re} and \ref{thm-at-dec}.

\begin{proof}[Proof of Theorem \ref{thm-mc}]
By Propositions \ref{thm-mc-re} and \ref{thm-at-dec} and Remark \ref{buceng}, we conclude that the spaces
$H_{X,\,L}(\mathcal{X})\cap \overline{R(L)}$, $\widetilde{H}^{M,\,\epsilon}_{X,\,L,\,\rm{mol}}(\mathcal{X})$,
and $\widetilde{H}^{M}_{X,\,L,\,\rm{at}}(\mathcal{X})$ coincide with equivalent quasi-norms. From this and a
density argument, it follows that the spaces $H_{X,\,L}(\mathcal{X})$, $H^{M,\,\epsilon}_{X,\,L,\,\rm{mol}}(\mathcal{X})$,
and $H^{M}_{X,\,L,\,\rm{at}}(\mathcal{X})$ coincide with equivalent quasi-norms. This finishes the proof of
Theorem \ref{thm-mc}.
\end{proof}

\section{Applications to Boundedness of Operators}\label{section4}
In this section, we give several applications of the atomic and the molecular characterizations of
$H_{X,\,L}(\mathcal{X})$ to the boundedness of spectral multipliers and Schr\"odinger groups and also to the
Littlewood--Paley characterization of $H_{X,\,L}(\mathcal{X})$.

\subsection{Boundedness of Spectral Multipliers}\label{sec4.3}
In this subsection, we prove a H\"ormander type spectral multiplier theorem for $L$ on $H_{X,\,L}(\mathcal{X})$.
We begin with some concepts. Let $s\in[0,\infty)$. The \emph{space} $C^s(\mathbb{R})$ is defined to be the set
of all the  functions $f$ on $\mathbb{R}$ such that
\begin{align*}
\|f\|_{C^s(\mathbb{R})}:=
\begin{cases}
\displaystyle\sum_{k=0}^s\sup_{x\in\mathbb{R}}\left|f^{(k)}(x)\right| &s\in\mathbb{Z}_+,\\
\displaystyle\sum_{k=0}^{\lfloor s \rfloor}\sup_{x\in\mathbb{R}}\left|f^{(k)}(x)\right|
+\left\|f^{(\lfloor s \rfloor)}\right\|_{\mathrm{Lip}(s-\lfloor s \rfloor)}
&s\notin\mathbb{Z}_+,
\end{cases}
\end{align*}
is finite, where, for any $k\in\mathbb{Z}_+$, $f^{(k)}$ denotes the $k$-order derivative of $f$,
$\lfloor s\rfloor$ denotes the \emph{maximal integer not more than} $s$, and
\begin{align*}
\left\|f^{(\lfloor s \rfloor)}\right\|_{\mathrm{Lip}(s-\lfloor s \rfloor)}
:=\sup_{x,y\in\mathbb{R}\,x\not=y}\frac{|f^{(\lfloor s \rfloor)}(x)-f^{(\lfloor s \rfloor)}(y)|}{|x-y|^{s-\lfloor s \rfloor}}.
\end{align*}
Let $\phi\in C^\infty_{\mathrm{c}}(\mathbb{R})$ satisfy that
\begin{align}\label{620}
\mathrm{supp}\,(\phi)\subset\left(\frac{1}{4},1\right)
\ \text{and}\ \sum_{i\in\mathbb{Z}}\phi\left(2^{-i}x\right)=1
\end{align}
for any $x\in(0,\infty).$ The following conclusion is the main result of this subsection.

\begin{theorem}\label{thm-spec}
Let $X(\mathcal{X})$ be a $\mathrm{BQBF}$ space satisfying both Assumptions \ref{vector1} and \ref{vector2} for some
$p\in(0,\infty)$, $s_0\in(0,\min\{p,1\}]$, and $q_0\in(s_0,2]$. Assume that $L$ is a non-negative self-adjoint
operator on $L^2(\mathcal{X})$ satisfying the Davies--Gaffney estimate \eqref{dg}. Let $s\in(\frac{n}{s_0},\infty)$
and $\phi\in C^\infty_{\mathrm{c}}(\mathbb{R})$ satisfy \eqref{620}. If the bounded Borel function $m:\
[0,\infty)\to\mathbb{C}$ satisfies that
\begin{align}\label{2022930c}
C(\phi,s):=\sup_{t\in(0,\infty)}\|\phi(\cdot)m(t\cdot)\|_{C^s(\mathbb{R})}+|m(0)|<\infty,
\end{align}
then
there exists a positive constant $C$ such that, for any $f\in H_{X,\,L}(\mathcal{X})$,
\begin{align*}
\|m(L)(f)\|_{H_{X,\,L}(\mathcal{X})}\leq C\|f\|_{H_{X,\,L}(\mathcal{X})}.
\end{align*}
\end{theorem}
Applying Theorem \ref{thm-spec} with $X(\mathcal{X}):=L^r(\mathcal{X})$, we obtain the
following conclusion; since their proofs are similar to that of Theorem \ref{ls10}, we omit the details here.
\begin{theorem}\label{ls40}
Let $L$ be a non-negative self-adjoint
operator on $L^2(\mathcal{X})$ satisfying the Davies--Gaffney estimate \eqref{dg}, $r\in(0,2)$, and $s\in(\frac{n}{\min\{1,r\}},\infty)$.
Assume that $\varphi\in C^\infty_{\mathrm{c}}(\mathbb{R})$ satisfies \eqref{620} and  $m:\ [0,\infty)\to\mathbb{C}$
is a bounded Borel function satisfying \eqref{2022930c}. Then there exists a positive constant $C$ such that,
for any $f\in H^{r}_{L}(\mathcal{X})$,
\begin{align*}
\|m(L)(f)\|_{H^{r}_{L}(\mathcal{X})}\leq C\|f\|_{H^{r}_{L}(\mathcal{X})}.
\end{align*}
\end{theorem}
\begin{remark}
Let $r\in (0,1]$. In this case, Theorem \ref{ls40} was obtained in \cite[Theorem 1.1]{dy11}.
However, in the case when $r\in(1,2)$, Theorem \ref{ls40} is new.
\end{remark}

To prove Theorem \ref{thm-spec}, we need the following conclusion.
\begin{lemma}\label{lem-spec}
Assume that $X(\mathcal{X})$ is a $\mathrm{BQBF}$ space satisfying both Assumptions \ref{vector1} and \ref{vector2} for
some $p\in(0,\infty)$, $s_0\in(0,\min\{p,1\}]$, and $q_0\in(s_0,2]$. Let $L$ be a non-negative self-adjoint
operator on $L^2(\mathcal{X})$ satisfying the Davies--Gaffney estimate \eqref{dg}, $m$ a bounded Borel function
on $\mathbb{R}$, and $M\in(\frac{n}{2}[\frac{1}{s_0}-\frac{1}{2}],\infty)\cap\mathbb{N}$. Assume that there
exist constants $D\in (\frac{n}{s_0},\infty)$ and $C\in(0,\infty)$ such that, for
any $j\in\mathbb{N}\cap [2,\infty)$ and any $f\in L^2(\mathcal{X})$
with $\mathrm{supp}\,(f)\subset B:=B(x_B,r_B),$ 
\begin{align}\label{spec00}
\left\|m(L)\left(I-e^{-r_B^2L}\right)^M(f)\right\|_{L^2(U_j(B))}\leq C2^{-jD}\|f\|_{L^2(\mathcal{X})},
\end{align}
where $x_B\in \mathcal{X}$ and $r_B\in(0,\infty)$.
Then there exists a positive constant $C_1$ such that, for any $f\in H_{X,\,L}(\mathcal{X})$,
\begin{align*}
\|m(L)(f)\|_{H_{X,\,L}(\mathcal{X})}\leq C_1\|f\|_{H_{X,\,L}(\mathcal{X})}.
\end{align*}
\end{lemma}

\begin{proof}
Let $f\in H_{X,\,L}(\mathcal{X})\cap \overline{R(L)}$. By Proposition \ref{thm-at-dec}, we find  that there
exists a sequence $\{\lambda_j\}_{j\in\mathbb{N}}\subset[0,\infty)$ and a sequence $\{\alpha_j\}_{j\in\mathbb{N}}$ of
$(X,2M)$-atoms associated, respectively, with the balls $\{B_j\}_{j\in\mathbb{N}}$ such that
\begin{align}\label{2022930a}
f=\sum_{j\in\mathbb{N}}\lambda_j\alpha_j
\end{align}
in $L^2(\mathcal{X})$ and
\begin{align}\label{2022930b}
\left\|\left\{\sum_{j=1}^\infty\left(\frac{\lambda_j}{\|\mathbf{1}_{B_j}\|_{X}}\right)^{s_0}\mathbf{1}_{B_j}
\right\}^{\frac{1}{s_0}}\right\|_X\lesssim \|f\|_{H_{X,\,L}(\mathcal{X})},
\end{align}
where the implicit positive constant is independent of $f$.

Now, we show that, for any $(X,2M)$-atom $\alpha$ associated with the ball $B:=B(x_B,r_B)\subset\mathcal{X}$
for some $x_B\in\mathcal{X}$ and $r_B\in(0,\infty)$, $m(L)(\alpha)$ is an $(X,M,D)$-molecule, associated with
the ball $B$, up to a harmless positive constant multiple. Since $\alpha$ is an $(X,2M)$-atom, it follows that
there exists a $b\in \mathcal{D}(L^{2M})$ such that $\alpha=L^{2M}(b)$. By an argument similar to that used in
the estimation of \cite[(3.4)]{dy11}, we conclude that, for any $k\in\{0,\ldots,M\}$ and $j\in\mathbb{Z}_+$,
\begin{align*}
\left\|\left(r_B^{2}L\right)^km(L)L^M(b)\right\|_{L^2(U_j(B))}
\lesssim2^{-jD} r_B^{2M}\left[\mu\left(2^jB\right)\right]^{\frac{1}{2}}\|\mathbf{1}_{B}\|_{X(\mathcal{X})}^{-1}.
\end{align*}
This proves that $m(L)(\alpha)$ is an $(X,M,D)$-molecule up to a harmless positive constant multiple.

From the above argument, we deduce that, for any $j\in\mathbb{N},$ $m(L)(\alpha_j)$ is an $(X,M,D)$-molecule up
to a harmless positive constant multiple. Meanwhile, by the boundedness of $m(L)$ on $L^2(\mathcal{X})$ and \eqref{2022930a},
we conclude that
\begin{align*}
m(L)f=\sum_{j\in\mathbb{N}}\lambda_jm(L)\alpha_j
\end{align*}
in $L^2(\mathcal{X})$. From this, \eqref{2022930b}, and Proposition \ref{thm-mc-re},  it follows that
\begin{align*}
\left\|m(L)f\right\|_{H_{X,\,L}(\mathcal{X})}\lesssim\left\|\left\{\sum_{j=1}^\infty
\left(\frac{\lambda_j}{\|\mathbf{1}_{B_j}\|_{X}}\right)^{s_0}\mathbf{1}_{B_j}\right\}^{\frac{1}{s_0}}\right\|_X
\lesssim \|f\|_{H_{X,\,L}(\mathcal{X})},
\end{align*}
which, together with the fact that $H_{X,\,L}(\mathcal{X})\cap \overline{R(L)}$ is dense in $H_{X,\,L}(\mathcal{X})$,
further implies that the conclusion of the present lemma holds true. This finishes the  proof of Lemma \ref{lem-spec}.
\end{proof}

Meanwhile, we also need the following useful conclusion whose proof is quite easy; we omit the details here.
\begin{lemma}\label{yi}
Let $t\in(0,\infty)$ and $0<\alpha<\beta<\infty$. Then
\begin{align*}
\sum_{\ell\in\mathbb{Z}}\left(2^\ell t\right)^{-\alpha}\min\left\{1, \left(2^\ell t\right)^\beta\right\}
\leq\frac{1}{1-2^{\alpha-\beta}}+\frac{1}{1-2^{-\alpha}}.
\end{align*}
\end{lemma}

Now, we prove Theorem \ref{thm-spec} via using Lemmas \ref{lem-spec} and \ref{yi}.

\begin{proof}[Proof of Theorem \ref{thm-spec}]
Since $m$ satisfies \eqref{2022930c} if and only if the function $\lambda\mapsto m(\lambda^2)$ has the
same property, we consider $m(\sqrt{L})$ instead of $m(L)$. Observe that $$m\left(\sqrt{L}\right)=[m(\cdot)-m(0)]\left(\sqrt{L}\right)+m(0)I.$$
Replacing $m$ by $m-m(0)$, we may assume that $m(0)=0$. Let $s\in(\frac{n}{s_0},\infty)$ and $M\in(\frac{s}{2},
\infty)\cap\mathbb{N}$. For any $i\in\mathbb{N}$, $r\in(0,\infty)$, and $x\in\mathbb{R}$, let
\begin{align*}
F_{r,M}(x):=m(x)\left(1-e^{-r^2x^2}\right)^M \ \text{and}\
F_{i,r,M}(x):=\phi\left(2^{-i}x\right)m(x)\left(1-e^{-r^2x^2}\right)^M.
\end{align*}
By \eqref{620}, we conclude that, for any $x\in\mathbb{R}$,
\begin{align*}
F_{r,M}(x)=\lim_{M\to\infty}\sum_{i=-M}^MF_{i,r,M}(x).
\end{align*}
From this and the functional calculi associated with $L$, it follows that, for any $g\in L^2(\mathcal{X})$,
\begin{align*}
m\left(\sqrt{L}\right)\left(I-e^{-r^2L}\right)^M(g)=F_{r,M}\left(\sqrt{L}\right)(g)
=\lim_{M\to\infty}\sum_{i=-M}^MF_{i,r,M}\left(\sqrt{L}\right)(g)
\end{align*}
in $L^2(\mathcal{X})$. By \cite[(4.8)]{dy11}, we find that, for any $b\in L^2(\mathcal{X})$ with
$\mathrm{supp}\,(b)\subset B:=B(x_B,r_B)$ for some $x_B\in\mathcal{X}$ and $r_B\in(0,\infty)$, any
$i\in\mathbb{Z}$, and any $j\in\{2,\ldots\}$,
\begin{align*}
\left\|F_{i,r_B,M}\left(\sqrt{L}\right)(b)\right\|_{L^2(U_j(B))}\lesssim
C(\phi,s)2^{-s(i+j)}r_B^{-s}\min\left\{1,2^{2Mi}r_B^{2M}\right\}\|b\|_{L^2(B)},
\end{align*}
which, combined with Lemma \ref{yi}, further implies that
\begin{align*}
&\left\|m\left(\sqrt{L}\right)\left(I-e^{-r^2L}\right)^M(b)\right\|_{L^2(U_j(B))}\\
&\quad\lesssim2^{-js}\lim_{M\to\infty}\sum_{i=-M}^M2^{-si}r_B^{-s}\min\left\{1,2^{2Mi}r_B^{2M}\right\}
\|b\|_{L^2(B)}\lesssim2^{-js}\|b\|_{L^2(B)},
\end{align*}
where $C(\phi,s)$ is the same as in \eqref{2022930c}.
This shows that $m(L)$ satisfies \eqref{spec00}. Thus, all the assumptions of Lemma \ref{lem-spec} are
satisfied and hence the desired conclusion of the present theorem holds true. This finishes the proof
of Theorem \ref{thm-spec}.
\end{proof}

\subsection{Littlewood--Paley Characterizations of $H_{X,\,L}(\mathcal{X})$}\label{sec4.1}
In this subsection, we establish the Littlewood--Paley characterization of $H_{X,\,L}(\mathcal{X})$.
Recall that, for any given $\lambda\in(0,\infty)$ and any $f\in L^2(\mathcal{X})$, the \emph{Littlewood--Paley
$g$-function} $g_L(f)$ and the \emph{Littlewood--Paley $g_{\lambda}^\ast$-function} $g_{\lambda,L}^\ast (f)$
are defined, respectively, by setting, for any $x\in\mathcal{X},$
\begin{align*}
g_L(f)(x):=\left[\int_0^\infty\left|t^2Le^{-t^2L}(f)(x)\right|^2\,\frac{dt}{t}\right]^{\frac{1}{2}}
\end{align*}
and
\begin{align*}
g_{\lambda,\,L}^\ast(f)(x):=\left\{\int_0^\infty\int_{\mathcal{X}}\left[\frac{t}{t+d(x,y)}\right]^\lambda
\left|t^2Le^{-t^2L}(f)(y)\right|^2\,\frac{d\mu(y)\,dt}{V(x,t)t}\right\}^{\frac{1}{2}}.
\end{align*}
In a similar way, the $g_L$-\emph{adapted} and the $g_{\lambda,L}^\ast$-\emph{adapted Hardy spaces}
$H_{X,\,L,\,g}(\mathcal{X})$ and $H_{X,\,L,\,g_{\lambda}^\ast}(\mathcal{X})$ are defined in the way same as
$H_{X,\,L}(\mathcal{X})$ with $S_L(f)$ replaced, respectively, by $g_L(f)$ and $g^\ast_{\lambda,\,L}(f)$.

If  $L$ satisfies the Davies--Gaffney estimate \eqref{dg}, we have the
following conclusion.
\begin{theorem}\label{thm-g-2}
Let $L$ be a non-negative self-adjoint operator on $L^2(\mathcal{X})$ satisfying
the Davies--Gaffney estimate \eqref{dg}. Assume that $X(\mathcal{X})$ is a $\mathrm{BQBF}$ space
satisfying both Assumptions \ref{vector1} and \ref{vector2} for some $p\in(0,\infty)$,
$s_0\in(0,\min\{p,1\}),$ and $q_0\in(s_0,2]$. Let $\lambda\in(\frac{2n}{s_0},\infty)$.
Then both $g_{L}$ and $g_{\lambda,\,L}^\ast$ are bounded from $H_{X,\,L}(\mathcal{X})$ to $X(\mathcal{X})$
\end{theorem}

To obtain the Littlewood--Paley characterization of $H_{X,\,L}(\mathcal{X})$ and the boundedness of
the operator $(I+L)^{-s}e^{itL}$ on $H_{X,\,L}(\mathcal{X})$, we need the following assumption on $L.$

\begin{assumption}
The kernels of the semigroup $\{e^{-tL}\}_{t\in(0,\infty)}$, denoted by $\{K_t\}_{t\in(0,\infty)}$, are
measurable functions on $\mathcal{X}\times\mathcal{X}$ and satisfy the Gaussian upper bound estimate, that is,
there exist positive constants $C$ and $c$ such that, for any $t\in(0,\infty)$ and $x,y\in\mathcal{X}$,
\begin{align}\label{gauss}
|K_t(x,y)|\leq\frac{C}{V(x,t^{1/2})}\exp\left\{-\frac{[d(x,y)]^2}{ct}\right\}.
\end{align}
\end{assumption}

\begin{remark}
Let $L$ be a non-negative self-adjoint operator on $L^2(\mathcal{X})$ satisfying the Gaussian upper bound
estimate \eqref{gauss}. Then it is easy to find that $L$ satisfies the Davies--Gaffney estimate \eqref{dg}.
Moreover, by \cite[Section 2.6]{hlmmy11}, we have $N(L)=\{0\}$ and hence $\overline{R(L)}=L^2(\mathcal{X}).$
\end{remark}

\begin{theorem}\label{thm-g}
Let $L$ be a non-negative self-adjoint operator on $L^2(\mathcal{X})$ satisfying the Gaussian upper bound
estimate \eqref{gauss}, and $X(\mathcal{X})$ be a $\mathrm{BQBF}$ space satisfying Assumption \ref{vector1} for some
$p\in(0,\infty)$. Assume that $s_0\in(0,\min\{p,1\})$, $\lambda\in(\frac{2n}{s_0},\infty),$
$X^{1/s_0}(\mathcal{X})$ is $\mathrm{BBF}$ space, and the Hardy--Littlewood maximal operator $\mathcal{M}$ is bounded
on $(X^{1/s_0}(\mathcal{X}))'.$ Then the spaces $H_{X,\,L}(\mathcal{X})$, $H_{X,\,L,\,g}(\mathcal{X}),$
and $H_{X,\,L,\,g^\ast_{\lambda}}(\mathcal{X})$ coincide with equivalent quasi-norms.
\end{theorem}

In the case when $X(\mathcal{X}):=L^{r}(\mathcal{X})$,
we denote $H_{X,\,L,\,g}(\mathcal{X})$ and $H_{X,\,L,\,g^\ast_{\lambda}}(\mathcal{X})$, respectively, by
$H_{L,\,g}^{r}(\mathcal{X})$ and $H_{L,\,g^\ast_{\lambda}}^{r}(\mathcal{X})$. Applying Theorems  \ref{thm-g-2} and \ref{thm-g} with $X(\mathcal{X}):=L^r(\mathcal{X})$, we have the
following conclusions; since their proofs are similar to that of Theorem \ref{ls10}, we omit the details here.
\begin{theorem}\label{ls20}
\begin{itemize}
\item[$\mathrm{(i)}$] Let  $L$ be a non-negative self-adjoint operator on $L^2(\mathcal{X})$ satisfying
the Davies --Gaffney estimate \eqref{dg}, $r\in(0,2)$, and $\lambda\in(\frac{2n}{\min\{1,r\}},\infty)$.
Then both $g_{L}$ and $g_{\lambda,\,L}^\ast$ are bounded from $H_{L}^{r}(\mathcal{X})$
to $L^r(\mathcal{X}).$
\item[$\mathrm{(ii)}$] Let $L$ be a non-negative self-adjoint operator on $L^2(\mathcal{X})$ satisfying the Gaussian upper bound
estimate \eqref{gauss}, $r\in(0,\infty)$, and $\lambda\in(\frac{2n}{\min\{1,r\}},\infty)$.
Then the spaces $H_{L}^{r}(\mathcal{X})$, $H_{L,\,g}^{r}(\mathcal{X}),$ and $H_{L,\,g^\ast_{\lambda}}^{r}(\mathcal{X})$ coincide with equivalent quasi-norms.
\end{itemize}
\end{theorem}

To prove Theorems \ref{thm-g-2} and \ref{thm-g},
we first show the following conclusion.
\begin{proposition}\label{jiushijie}
Let $L$ be a non-negative self-adjoint operator on $L^2(\mathcal{X})$ satisfying
the Davies--Gaffney estimate \eqref{dg}. Assume that $X(\mathcal{X})$ is a $\mathrm{BQBF}$ space
satisfying both Assumptions \ref{vector1} and \ref{vector2} for some $p\in(0,\infty)$,
$s_0\in(0,\min\{p,1\}),$ and $q_0\in(s_0,2]$. Then $g_{L}$ is bounded from $H_{X,\,L}(\mathcal{X})$
to $X(\mathcal{X})$.
\end{proposition}
\begin{proof}
Let $f\in H_{X,\,L}(\mathcal{X})\cap\overline{R(L)}$ and $M\in(\frac{n}{2}[\frac{1}{s_0}-\frac{1}{2}],
\infty)\cap\mathbb{N}$. By Proposition \ref{thm-at-dec}, we find that  there exists a sequence $\{\lambda_j\}_{j\in\mathbb{N}}
\subset[0,\infty)$ and a sequence $\{\alpha_j\}_{j\in\mathbb{N}}$ of $(X,M)$-atoms associated, respectively,
with the balls $\{B_j\}_{j\in\mathbb{N}}$ such that
\begin{align}\label{banban}
f=\sum_{j\in\mathbb{N}}\lambda_j\alpha_j
\end{align}
in $L^2(\mathcal{X})$ and
\begin{align}\label{banban2}
\left\|\left\{\sum_{j\in\mathbb{N}}
\left[\frac{\lambda_j}{\|\mathbf{1}_{B_j}\|_{X(\mathcal{X})}}\right]^{s_0}\mathbf{1}_{B_j}
\right\}^{\frac{1}{s_0}}\right\|_X\lesssim \|f\|_{H_{X,\,L}(\mathcal{X})},
\end{align}
where the implicit positive constant is independent of $f$.

We first show that there exists a positive constant $C$  such
that, for any $(X,M)$-atom $\alpha$, associated with the ball $B:=B(x_B,r_B)$ for some $x_B\in\mathcal{X}$ and
$r_B\in(0,\infty)$, and for any $i\in\mathbb{Z}_+$,
\begin{align}\label{aibile}
\left\|g_L(\alpha)\right\|_{L^2(U_i(B))}
\leq C2^{-i\eta}[\mu(B)]^{\frac{1}{2}}\|\mathbf{1}_B\|_{X(\mathcal{X})}^{-1}
\end{align}
for some constant $\eta\in(n[\frac{1}{s_0}-\frac{1}{2}],2M)$.
Indeed, from the Tonelli theorem, \eqref{jiewei}, and \eqref{bound},  we deduce that, for any $h\in L^2(\mathcal{X})$,
\begin{align}\label{L2}
\|g_L(h)\|_{L^2(\mathcal{X})}^2\sim\|S_L(h)\|_{L^2(\mathcal{X})}^2
\lesssim\|h\|_{L^2(\mathcal{X})}^2.
\end{align}
If $i\in\{0,1\},$ then, by \eqref{L2} and \eqref{chichun-at}, we conclude that
\begin{align}\label{aibile3}
\|g_L(\alpha)\|_{L^2(U_i(B))}
\leq\|g_L(\alpha)\|_{L^2(\mathcal{X})}
\lesssim\|\alpha\|_{L^2(\mathcal{X})}
\leq[\mu(B)]^{\frac{1}{2}}\|\mathbf{1}_B\|_{X(\mathcal{X})}^{-1}.
\end{align}
Let $i\in\mathbb{N}\cap[2,\infty).$ Then, from the Tonelli theorem, it follows that
\begin{align}\label{beilun0}
\int_{U_i(B)}\left[g_L(\alpha)(x)\right]^2\,d\mu(x)
&=\int_0^{r_B}\int_{U_i(B)}\left|t^2Le^{-t^2L}(\alpha)(x)\right|^2\,d\mu(x)\,\frac{dt}{t}\\ \nonumber
&\quad+\int_{r_B}^\infty\cdots\\\notag
&=:\mathrm{I}+\mathrm{II}.
\end{align}
By Remark \ref{1515}, the estimate that $\mathrm{dist}(U_i(B),B)\sim 2^ir_B$, and \eqref{chichun-at}, we find that
\begin{align}\label{gailun1}
\mathrm{I}&\lesssim\|\alpha\|_{L^2(B)}^2\int_0^{r_B}\exp\left\{-c\frac{(2^ir_B)^2}{t^2}\right\}\,\frac{dt}{t}
\lesssim\|\alpha\|_{L^2(B)}^2\int_0^{r_B}\left(\frac{t}{2^ir_B}\right)^{2\eta}\,\frac{dt}{t}\\ \nonumber
&\lesssim2^{-2i\eta}\mu(B)\|\mathbf{1}_{B}\|_{X(\mathcal{X})}^{-2},
\end{align}
where  $c$ is a positive
constant depending only on $L$. Since $\alpha$ is an $(X,M)$-atom, it follows that there exists a $b\in \mathcal{D}(L^M)$
such that $\alpha=L^M(b)$ and \eqref{chichun-at} holds true for any $k\in\{0,\ldots,M\}.$
This, combined with Remark \ref{1515}, further implies that
\begin{align*}
\mathrm{II}
&=\int_{r_B}^\infty\int_{U_i(B)}t^{-4M}\left|(t^2L)^{M+1}e^{-t^2L}(b)(x)
\right|^2\,d\mu(x)\,\frac{dt}{t}\\
&\lesssim\|b\|_{L^2(B)}^2\int_{r_B}^\infty
t^{-4M}\exp\left\{-c\frac{(2^ir_B)^2}{t^2}\right\}\,\frac{dt}{t}\\
&\lesssim\|b\|_{L^2(B)}^2\int_{r_B}^\infty
t^{-4M}\left(\frac{t}{2^ir_B}\right)^{2\eta}\,\frac{dt}{t}\\
&\sim2^{-2i\eta}\|b\|_{L^2(B)}^2
\lesssim r_B^{4M}2^{-2i\eta}\mu(B)\left\|\mathbf{1}_B\right\|_{X(\mathcal{X})}^{-2}.
\end{align*}
By this, \eqref{gailun1}, \eqref{beilun0}, and \eqref{aibile3}, we conclude that \eqref{aibile} holds true.

Therefore, using \eqref{aibile} and applying Proposition \ref{pras} with $q=2$ and $\theta=2^i$, we obtain
\begin{align}\label{mingbai}
&\left\|\sum_{j\in\mathbb{N}}
\left[\lambda_jg_L(\alpha_j)\mathbf{1}_{U_i(B_j)}\right]^{s_0}\right\|_{X^{\frac{1}{s_0}}(\mathcal{X})}\\ \nonumber
&\quad\lesssim2^{-is_0[\eta-n(\frac{1}{s_0}-\frac{1}{2})]}
\left\|\sum_{j\in\mathbb{N}}\left[\frac{\lambda_j}
{\|\mathbf{1}_{B_j}\|_{X(\mathcal{X})}}\right]^{s_0}
\mathbf{1}_{B_j}\right\|_{X^{\frac{1}{s_0}}(\mathcal{X})}.
\end{align}
On the other hand, from \eqref{L2}, \eqref{banban}, and \eqref{jiben}, we infer that, for almost every
$x\in\mathcal{X},$
\begin{align*}
g_L(f)(x)&\leq\sum_{j\in\mathbb{N}} \lambda_jg_L(\alpha_j)(x)
=\sum_{j\in\mathbb{N}}\sum_{i\in\mathbb{Z}_+}
\lambda_jg_L(\alpha_j)(x)\mathbf{1}_{U_i(B_j)}(x)\\
&\leq\left\{\sum_{j\in\mathbb{N}}\sum_{i\in\mathbb{Z}_+}\left[\lambda_jg_L(\alpha_j)
\mathbf{1}_{U_i(B_j)}(x)\right]^{s_0}\right\}^{\frac{1}{s_0}}.
\end{align*}
By this, the assumption that $X^{\frac{1}{s_0}}$ is a $\mathrm{BBF}$ space, \eqref{mingbai}, and \eqref{banban2}, we find that
\begin{align*}
\|g_L(f)\|_{X(\mathcal{X})}^{s_0}&=\left\|[g_L(f)]^{s_0}\right\|_{X^{\frac{1}{s_0}}}\\
&\leq\left\|\sum_{i\in\mathbb{Z}_+}\sum_{j\in\mathbb{N}}
\left[\lambda_jg_L(\alpha_j)\mathbf{1}_{U_i(B_j)}\right]^{s_0}\right\|_{X^{\frac{1}{s_0}}(\mathcal{X})}\\
&\leq\sum_{i\in\mathbb{Z}_+}\left\|\sum_{j\in\mathbb{N}}
\left[\lambda_jg_L(\alpha_j)\mathbf{1}_{U_i(B_j)}\right]^{s_0}\right\|_{X^{\frac 1{s_0}}(\mathcal{X})}\\
&\lesssim\sum_{i\in\mathbb{Z}_+}2^{-is_0[\eta-(\frac{1}{s_0}-\frac{1}{2})n]}
\left\|\sum_{j\in\mathbb{N}}\left[\frac{\lambda_j}
{\|\mathbf{1}_{B_j}\|_{X(\mathcal{X})}}\right]^{s_0}
\mathbf{1}_{B_j}\right\|_{X^{\frac 1 {s_0}}(\mathcal{X})}\\
&\sim \|f\|_{H_{X,L}(\mathcal{X})}^{s_0},
\end{align*}
which, together with the fact that $H_{X,\,L}(\mathcal{X})\cap \overline{R(L)}$ is dense in $H_{X,\,L}(\mathcal{X})$, further
implies that the conclusion of the present proposition holds true. This finishes the proof of Proposition \ref{jiushijie}.
\end{proof}

To prove Theorems \ref{thm-g-2} and \ref{thm-g}, we borrow some ideas from the extrapolation theorem and its proof. We first
recall the concept of weighted Lebesgue spaces on $\mathcal{X}$. Denote the set of all balls of
$\mathcal{X}$ by $\mathbb{B}$. A locally integrable function $\omega:\ \mathcal{X}\to[0,\infty)$ is call an
$A_p(\mathcal{X})$-\emph{weight} with $p\in[1,\infty)$ if
\begin{align*}
[w]_{A_p(\mathcal{X})}:=\sup_{B\in\mathbb{B}}\left\{[\mu(B)]^{-p}\|\omega\|_{L^1(B)}
\|\omega^{-1}\|_{L^{\frac{1}{p-1}}(B)}\right\}<\infty,
\end{align*}
where $\frac{1}{p-1}:=\infty$ when $p=1.$ Moreover, let
\begin{align*}
A_\infty(\mathcal{X}):=\bigcup_{p\in[1,\infty)}A_p(\mathcal{X}).
\end{align*}
For any given $r\in(0,\infty)$ and $\omega\in A_\infty(\mathcal{X})$, the \emph{weighed Lebesgue space}
$L^r_\omega(\mathcal{X})$ is defined by setting
\begin{align*}
L^r_\omega(\mathcal{X}):=\left\{f\in\mathscr{M}(\mathcal{X}):\ \|f\|_{L^r_\omega(\mathcal{X})}<\infty\right\},
\end{align*}
where
\begin{align*}
\|f\|_{L^r_\omega(\mathcal{X})}:=\left[\int_\mathcal{X} |f(z)|^r\omega(z)\,d\mu(z)\right]^{\frac{1}{r}}.
\end{align*}
The following lemma  can be proved by a slight modification of
the proof of its Euclidean case \cite[Lemma 4.6]{dlyyz23};
we omit the details here.
\begin{lemma}\label{20232211}
Let $X(\mathcal{X})\subset\mathscr{M}(\mathcal{X})$ be a linear normed space, equipped with a norm
$\|\cdot\|_{X(\mathcal{X})}$ which makes sense for all functions in $\mathscr{M}(\mathcal{X})$.
Assume that the Hardy--Littlewood maximal operator $\mathcal{M}$ is bounded on $X(\mathcal{X})$.
For any $g\in X(\mathcal{X})$ and $x\in\mathcal{X}$, let
\begin{align*}
R_{X(\mathcal{X})}g(x):=\sum_{k=0}^{\infty}
\frac{\mathcal{M}^kg(x)}{2^k\|\mathcal{M}\|^k_{X(\mathcal{X})\to X(\mathcal{X})}},
\end{align*}
where, for any $k\in\mathbb{N},\,\mathcal{M}^k:=\mathcal{M}\circ\cdots\circ\mathcal{M}$ is the $k$ iterations
of the Hardy--Littlewood maximal operator and $\mathcal{M}^0g(x):=|g(x)|.$ Then, for any $g\in X(\mathcal{X})$
and $x\in\mathcal{X}$,
\begin{itemize}
\item[{\rm(i)}] $|g(x)|\leq R_{X(\mathcal{X})}g(x);$
\item[{\rm(ii)}] $R_{X(\mathcal{X})}g\in A_1(\mathcal{X})$ and $[R_{X(\mathcal{X})}g]_
{A_1(\mathcal{X})}\leq2\|\mathcal{M}\|_{X(\mathcal{X})\to X(\mathcal{X})},$
where $\|\mathcal{M}\|_{X(\mathcal{X})\to X(\mathcal{X})}$ denotes the operator norm of
$\mathcal{M}$ mapping $X(\mathcal{X})$ to $X(\mathcal{X})$;
\item[{\rm(iii)}] $\|R_{X(\mathcal{X})}g\|_{X(\mathcal{X})}\leq2\|g\|_{X(\mathcal{X})}.$
\end{itemize}
\end{lemma}
For any  $\lambda\in(0,\infty)$, any
measurable function $F$ on $\mathcal{X}^+$, and any $x\in\mathcal{X}$, let
\begin{align*}
S(F)(x):=\left[\int_0^\infty\int_{B(x,t)}|F(y,t)|^2\,\frac{d\mu(y)\,dt}{V(x,t)t}\right]^{\frac{1}{2}}
\end{align*}
and
\begin{align*}
g_{\lambda}^\ast(F)(x):=\left\{\int_0^\infty\int_{\mathcal{X}}\left[\frac{t}{t+d(x,y)}\right]^\lambda
|F(y,t)|^2\,\frac{d\mu(y)\,dt}{V(x,t)t}\right\}^{\frac{1}{2}}.
\end{align*}
The following lemma is just \cite[Lemma 3.2(i)]{gy14}.
\begin{lemma}\label{202322111}
Let $r\in(0,2)$, $p\in[1,\infty)$, $\omega\in A_{p}(\mathcal{X})$, and $\lambda\in(\frac{2pn}{r},\infty)$.
Then there exists a positive constant $C$ such that, for any measurable function $F$ on $\mathcal{X}^+$
with $\|S(F)\|_{L^r_\omega(\mathcal{X})}<\infty,$
\begin{align*}
\left\|g_{\lambda}^\ast(F)\right\|_{L^r_\omega(\mathcal{X})}
\leq C\|S(F)\|_{L^r_\omega(\mathcal{X})}.
\end{align*}
\end{lemma}
With the help of Lemma \ref{20232211}, we generalize Lemma \ref{202322111}
with $L^r_\omega(\mathcal{X})$ replaced by a $\mathrm{BQBF}$ space $X(\mathcal{X}).$
\begin{proposition}\label{kandejian}
Let $X(\mathcal{X})$ be a $\mathrm{BQBF}$ space, $s_0\in(0,2)$, and $\lambda\in(\frac{2n}{s_0},\infty)$.
Assume that $X^{1/s_0}(\mathcal{X})$ is a $\mathrm{BBF}$ space and $\mathcal{M}$ is bounded on
$(X^{1/s_0}(\mathcal{X}))'.$ Then there exists a positive constant $C$ such that, for any
measurable function $F$ on $\mathcal{X}^+$ with $\|S(F)\|_{X(\mathcal{X})}<\infty,$
\begin{align*}
\left\|g_{\lambda}^\ast(F)\right\|_{X(\mathcal{X})}\leq C\|S(F)\|_{X(\mathcal{X})}.
\end{align*}
\end{proposition}
\begin{proof}
Let $Y(\mathcal{X}):=X^{1/s_0}(\mathcal{X}).$ Then, by Lemmas \ref{2023221}, \ref{20232211}, and
\ref{202322111}, we find that, for any measurable function $F$ on $\mathcal{X}^+$ with $\|S(F)\|_{X(\mathcal{X})}<\infty,$
\begin{align*}
\left\|g_{\lambda}^\ast(F)\right\|_{X(\mathcal{X})}^{s_0}
&=\left\|\left[g_{\lambda}^\ast(F)\right]^{s_0}\right\|_{Y(\mathcal{X})}
=\left\|\left[g_{\lambda}^\ast(F)\right]^{s_0}\right\|_{Y''(\mathcal{X})}\\
&=\sup_{\|h\|_{Y'(\mathcal{X})}=1}\int_{\mathcal{X}}\left[g_{\lambda}^\ast(F)(x)\right]^{s_0}h(x)\,d\mu(x)\\
&\leq\sup_{\|h\|_{Y'(\mathcal{X})}=1}\int_{\mathcal{X}}\left[g_{\lambda}^\ast(F)(x)\right]^{s_0}R_{Y'(\mathcal{X})}h(x)\,d\mu(x)\\
&\lesssim\sup_{\|h\|_{Y'(\mathcal{X})}=1}\int_{\mathcal{X}}\left[S(F)(x)\right]^{s_0}R_{Y'(\mathcal{X})}h(x)\,d\mu(x)\\
&\lesssim\sup_{\|h\|_{Y'(\mathcal{X})}=1}\|[S(F)]^{s_0}\|_{Y(\mathcal{X})} \|R_{Y'(\mathcal{X})}h\|_{Y'(\mathcal{X})}\\
&\lesssim\sup_{\|h\|_{Y'(\mathcal{X})}=1}\left\|[S(F)]^{s_0}\right\|_{Y(\mathcal{X})} \|h\|_{Y'(\mathcal{X})}
\lesssim\left\|[S(F)]^{s_0}\right\|_{Y(\mathcal{X})}\\
&=\|S(F)\|_{X(\mathcal{X})}^{s_0}.
\end{align*}
This finishes the proof of Proposition \ref{kandejian}.
\end{proof}
Next, we prove Theorem \ref{thm-g-2} via using Propositions \ref{jiushijie} and \ref{kandejian}.
\begin{proof}[Proof of Theorem \ref{thm-g-2}]
By Proposition \ref{jiushijie}, we immediately find that $g_{L}$ is bounded from $H_{X,\,L}(\mathcal{X})$
to $X(\mathcal{X})$. Meanwhile, from the assumption that $\mathcal{M}$ is bounded on the $\frac{1}
{(q_0/s_0)'}$-convexification of the associate space $(X^{1/s_0})'(\mathcal{X})$ and the fact that
$(q_0/s_0)'>1$, we deduce that $\mathcal{M}$ is bounded on
$(X^{1/s_0})'(\mathcal{X})$. Thus, all the assumptions of Proposition \ref{kandejian} are satisfied. Therefore,
by Proposition \ref{kandejian}, we then conclude that $g_{\lambda,\,L}^\ast$ is bounded from
$H_{X,\,L}(\mathcal{X})$ to $X(\mathcal{X})$. This finishes the proof of Theorem \ref{thm-g-2}.
\end{proof}

Recall that, for any $t,\lambda\in(0,\infty)$ and $\varphi\in\mathcal{S}(\mathbb{R})$, the \emph{Peetre type maximal
operator} $\varphi^\ast_\lambda(tL)$ is defined by setting, for any $f\in L^2(\mathcal{X})$ and $x\in\mathcal{X}$,
\begin{align*}
\varphi^\ast_\lambda(tL)(f)(x):=\sup_{y\in\mathcal{X}}\frac{|\varphi(tL)f(y)|}{[1+t^{-1/2}d(x,y)]^\lambda}.
\end{align*}
The following lemma is just  \cite[Lemma 3.4]{h17}.

\begin{lemma}\label{point}
Let $L$ be a non-negative self-adjoint operator on $L^2(\mathcal{X})$ satisfying the Gaussian upper bound
estimates \eqref{gauss}. For any $\xi\in\mathbb{R}$, let $\varphi(\xi):=\xi e^{-\xi}.$ Then, for any given
$u,r\in(0,\infty)$ and $\lambda\in(\frac{n}{2},\infty)$, there exists a positive constant $C$ such that,
for any $\ell\in\mathbb{Z}$, $t\in[1,2]$,  $f\in L^2(\mathcal{X})$, and $x\in\mathcal{X}$,
\begin{align*}
&\left[\varphi^\ast_\lambda\left(2^{-2\ell}t^2L\right)(f)(x)\right]^r\\ \notag
&\quad\leq C\sum_{j=\ell}^\infty 2^{-(j-\ell)ur}\int_{\mathcal{X}}\frac{|\varphi(2^{-2j}t^2L)(f)(z)|^r}
{V(z,2^{-\ell})[1+2^{\ell}d(x,z)]^{\lambda r}}\,d\mu(z).	
\end{align*}
\end{lemma}
Now, we prove Theorem \ref{thm-g} by using Proposition \ref{kandejian} and Lemma \ref{point}.

\begin{proof}[Proof of Theorem \ref{thm-g}]
By Proposition \ref{kandejian}, we have
\begin{align}\label{yifang1}
\left[H_{X,\,L}(\mathcal{X})\cap L^2(\mathcal{X})\right]\subset
\left[H_{X,\,L,\,g_{\lambda}^\ast}(\mathcal{X})\cap L^2(\mathcal{X})\right]
\end{align}
and, for any $f\in H_{X,\,L}(\mathcal{X})\cap L^2(\mathcal{X}),$
\begin{align*}
\|f\|_{H_{X,\,L,\,g_{\lambda}^\ast}(\mathcal{X})}\lesssim\|f\|_{H_{X,\,L}(\mathcal{X})}.
\end{align*}
Next, we show that
\begin{align}\label{yifang3}
\left[H_{X,\,L,\,g_{\lambda}^\ast}(\mathcal{X})\cap L^2(\mathcal{X})\right]
\subset\left[H_{X,\,L,\,g}(\mathcal{X})\cap L^2(\mathcal{X})\right]
\end{align}
and, for any $f\in H_{X,\,L,\,g_{\lambda}^\ast}(\mathcal{X})\cap L^2(\mathcal{X}),$
\begin{align*}
\|f\|_{H_{X,\,L,\,g}(\mathcal{X})}\lesssim\|f\|_{H_{X,\,L,\,g_{\lambda}^\ast}(\mathcal{X})}.
\end{align*}
Let $f\in H_{X,\,L,\,g_{\lambda}^\ast}(\mathcal{X})\cap L^2(\mathcal{X})$
and $\varphi$ be the same as in Lemma \ref{point}.
From Lemma \ref{point}, we deduce
that, for any $\ell\in\mathbb{Z}$ and $x\in\mathcal{X},$
\begin{align*}
&\int_1^2\left|\varphi_{(\lambda+n)/2}^\ast\left(2^{-2\ell}t^2L\right)(f)(x)\right|^2\,\frac{dt}{t}\\ \notag
&\quad\lesssim\sum_{j=\ell}^\infty 2^{-2(j-\ell)u}\int_1^2\int_{\mathcal{X}}\frac{|\varphi(2^{-2j}t^2L)(f)(z)|^2}
{V(z,2^{-\ell})[1+2^{\ell}d(x,z)]^{\lambda+n}}\,\frac{d\mu(z)\,dt}{t}\\ \notag
&\quad\lesssim\sum_{j=\ell}^\infty 2^{-2(j-\ell)u}\int_1^2\int_{\mathcal{X}}\frac{|\varphi(2^{-2j}t^2L)(f)(z)|^2}
{V(z,2^{-j}t)[1+2^{\ell}t^{-1}d(x,z)]^{\lambda+n}}\,\frac{d\mu(z)\,dt}{t}\\ \notag
&\quad\lesssim\sum_{j\in\mathbb{Z}}2^{-|j-\ell|(2u-\lambda-n)}
\int_1^2\int_{\mathcal{X}}\frac{|\varphi(2^{-2j}t^2L)
(f)(z)|^2}{V(x,2^{-j}t)[1+2^{j}t^{-1}d(x,z)]^{\lambda}}\,\frac{d\mu(z)\,dt}{t}
\end{align*}
and hence
\begin{align*}
&\int_0^\infty\left|\varphi_{(\lambda+n)/2}^\ast\left(t^2L\right)(f)(x)\right|^2\,\frac{dt}{t}\\ \notag
&\quad=\sum_{\ell\in\mathbb{Z}}\int_1^2\left|
\varphi_{(\lambda+n)/2}^\ast\left(2^{-2\ell}t^2L\right)(f)(x)\right|^2\,\frac{dt}{t}\\ \notag
&\quad\lesssim\sum_{\ell\in\mathbb{Z}}\sum_{j\in\mathbb{Z}}2^{-|j-\ell|(2u-\lambda-n)}
\int_1^2\int_{\mathcal{X}}\frac{|\varphi(2^{-2j}t^2L)(f)(z)|^2}
{V(x,2^{-j}t)[1+2^{j}t^{-1}d(x,z)]^{\lambda}}\,
\frac{d\mu(z)\,dt}{t}\\ \notag
&\quad\sim\left[g_{\lambda,\,L}^\ast(f)(x)\right]^2,
\end{align*}
which, together with the definition of $\varphi^\ast_{\lambda}(f)$, further implies that
\begin{align*}
\|f\|_{H_{X,\,L,\,g}(\mathcal{X})}&=\left\|\left[\int_0^\infty
\left|\varphi\left(t^2L\right)(f)\right|^2\,\frac{dt}{t}\right]^{\frac{1}{2}}\right\|_{X(\mathcal{X})}\\
&\leq\left\|\left[\int_0^\infty\left|\varphi^\ast_{(\lambda+n)/2}\left(t^2L\right)(f)\right|^2\,
\frac{dt}{t}\right]^{\frac{1}{2}}\right\|_{X(\mathcal{X})}\\
&\lesssim\left\|g_{\lambda,\,L}^\ast(f)\right\|_{X(\mathcal{X})}
\sim\|f\|_{H_{X,\,L,\,g^\ast_\lambda}(\mathcal{X})}.
\end{align*}
This finishes the proof of \eqref{yifang3}.

Finally, we prove that
\begin{align}\label{yifang2}
\left[H_{X,\,L,\,g}(\mathcal{X})\cap L^2(\mathcal{X})\right]
\subset\left[H_{X,\,L}(\mathcal{X})\cap L^2(\mathcal{X})\right]
\end{align}
and, for any $f\in H_{X,\,L,\,g}(\mathcal{X})\cap L^2(\mathcal{X}),$
\begin{align*}
\|f\|_{H_{X,\,L}(\mathcal{X})}\lesssim\|f\|_{H_{X,\,L,\,g}(\mathcal{X})}.
\end{align*}

Let $f\in H_{X,\,L,\,g}(\mathcal{X})\cap L^2(\mathcal{X}).$  It is obvious that, for any $t\in(0,\infty)$ and $x,y\in\mathcal{X}$ satisfying $d(x,y)<t$,
\begin{align*}
\left|\varphi\left(t^2L\right)f(y)\right|\leq2^\lambda\varphi^\ast_\lambda\left(t^2L\right)(f)(x).
\end{align*}
This, combined with the Tonelli theorem,  further implies that
\begin{align}\label{liu1}
\|f\|_{H_{X,\,L}(\mathcal{X})}&\sim\left\|\left[\int_0^\infty\int_{B(\cdot,t)}\left|\varphi\left(t^2L\right)
f(y)\right|^2\,\frac{d\mu(y)\,dt}{V(\cdot,t)t}\right]^{\frac{1}{2}}\right\|_{X(\mathcal{X})}\\ \notag
&\lesssim\left\|\left[\int_0^\infty\left|\varphi_{\lambda}^\ast\left(t^2L\right)f\right|^2\,\frac{dt}{t}
\right]^{\frac{1}{2}}\right\|_{X(\mathcal{X})}\\\notag
&=\left\|\left[\sum_{j\in\mathbb{Z}}\int_0^1\left|\varphi_{\lambda}^\ast\left(2^{-2j}t^2L\right)f\right|^2\,
\frac{dt}{t}\right]^{\frac{1}{2}}\right\|_{X(\mathcal{X})}.
\end{align}
Since $\lambda>\frac{2n}{s_0}>\frac{2n}{\min\{p,2\}},$ we can take an
$r\in(\frac{2n}{\lambda},\min\{p,2\})$. For any $x\in\mathcal{X}$, let
\begin{align*}
b_{j}(x):=\int_1^2\left|\varphi\left(2^{-2j}t^2L\right)(f)(x)\right|^2\,\frac{dt}{t}.
\end{align*}
Then, by Lemma \ref{point}, the H\"older inequality, and the Minkowski integral inequality, we find that
\begin{align}\label{liu2}
&\left[\int_1^2\left|\varphi_{\lambda}^\ast\left(2^{-2\ell}t^2L\right)(f)(x)\right|^2\,
\frac{dt}{t}\right]^{\frac{r}{2}}\\ \notag
&\quad\lesssim\sum_{j=\ell}^\infty 2^{-(j-\ell)ur}\left\{\int_1^2\left[\int_{\mathcal{X}}\frac{|\varphi(2^{-2j}t^2L)
(f)(z)|^r}{V(z,2^{-\ell})[1+2^{\ell}d(x,z)]^{\lambda r}}\,d\mu(z)\right]^{\frac{2}{r}}\,\frac{dt}{t}\right\}^{\frac{r}{2}}\\ \notag
&\quad\lesssim\sum_{j=\ell}^\infty 2^{-(j-\ell)ur}\int_{\mathcal{X}}\frac{[b_j(z)]^{\frac{r}{2}}}
{V(z,2^{-\ell})[1+2^{\ell}d(x,z)]^{\lambda r}}\,d\mu(z).
\end{align}
Moreover, it is well known that, for any given $\alpha\in(2n,\infty)$, there exists a positive constant $C$
such that, for any $\mu$-measurable function $g$ on $\mathcal{X}$, any $t\in(0,\infty)$, and any $x\in\mathcal{X},$
\begin{align*}
\int_{\mathcal{X}}\frac{|g(y)|}{V(y,t)[1+t^{-1}d(x,y)]^\alpha}\,d\mu(y)\leq C\mathcal{M}(g)(x)
\end{align*}
(see, for instance, \cite[Lemma 2.1]{h17}), which, together with \eqref{liu2} and the H\"older inequality, further
implies that
\begin{align*}
&\int_1^2\left|\varphi_{\lambda}^\ast\left(2^{-2\ell}t^2L\right)(f)(x)\right|^2\,\frac{dt}{t}\\ \notag
&\quad\lesssim\left\{\sum_{j=\ell}^\infty 2^{-(j-\ell)ur}\mathcal{M}
\left([b_j]^{\frac{r}{2}}\right)(x)\right\}^{\frac{2}{r}}
\lesssim\sum_{j\in\mathbb{Z}}2^{-|j-\ell|u}\left\{\mathcal{M}
\left([b_j]^{\frac{r}{2}}\right)(x)\right\}^{\frac{2}{r}}.
\end{align*}
From this, \eqref{liu1}, and Assumption \ref{vector1}, we deduce that
\begin{align*}
\|f\|_{H_{X,\,L}(\mathcal{X})}&\lesssim\left\|\left\{\sum_{\ell\in\mathbb{Z}}\int_1^2\left|
\varphi_{\lambda}^\ast\left(2^{-2\ell}t^2L\right)(f)\right|^2\,\frac{dt}{t}\right\}^{\frac{1}{2}}
\right\|_{X(\mathcal{X})}\\
&\lesssim\left\|\left(\sum_{j\in\mathbb{Z}}\left\{\mathcal{M}\left([b_j]^{\frac{r}{2}}\right)
\right\}^{\frac{2}{r}}\right)^{\frac{1}{2}}\right\|_{X(\mathcal{X})}
\lesssim\left\|\left(\sum_{j\in\mathbb{Z}}b_j\right)^{\frac{1}{2}}\right\|_{X(\mathcal{X})}\\ \notag
&\sim\|g_{L}(f)\|_{X(\mathcal{X})}.
\end{align*}
This proves \eqref{yifang2}.

Finally, by \eqref{yifang1}, \eqref{yifang3}, and \eqref{yifang2}, we conclude that the spaces $H_{X,\,L}(\mathcal{X})
\cap L^2(\mathcal{X}), H_{X,\,L,\,g}(\mathcal{X})\cap L^2(\mathcal{X}),$ and $H_{X,\,L,\,g_{\lambda}^\ast}(\mathcal{X})
\cap L^2(\mathcal{X})$ coincide with equivalent quasi-norms. From this and a density argument, it follows that
the spaces $H_{X,\,L}(\mathcal{X}), H_{X,\,L,\,g}(\mathcal{X}),$ and
$H_{X,\,L,\,g_{\lambda}^\ast}(\mathcal{X})$  coincide with equivalent quasi-norms.
This finishes the proof of Theorem \ref{thm-g}.
\end{proof}

\subsection{Boundedness of Schr\"odinger Groups}\label{sec4.2}
In this subsection, we prove the boundedness of Schr\"odinger groups generated by $L$
on the Hardy type space $H_{X,\,L}(\mathcal{X})$.

Let $L$ be a non-negative self-adjoint operator on $L^2(\mathcal{X})$. By the spectral theorem (see, for
instance, \cite{m86}), we conclude that the Schr\"odinger group $\{e^{itL}\}_{t\in\mathbb{R}}$ is defined by
setting, for any $t\in\mathbb{R}$,
\begin{align*}
e^{itL}:=\int_0^\infty e^{it\lambda}\,dE_L(\lambda),
\end{align*}
where $E_L(\lambda)$ denotes the projection-valued measure determined by $L$. The following conclusion is
the main result of this subsection.

\begin{theorem}\label{sch}
Let $L$ be a non-negative self-adjoint operator on $L^2(\mathcal{X})$ satisfying the Gaussian upper bound
estimate \eqref{gauss}. Assume that $X(\mathcal{X})$ is a $\mathrm{BQBF}$ space satisfying both Assumptions \ref{vector1}
and \ref{vector2} for some $p\in(0,\infty)$, $s_0\in(0,\min\{p,1\}]$, and $q_0\in(s_0,2]$. Let
$s\in (n[\frac{1}{s_0}-\frac{1}{2}],\infty)$. Then there exists a positive constant $C$ such that, for any
 $t\in\mathbb{R}$ and $f\in H_{X,\,L}(\mathcal{X})$,
\begin{align}\label{190}
\left\|(I+L)^{-s}e^{itL}(f)\right\|_{H_{X,\,L}(\mathcal{X})}\leq C(1+|t|)^{n(\frac{1}{s_0}-\frac{1}{2})}\|f\|_{H_{X,\,L}(\mathcal{X})}.
\end{align}
\end{theorem}
Applying Theorem  \ref{sch} with $X(\mathcal{X}):=L^r(\mathcal{X})$, we have the
following conclusion; since their proofs are similar to that of Theorem \ref{ls10}, we omit the details here.
\begin{theorem}\label{ls30}
Let  $L$ be a non-negative self-adjoint operator on $L^2(\mathcal{X})$ satisfying the Gaussian upper bound
estimate \eqref{gauss}, $r\in(0,2)$, and $s\in(\frac{n}{\min\{1,r\}}-\frac{n}{2},\infty)$.
Then there exists a positive constant $C$ such that, for any
 $t\in\mathbb{R}$ and $f\in H^{r}_{L}(\mathcal{X})$,
\begin{align*}
\left\|(I+L)^{-s}e^{itL}f\right\|_{H^{r}_{L}(\mathcal{X})}\leq C(1+|t|)^{\frac{n}{\min\{1,r\}}-\frac{n}{2}}\|f\|_{H^{r}_{L}(\mathcal{X})}.
\end{align*}
\end{theorem}

\begin{remark} We point out that
\cite[Theorem 1.1]{bl22} shows that 
Theorem \ref{ls30} with $r\in (0,1]$ and 
$s=n(\frac{1}{r}-\frac{1}{2})$ also holds true.
\end{remark}

To prove Theorem \ref{sch}, we need several lemmas. Let $\psi\in \mathcal{S}(\mathbb{R})$ be such that
$\mathrm{supp}\,(\psi)\subset [1/2,2]$, $\int_{\mathbb{R}}\psi(t)\,\frac{dt}{t}\not=0,$ and
\begin{align*}
\sum_{j\in\mathbb{Z}}\psi\left(2^{-j}t\right)=1
\end{align*}
for any $t\in(0,\infty)$. For any  $f\in L^2(\mathcal{X})$, the discrete square function
$S_{L,\,\psi}(f)$ is defined by setting, for any $x\in\mathcal{X},$
\begin{align*}
S_{L,\,\psi}(f)(x):=\left[\sum_{j\in\mathbb{Z}}\left|\psi\left(2^{-j}L\right)(f)(x)\right|^2\right]^{1/2}.
\end{align*}
Motivated by \cite[Theorem 2.5]{bl22}, we have the following lemma.

\begin{lemma}\label{pingfang}
Let $L$ be a non-negative self-adjoint operator on $L^2(\mathcal{X})$ satisfying the Gaussian upper bound
estimate \eqref{gauss}. Assume that $X(\mathcal{X})$ is a $\mathrm{BQBF}$ space satisfying Assumption \ref{vector1}
for some $p\in(0,\infty)$. Then there exists a positive constant $C$ such that, for any $f\in L^2(\mathcal{X})$
with $\|S_{L,\,\psi}(f)\|_{X(\mathcal{X})}<\infty$,
\begin{align*}
\|f\|_{H_{X,\,L}(\mathcal{X})}\leq C\left\|S_{L,\,\psi}(f)\right\|_{X(\mathcal{X})}.
\end{align*}
\end{lemma}
\begin{proof}
Let  $\lambda\in(\frac{n}{p},
\infty)$, $r\in(0,\min\{1,p\}),$  $\varphi(\xi):=\xi e^{-\xi}$ for any $\xi\in\mathbb{R},$ and $f\in L^2(\mathcal{X})$ with $\|S_{L,\,\psi}(f)\|_{X(\mathcal{X})}<\infty$.
By \cite[p.\, 8 and p.\, 11]{bl22}, we have, for any $j\in\mathbb{Z}$ and $x\in\mathcal{X},$
\begin{align*}
\psi^\ast_{\lambda}\left(2^{-j}L\right)f(x)\lesssim\left[\mathcal{M}\left(\left|\psi\left(2^{-j}L\right)
f\right|^r\right)(x)\right]^{\frac{1}{r}}
\end{align*}
and
\begin{align*}
\int_0^\infty\left|\varphi_{\lambda}^\ast(tL)f(x)\right|^2\,\frac{dt}{t}
\lesssim\sum_{j\in\mathbb{Z}}\left[\psi^\ast_{\lambda}\left(2^{-j}L\right)f(x)\right]^2.
\end{align*}
Using this, \eqref{liu1}, and Assumption \ref{vector1}, we find that
\begin{align*}
\|f\|_{H_{X,\,L}(\mathcal{X})}&\lesssim\left\|\left[\int_0^\infty\left|\varphi_{\lambda}^\ast
(tL)f\right|^2\,\frac{dt}{t}\right]^{\frac{1}{2}}\right\|_{X(\mathcal{X})}\\
&\lesssim\left\|\left\{\sum_{j\in\mathbb{N}}\left[\mathcal{M}\left(\left|\psi\left(2^{-j}L\right)f\right|^r
\right)\right]^{\frac{2}{r}}\right\}^{\frac{1}{2}}\right\|_{X(\mathcal{X})}\\
&\lesssim \left\|\left[\sum_{j\in\mathbb{N}}\left|\psi\left(2^{-j}L\right)f\right|^2
\right]^{\frac{1}{2}}\right\|_{X(\mathcal{X})}\sim\left\|S_{L,\psi}(f)\right\|_{X(\mathcal{X})}.
\end{align*}
This finishes the proof of Lemma \ref{pingfang}.
\end{proof}

Now, we prove Theorem \ref{sch} by using Lemmas \ref{pingfang} and \ref{yi}.

\begin{proof}[Proof of Theorem \ref{sch}]
Let  $M\in(\frac{s}{2},\infty)\cap\mathbb{N}$ and
$F(\lambda):=(1+\lambda)^{-s}e^{it\lambda}$ for any $\lambda\in(0,\infty)$. 
By Lemma \ref{pingfang}, we conclude that, for any $t\in\mathbb{R}$ and $f\in H_{X,\,L}(\mathcal{X})\cap L^2(\mathcal{X})$,
\begin{align*}
\left\|(I+L)^{-s}e^{itL}(f)\right\|_{H_{X,\,L}(\mathcal{X})}\lesssim
\left\|S_{L,\,\psi}\left(F(L)f\right)\right\|_{X(\mathcal{X})}.
\end{align*}
Thus, to show \eqref{190}, it suffices to prove that, for any $t\in\mathbb{R}$ and $f\in H_{X,\,L}(\mathcal{X})\cap L^2(\mathcal{X})$,
\begin{align}\label{2127}
\left\|S_{L,\,\psi}(F(L)f)\right\|_{X(\mathcal{X})}
\lesssim(1+|t|)^{n(\frac{1}{s_0}-\frac{1}{2})}\|f\|_{H_{X,\,L}(\mathcal{X})}.
\end{align}

Let $f\in H_{X,\,L}(\mathcal{X})\cap L^2(\mathcal{X})$.
From Theorem \ref{thm-mc}, we deduce that there exists a sequence $\{\lambda_j\}_{j\in\mathbb{N}}\subset[0,\infty)$
and a sequence $\{\alpha_j\}_{j\in\mathbb{N}}$ of $(X,M)$-atoms associated, respectively, with the balls
$\{B_j\}_{j\in\mathbb{N}}\subset \mathcal{X}$ such that
\begin{align}\label{can0}
f=\sum_{j\in\mathbb{N}}\lambda_j\alpha_j
\end{align}
in $L^2(\mathcal{X})$ and
\begin{align}\label{can1}
\left\|\left\{\sum_{j=1}^\infty
\left(\frac{\lambda_j}{\|\mathbf{1}_{B_j}\|_{X(\mathcal{X})}}\right)^{s_0}
\mathbf{1}_{B_j}\right\}^{\frac{1}{s_0}}\right\|_{X(\mathcal{X})}\lesssim \|f\|_{H_{X,\,L}(\mathcal{X})},
\end{align}
where the implicit positive constant is independent of $f$. By \eqref{can0}, we have, for almost every
$x\in\mathcal{X},$
\begin{align}\label{2008}
S_{L,\,\psi}\left(F(L)f\right)(x)\leq\sum_{j\in\mathbb{N}}\lambda_{j}S_{L,\,\psi}\left(F(L)\alpha_j\right)(x).
\end{align}
Moreover, it is easy to find that, for any $j\in\mathbb{N}$,
\begin{align*}
I=\left(I-e^{-r_{B_j}^2L}\right)^M+\sum_{k=1}^M(-1)^{k+1}C_M^ke^{-kr_{B_j}^2L}
=:\left(I-e^{-r_{B_j}^2L}\right)^M+P\left(r_{B_j}^2L\right),
\end{align*}
where $I$ denotes the identity operator on $L^2(\mathcal{X})$ and, for any $k\in\{0,\ldots,M\}$,
$C_M^k:=\frac{M!}{k!(M-k)!}$. Thus, for any $j\in\mathbb{N}$ and
$x\in\mathcal{X},$
\begin{align*}
S_{L,\,\psi}\left(F(L)\alpha_j\right)(x)&\leq
S_{L,\,\psi}\left(\left[I-e^{-r_{B_j}^2L}\right]^MF(L)\alpha_j\right)(x)\\
&\quad+ S_{L,\,\psi}\left(P\left(r_{B_j}^2L\right)F(L)\alpha_j\right)(x).
\end{align*}
From this, \eqref{2008}, \eqref{jiben}, and the assumption that $X^{\frac{1}{s_0}}(\mathcal{X})$ is a
$\mathrm{BBF}$ space, it follows that
\begin{align}\label{2128}
&\left\|S_{L,\,\psi}(F(L)f)\right\|_{X(\mathcal{X})}^{s_0}\\ \notag
&\quad\leq\sum_{i\in\mathbb{Z}_+}\left\|\sum_{j\in\mathbb{N}}\left[\lambda_{j}S_{L,\,\psi}
\left(\left[I-e^{-r_{B_j}^2L}\right]^MF(L)\alpha_{j}\right)\mathbf{1}_{S_i(B_{t,j})}\right]^{s_0}
\right\|_{X^{\frac{1}{s_0}}(\mathcal{X})}\\ \notag
&\quad\quad+\sum_{i\in\mathbb{Z}_+}\left\|\sum_{j\in\mathbb{N}}\left[\lambda_{j}S_{L,\,\psi}
\left(P\left(r_{B_j}^2L\right)F(L)\alpha_{j}\right)\mathbf{1}_{S_i(B_{t,j})}\right]^{s_0}
\right\|_{X^{\frac{1}{s_0}}(\mathcal{X})}\\ \notag
&\quad=:\sum_{i\in\mathbb{Z}_+}E_{i,1}+\sum_{i\in\mathbb{Z}_+}E_{i,2},
\end{align}
where $B_{t,j}:=(1+|t|)B_j$ for any $j\in\mathbb{N}.$

Next, we show that, for any $(X,M)$-atom $\alpha$, associated with the ball $B:=B(x_B,r_B)$
for some $x_B\in\mathcal{X}$ and $r_B\in(0,\infty),$ and any $i\in\mathbb{Z}_+$,
\begin{align}\label{1959}
\left\|S_{L,\,\psi}\left(\left[I-e^{-r_{B}^2L}\right]^MF(L)\alpha\right )\right\|_{L^2(U_i(B_{t}))}\lesssim
2^{-is}\left[\mu(B)\right]^{\frac{1}{2}}\|\mathbf{1}_{B}\|_{X(\mathcal{X})}^{-1}
\end{align}
and
\begin{align}\label{2114}
\left\|S_{L,\,\psi}\left(P\left(r_{B}^2L\right)F(L)\alpha\right)\right\|_{L^q(U_i(B_{t}))}
\lesssim2^{-is}
\left[\mu(B)\right]^{\frac{1}{2}}\|\mathbf{1}_{B}\|_{X(\mathcal{X})}^{-1},
\end{align}
where $B_t:=(1+|t|)B$.

We first prove \eqref{1959}. If $i\in\{0,1\}$, by the functional calculi associated with $L$ and by \eqref{chichun-at}, we have
\begin{align*}
\left\|S_{L,\,\psi}\left(\left[I-e^{-r_{B}^2L}\right]^MF(L)\alpha\right)\right\|_{L^2(U_i(B_{t}))}
\lesssim\|\alpha\|_{L^2(\mathcal{X})}
\leq \left[\mu(B)\right]^{\frac{1}{2}}\|\mathbf{1}_{B}\|_{X(\mathcal{X})}^{-1}.
\end{align*}
Let $i\in\mathbb{N}\cap[2,\infty).$ For any $\ell\in\mathbb{Z}$ and $\lambda\in(0,\infty),$ let
\begin{align*}
F_{\ell,r_B}(\lambda):=\psi\left(2^{-\ell}\lambda\right)\left(1-e^{-r_B^2\lambda}\right)^MF(\lambda).
\end{align*}
Using \eqref{jiben} with $s:=2$, we conclude that
\begin{align}\label{1961}
&\left\|S_{L,\,\psi}\left(\left[I-e^{-r_{B}^2L}\right]^MF(L)\alpha\right)\right\|_{L^2(U_i(B_{t}))}\\ \notag
&\quad\leq\sum_{\ell\in\mathbb{Z}}\left\|F_{\ell,r_B}(L)\alpha\right\|_{L^2(U_i(B_{t}))}
=\sum_{\ell<0}\cdots+\sum_{0\leq \ell<i}\cdots+\sum_{\ell\geq i}\cdots\\\notag
&\quad=:\mathrm{I}_1+\mathrm{I}_2+\mathrm{I}_3.
\end{align}
Let $h:=s+\theta$ for some $\theta \in(0,s)$. By \cite[(24) and (25)]{bl22}, we find that, for any
$\ell\in\mathbb{Z}$ and $g\in L^2(B)$,
\begin{align*}
&\left\|F_{\ell,r_B}(L)g\right\|_{L^2(U_i(B_t))}\\ \notag
&\quad\lesssim2^{-hi}\left(2^\ell r_B^2\right)^{-\frac{h}{2}}\max\left\{1,2^{\theta \ell}\right\}
\min\left\{1,\left(2^\ell r_B^2\right)^M\right\}\|g\|_{L^2(B)}.
\end{align*}
Using this, \eqref{chichun-at}, and Lemma \ref{yi}, we obtain
\begin{align}\label{1962}
\mathrm{I}_1
&=\sum_{\ell<0}\left\|F_{\ell,r_B}(L)\alpha\right\|_{L^2(U_i(B_{t}))}\\\notag
&\lesssim
2^{-hi}\sum_{\ell<0}\left(2^\ell r_B^2\right)^{-\frac{h}{2}}
\min\left\{1,\left(2^\ell r_B^2\right)^M\right\}\|\alpha\|_{L^2(B)}
\lesssim
2^{-hi}[\mu(B)]^{\frac{1}{2}}\|\mathbf{1}_{B}\|_{X(\mathcal{X})}^{-1}
\end{align}
and
\begin{align}\label{1964}
\mathrm{I}_{2}
&=\sum_{0\leq \ell<i}\left\|F_{\ell,r_B}(L)\alpha\right\|_{L^2(U_i(B_{t}))}\\\notag
&\lesssim2^{-h i}
\sum_{0\leq \ell<i}2^{\theta \ell}\left(2^\ell r_B^2\right)^{-\frac{h}{2}}
\min\left\{1,\left(2^\ell r_B^2\right)^M\right\}\|\alpha\|_{L^2(B)}\\ \notag
&\leq2^{-h i+\theta i}
\sum_{\ell\in\mathbb{Z}}\left(2^\ell r_B^2\right)^{-\frac{h}{2}}
\min\left\{1,\left(2^\ell r_B^2\right)^M\right\}\|\alpha\|_{L^2(B)}\\ \notag
&\lesssim2^{-si}
[\mu(B)]^{\frac{1}{2}}\|\mathbf{1}_{B}\|_{X(\mathcal{X})}^{-1}.
\end{align}
Furthermore, observe that, for any $\ell\in\mathbb{Z},$
\begin{align*}
\left\|F_{\ell,r_B}\right\|_{L^\infty(0,\infty)}\lesssim2^{-\ell s}
\min\left\{1,\left(2^\ell r_B^2\right)^{M}\right\}.
\end{align*}
From this, the functional calculi associated with $L$, and  \eqref{chichun-at},  it follows that
\begin{align*}
\mathrm{I}_{3}
&=\sum_{ \ell\geq i}\left\|F_{\ell,r_B}(L)\alpha\right\|_{L^2(U_i(B_{t}))}
\lesssim
\sum_{\ell\geq i}2^{-\ell s}\min\left\{1,\left(2^\ell r_B^2\right)^{M}\right\}
\|\alpha\|_{L^2(B)}\\
&\lesssim2^{-si}\left[\mu(B)\right]^{\frac{1}{2}}
\|\mathbf{1}_{B}\|_{X(\mathcal{X})}^{-1}.
\end{align*}
By this, \eqref{1964}, \eqref{1962}, and \eqref{1961}, we find that \eqref{1959} holds true.

For any $\ell\in\mathbb{Z}$ and $\lambda\in(0,\infty)$, let
\begin{align*}
G_{\ell,r_B}(\lambda):=\psi\left(2^{-\ell}\lambda\right)\left(r_B^2\lambda\right)^M
P\left(r_B^2\lambda\right)F(\lambda).
\end{align*}
By \cite[p. 18 and Lemma 3.1]{bl22}, we obtain, for any $\ell\in\mathbb{Z}$, $i\in\mathbb{N}$,
and $g\in L^2(B)$,
\begin{align*}
\left\|G_{\ell,r_B}\right\|_{L^\infty(0,\infty)}\lesssim2^{-\ell s}\left(2^\ell r_B^2\right)^{-M}
\min\left\{1,\left(2^\ell r_B^2\right)^{2M}\right\}
\end{align*}
and
\begin{align*}
&\left\|G_{\ell,r_B}g\right\|_{L^2(S_i(B_t))}\\
&\quad\lesssim2^{-hi}\left(2^\ell r_B^2\right)^{-M-\frac{h}{2}}\max\left\{1,2^{\theta \ell}\right\}
\min\left\{1,\left(2^\ell r_B^2\right)^{2M}\right\}\|g\|_{L^2(B)}.
\end{align*}
Using this, by an argument similar to that used in the proof of \eqref{1959} with $F_{\ell,r_B}$ and
$\alpha$ replaced, respectively, by $G_{\ell,r_B}$ and $r_B^{-2M}\alpha$, we
conclude that \eqref{2114} holds true.

By \eqref{1959}, Proposition \ref{pras} with $q:=2$ and $\theta:=2^i(1+|t|),$
and \eqref{can1}, we have, for any $i\in\mathbb{Z}_+$,
\begin{align}\label{2126}
E_{i,1}
&=\left\|\sum_{j\in\mathbb{N}}\left\{\lambda_{j}S_{L,\,\psi}
\left(\left[I-e^{-r_{B_j}^2L}\right]^MF(L)\alpha_{j}\right)\mathbf{1}_{S_i(B_{t,j})}\right\}^{s_0}
\right\|_{X^{\frac{1}{s_0}}(\mathcal{X})}\\\notag
&\lesssim
[2^i(1+|t|)]^{(1-\frac{s_0}{2})n}2^{-iss_0}\left\|\sum_{j=1}^\infty
\left[\frac{\lambda_j}{\|\mathbf{1}_{B_j}\|_{X(\mathcal{X})}}\right]^{s_0}
\mathbf{1}_{B_j}\right\|_{X^{\frac{1}{s_0}}(\mathcal{X})}\\\notag
&\lesssim
2^{-is_0[s-n(\frac{1}{s_0}-\frac{1}{2})]}(1+|t|)^{(1-\frac{s_0}{2})n}\|f\|_{H_{X,\,L}(\mathcal{X})}^{s_0}.
\end{align}
Meanwhile, similarly to \eqref{2126}, we also have, for any $i\in\mathbb{Z}_+,$
\begin{align*}
E_{i,2}\lesssim
2^{-is_0[s-n(\frac{1}{s_0}-\frac{1}{2})]}(1+|t|)^{(1-\frac{s_0}{2})n}\|f\|_{H_{X,\,L}(\mathcal{X})}^{s_0}.
\end{align*}
From this, \eqref{2126}, and \eqref{2128}, it follows that \eqref{2127} holds true, which completes the proof
of Theorem \ref{sch}.
\end{proof}

\section{Applications to Specific Function Spaces}\label{section5}
In this section, we apply the main results obtained in Sections \ref{section3} and \ref{section4} to
some specific spaces associated with operators, including  Orlicz--Hardy spaces
(see Subsection \ref{secos} below), weighted Hardy spaces (see Subsection \ref{secwls} below), and variable
Hardy spaces (see Subsection \ref{secvls} below). These examples show that the results obtained in this article
have wide generality and applications.

\subsection{Orlicz--Hardy Spaces}\label{secos}
Let us first recall the concepts of Orlicz functions and Orlicz spaces. A function $\Phi:\ [0,\infty)
\to[0,\infty)$ is called an \emph{Orlicz function} if $\Phi$ is non-decreasing, $\Phi(0)=0$, $\Phi(t)>0$
for any $t\in(0,\infty)$, and $\lim_{t\to\infty}\Phi(t)=\infty.$ Moreover, an Orlicz function $\Phi$ is
said to be of \emph{lower} [resp. \emph{upper}] \emph{type} $r$ for some $r\in\mathbb{R}$ if there exists
a positive constant $C_{(r)}$ such that, for any $t\in[0,\infty)$ and $s\in(0,1)$ [resp. $s\in[1,\infty)$],
$$\Phi(st)\le C_{(r)} s^r\Phi(t).$$
In what follows, we \emph{always} assume that $\Phi:\ [0,\infty)\to[0,\infty)$ is an Orlicz function with
positive lower type $r_{\Phi}^-$ and positive upper type $r_{\Phi}^+$. The \emph{Orlicz space $L^\Phi(\mathcal{X})$}
is defined to be the set of all the $\mu$-measurable functions $f$ on $\mathcal{X}$ with finite \emph{(quasi-)norm}
\begin{align*}
\|f\|_{L^\Phi(\mathcal{X})}:=\inf\left\{\lambda\in(0,\infty):\ \int_{\mathcal{X}}\Phi\left(\frac{|f(x)|}
{\lambda}\right)\,d\mu(x)\le1\right\}.
\end{align*}
We point out that $L^\Phi(\mathcal{X})$ is a quasi-Banach function space and hence a $\mathrm{BQBF}$ space (see, for
instance, \cite[Subsection 8.3]{syy21}). In the case when $X(\mathcal{X}):=L^{\Phi}(\mathcal{X})$, we denote
$H_{X,\,L}(\mathcal{X})$, $H_{X,\,L,\,\rm{mol}}^{M,\,\epsilon}(\mathcal{X})$, $H_{X,\,L,\,\rm{at}}^{M}(\mathcal{X})$, $H_{X,\,L,\,g}(\mathcal{X})$, and $H_{X,\,L,\,g^\ast_{\lambda}}(\mathcal{X})$, respectively, by
$H_{L}^{\Phi}(\mathcal{X})$, $H_{L,\,\rm{mol}}^{\Phi,\,M,\,\epsilon}(\mathcal{X})$,
$H_{L,\,\rm{at}}^{\Phi,\,M}(\mathcal{X})$, $H_{L,\,g}^{\Phi}(\mathcal{X})$,
and $H_{L,\,g^\ast_{\lambda}}^{\Phi}(\mathcal{X})$. Applying Theorem \ref{thm-mc} with $X(\mathcal{X}):=L^{\Phi}
(\mathcal{X})$, we have the following conclusion.

\begin{theorem}\label{os1}
Let the operator $L$ be the same as in Theorem \ref{thm-mc} and $\Phi$ be an Orlicz function with positive
lower type $r_{\Phi}^-$ and positive upper type $r_{\Phi}^+$. Assume that $0<r_{\Phi}^-\leq r_{\Phi}^+<2$,
$M\in(\frac{n}{2}[\frac{1}{\min\{1,r_{\Phi}^-\}}-\frac{1}{2}],\infty)\cap\mathbb{N}$, and $\epsilon\in(\frac{n}
{\min\{1,r_{\Phi}^-\}},\infty)$. Then the spaces $H_{L}^{\Phi}(\mathcal{X})$, $H_{L,\,\rm{mol}}
^{\Phi,\,M,\,\epsilon}(\mathcal{X})$, and $H_{L,\,\rm{at}}^{\Phi,\,M}(\mathcal{X})$ coincide with equivalent
quasi-norms.
\end{theorem}

\begin{proof}
Let $p:=r_{\Phi}^-$, $q_0:=2,$ and $s_0\in (0,\min\{1,r_{\Phi}^-\})$ satisfy both $M>\frac{n}{2}
(\frac{1}{s_0}-\frac{1}{2})$ and $\epsilon>\frac{n}{s_0}$. By \cite[Theorem 6.6]{fmy20}, we find that
$X(\mathcal{X}):=L^\Phi(\mathcal{X})$ satisfies Assumption \ref{vector1} for the aforementioned $p$. From
\cite[p. 61, Proposition 4 and p. 100, Proposition 1]{rr91}, we deduce that
\begin{align}\label{qi}
\left[\left(\left[L^\Phi(\mathcal{X})\right]^{\frac{1}{s_0}}\right)'\right]
^{\frac{1}{(2/s_0)'}}=L^{\Psi}(\mathcal{X}),
\end{align}
where, for any $t\in[0,\infty)$,
$$\Psi(t):=\sup_{h\in(0,\infty)}\left[t^{1/(2/s_0)'}h-\Phi\left(h^{1/s_0}\right)\right].$$
Moreover, by \cite[Proposition 7.8(i)]{shyy17}, we conclude that $\Psi$ is an Orlicz function with positive
lower type $r_{\Psi}^-:=(r_{\Phi}^+/s_0)'/(2/s_0)'\in(1,\infty).$ This, combined with \eqref{qi} and
\cite[Theorem 6.6]{fmy20}, further implies that $X(\mathcal{X}):=L^\Phi(\mathcal{X})$  satisfies Assumption
\ref{vector2} for the aforementioned $s_0$ and $q_0.$ Therefore, all the assumptions of  Theorem
\ref{thm-mc} are satisfied with $X(\mathcal{X}):=L^{\Phi}(\mathcal{X})$, which further implies the
desired conclusions of the present theorem. This finishes the  proof of Theorem \ref{os1}.
\end{proof}

\begin{remark}
We point out that Theorem \ref{os1} under the assumption that $r_\Phi^+\in(0,1]$ (see also \cite[Theorem 5.5]{yy14}) was established in \cite[Theorem 5.5]{jy11}. To the best of our
knowledge,  Theorem \ref{os1} when $r_\Phi^+\in(1,2)$ is new.
\end{remark}

By Theorems \ref{thm-spec}, \ref{thm-g}, and \ref{sch} with $X(\mathcal{X}):=L^{\Phi}(\mathcal{X})$,
we have the  following conclusions; since their proofs are similar to that of Theorem \ref{os1},
we omit the details here.

\begin{theorem}\label{os4}
Let $L$ be the same as in Theorem \ref{thm-spec} and $\Phi$ be an Orlicz function with positive lower type
$r_{\Phi}^-$ and positive upper type $r_{\Phi}^+$. Assume that $0<r_{\Phi}^-\leq r_{\Phi}^+<2$
and $s\in(\frac{n}{\min\{1,r_{\Phi}^-\}},\infty)$. Let $\varphi\in C^\infty_{\mathrm{c}}(\mathbb{R})$ satisfy \eqref{620}
and  $m:\ [0,\infty)\to\mathbb{C}$ be a bounded Borel function satisfying \eqref{2022930c}. Then there exists a
positive constant $C$ such that, for any $f\in H^{\Phi}_{L}(\mathcal{X})$,
\begin{align*}
\|m(L)(f)\|_{H^{\Phi}_{L}(\mathcal{X})}\leq C\|f\|_{H^{\Phi}_{L}(\mathcal{X})}.
\end{align*}
\end{theorem}

\begin{theorem}\label{os2}
Let   $\Phi$ be an Orlicz function with positive lower type
$r_{\Phi}^-$ and positive upper type $r_{\Phi}^+$.
\begin{itemize}
\item [$\mathrm{(i)}$]Let $L$ be the same as in Theorem \ref{thm-g-2}. Assume that $0<r_{\Phi}^-\leq r_{\Phi}^+<2$ and
$\lambda\in(\frac{2n}{\min\{1,r_{\Phi}^-\}},\infty)$. Then
both $g_{L}$ and $g_{\lambda,\,L}^\ast$ are bounded from
$H_{L}^{\Phi}(\mathcal{X})$ to $L^{\Phi}(\mathcal{X}).$
\item [$\mathrm{(i)}$]Let $L$ be the same as in Theorem \ref{thm-g}. Assume that $0<r_{\Phi}^-\leq r_{\Phi}^+<\infty$ and
$\lambda\in(\frac{2n}{\min\{1,r_{\Phi}^-\}},\infty)$. Then the spaces $H_{L}^{\Phi}(\mathcal{X})$,
$H_{L,\,g}^{\Phi}(\mathcal{X}),$ and $H_{L,\,g^\ast_{\lambda}}^{\Phi}(\mathcal{X})$ coincide with equivalent quasi-norms.
\end{itemize}
\end{theorem}

\begin{theorem}\label{os3}
Let $L$ be the same as in Theorem \ref{sch} and $\Phi$ be an Orlicz function with positive lower type
$r_{\Phi}^-$ and positive upper type $r_{\Phi}^+$. Assume that $0<r_{\Phi}^-\leq r_{\Phi}^+<2$ and
$s\in(\frac{n}{\min\{1,r_{\Phi}^-\}}-\frac{n}{2},\infty)$. Then there exists a positive constant $C$ such that,
for any  $t\in\mathbb{R}$ and $f\in H^{\Phi}_{L}(\mathcal{X})$,
\begin{align*}
\left\|(I+L)^{-s}e^{itL}f\right\|_{H^{\Phi}_{L}(\mathcal{X})}\leq C(1+|t|)^s\|f\|_{H^{\Phi}_{L}(\mathcal{X})}.
\end{align*}
\end{theorem}

\begin{remark}
\begin{itemize}
\item [$\mathrm{(i)}$]  In the case when $r_{\Phi}^+\in (0,1]$, Theorem \ref{os4} was obtained in
\cite[Theorem 6.10]{yy14}. However, in the case when $r_{\Phi}^+\in (1,2)$, Theorem \ref{os4} is new.
\item [$\mathrm{(ii)}$] To the best of our knowledge, Theorems \ref{os2}(ii) and \ref{os3} are completely new. Moreover,
\cite[Theorems 6.3 and 6.7]{yy14} is a part of Theorem \ref{os2}(i).
\end{itemize}
\end{remark}

\subsection{Weighted Hardy Spaces}\label{secwls}
For any $\omega\in A_\infty(\mathcal{X})$, the \emph{critical index} $q_\omega$ of $\omega$ is defined by setting
\begin{align*}
q_\omega:=\inf\left\{p\in[1,\infty):\ \omega\in A_p(\mathcal{X})\right\}.
\end{align*}
We point out that, for any given $r\in(0,\infty)$ and $\omega\in A_\infty(\mathcal{X})$,
$L^r_\omega(\mathcal{X})$ is a $\mathrm{BQBF}$ space, but not necessarily a quasi-Banach function space
(see, for instance, \cite[Subsection 7.1]{shyy17} for the Euclidean space case). In the case when
$X(\mathcal{X}):=L^{r}_\omega(\mathcal{X})$, we denote $H_{X,\,L}(\mathcal{X})$, $H_{X,\,L,\,\rm{mol}}
^{M,\,\epsilon}(\mathcal{X}),$ $H_{X,\,L,\,\rm{at}}^{M}(\mathcal{X}),$ $H_{X,\,L,\,g}(\mathcal{X})$, and
$H_{X,\,L,\,g^\ast_{\lambda}}(\mathcal{X})$, respectively, by $H_{\omega,\,L}^{r}(\mathcal{X})$,
$H_{\omega,\,L,\,\rm{mol}}^{r,\,M,\,\epsilon}(\mathcal{X}),$
$H_{\omega,\,L,\,\rm{at}}^{r,\,M}(\mathcal{X}),$ $H_{\omega,\,L,\,g}^{r}(\mathcal{X})$, and
$H_{\omega,\,L,\,g^\ast_{\lambda}}^{r}(\mathcal{X})$. Applying Theorem \ref{thm-mc} with $X(\mathcal{X}):
=L^{r}_\omega(\mathcal{X})$, we obtain the  following conclusion.

\begin{theorem}\label{wls1}
Let the operator $L$ be the same as in Theorem \ref{thm-mc}. Assume that $r\in(0,2)$, $\omega\in A_\infty
(\mathcal{X})$, $s_0\in(0,\min\{1,r/q_\omega\})$, and $\omega^{1-(r/s_0)'}\in A_{(r/s_0)'/(2/s_0)'}(\mathcal{X})$.
Let $M\in(\frac{n}{2}[\frac{1}{s_0}-\frac{1}{2}],\infty)\cap\mathbb{N}$ and $\epsilon\in(\frac{n}{s_0},\infty)$.
Then the spaces $H_{\omega,\,L}^{r}(\mathcal{X})$, $H_{\omega,\,L,\,\rm{mol}}^{r,\,M,\,\epsilon}(\mathcal{X}),$ and
$H_{\omega,\,L,\,\rm{at}}^{r,\,M}(\mathcal{X})$ coincide with equivalent quasi-norms.
\end{theorem}
\begin{proof}
Let $p:=r/q_\omega$ and $q_0:=2.$ By \cite[Theorem 6.5(ii)]{fmy20} and the definition of $q_\omega,$
we find that $X(\mathcal{X}):=L^{r}_\omega(\mathcal{X})$ satisfies Assumption \ref{vector1} for the
aforementioned $p.$ Notice that
\begin{align*}
\left[\left(\left[L^{r}_\omega(\mathcal{X})\right]^{\frac{1}{s_0}}\right)'\right]^{\frac{1}{(2/s_0)'}}
=L^{(r/s_0)'/(2/s_0)'}_{\omega^{1-(r/s_0)'}}(\mathcal{X}).
\end{align*}
Using this and the assumption that $\omega^{1-(r/s_0)'}\in A_{(r/s_0)'/(2/s_0)'}(\mathcal{X})$, we conclude
that the Hardy--Littlewood maximal operator $\mathcal{M}$ is bounded on the $\frac{1}{(2/s_0)'}$-convexification
of the associate space $([L^{r}_\omega(\mathcal{X})]^{\frac{1}{s_0}})'$. Thus, all the assumptions of
Theorem \ref{thm-mc} are satisfied with $X(\mathcal{X}):=L^r_\omega(\mathcal{X})$, which further implies
the desired conclusions of the present theorem. This finishes the  proof of Theorem \ref{wls1}.
\end{proof}

By Theorems \ref{thm-spec}, \ref{thm-g}, and \ref{sch} with $X(\mathcal{X}):=L^{r}_\omega(\mathcal{X})$,
we have the following conclusions; since their proofs are similar to that of Theorem \ref{wls1}, we omit
the details here.

\begin{theorem}\label{wls4}
Let $L$ be the same as in Theorem \ref{thm-spec}, $r,\omega,$ and $s_0$ be  the same as in Theorem \ref{wls1}, and $s\in
(\frac{n}{s_0},\infty)$. Assume that $\varphi\in C^\infty_{\mathrm{c}}(\mathbb{R})$ satisfies \eqref{620}
and $m:\ [0,\infty)\to\mathbb{C}$ is a bounded Borel function satisfying \eqref{2022930c}.
Then there exists a positive constant $C$ such that, for any $f\in H^{r}_{\omega,\,L}(\mathcal{X})$,
\begin{align*}
\|m(L)(f)\|_{H^{r}_{\omega,\,L}(\mathcal{X})}\leq C\|f\|_{H^{r}_{\omega,\,L}(\mathcal{X})}.
\end{align*}
\end{theorem}

\begin{theorem}\label{wls2}
\begin{itemize}
\item [$\mathrm{(i)}$]
Let $L$ be the same as in Theorem \ref{thm-g-2}, $r,\omega,$ and $s_0$ be the same as in Theorem \ref{wls1}, and $\lambda
\in(\frac{2n}{s_0},\infty)$. Then
both $g_L$ and $g_{\lambda,\,L}^\ast$ are bounded from
$H_{\omega,\,L}^{r}(\mathcal{X})$ to $L^r_{\omega}(\mathcal{X}).$
\item [$\mathrm{(ii)}$]
Let $L$ be the same as in Theorem \ref{thm-g}.
Assume that $\omega\in A_\infty(\mathcal{X})$, $r\in(0,\infty)$, and
$\lambda\in(\frac{2n}{\min\{1,r/q_\omega\}},\infty).$
 Then the spaces $H_{\omega,\,L}^{r}(\mathcal{X})$, $H_{\omega,\,L,\,g}^{r}
(\mathcal{X}),$ and $H_{\omega,\,L,\,g^\ast_{\lambda}}^{r}(\mathcal{X})$ coincide with equivalent quasi-norms.
\end{itemize}
\end{theorem}

\begin{theorem}\label{wls3}
Let $L$ be the same as in Theorem \ref{sch}, $r,\omega,$ and $s_0$ be the same as in Theorem \ref{wls1}, and $s\in
(\frac{n}{s_0}-\frac{n}{2},\infty)$. Then there exists a positive constant $C$ such that, for any
 $t\in\mathbb{R}$ and $f\in H^{r}_{\omega,\,L}(\mathcal{X})$,
\begin{align*}
\left\|(I+L)^{-s}e^{itL}f\right\|_{H^{r}_{\omega,\,L}(\mathcal{X})}
\leq C(1+|t|)^s\|f\|_{H^{r}_{\omega,\,L}(\mathcal{X})}.
\end{align*}
\end{theorem}

\begin{remark}
\begin{itemize}
\item [$\mathrm{(i)}$] When $r\in(0,1]$,  Theorem \ref{wls4} in this case was obtained in
\cite[Theorem 6.10]{yy14} under a different assumption on the weight $\omega$ (see \cite[Theorem 6.10]{yy14}
for the details). Moreover, we point out that \cite[Theorem 1.2]{h17} is a part of Theorem \ref{wls2}(ii).
\item [$\mathrm{(ii)}$] To the best of our knowledge, Theorem \ref{wls3} is completely new.
\end{itemize}
\end{remark}

\subsection{Variable Hardy Spaces}\label{secvls}
We first recall the concept of variable Lebesgue spaces on $\mathcal{X}$. Let $r(\cdot):\ \mathcal{X}\to(0,\infty)$
be a $\mu$-measurable function,
$$
\widetilde{r}_-:=\underset{x\in\mathcal{X}}{\operatorname{ess\,inf}}\,r(x),\text{and}\ \
\widetilde r_+:=\underset{x\in\mathcal{X}}{\operatorname{ess\,sup}}\,r(x).
$$
The \emph{variable Lebesgue space $L^{r(\cdot)}(\mathcal{X})$} associated with the function $r:\ \mathcal{X}
\to(0,\infty)$ is defined to be the set of all the $\mu$-measurable functions $f$ on $\mathcal{X}$ with
the finite \emph{quasi-norm}
$$
\|f\|_{L^{r(\cdot)}(\mathcal{X})}:=\inf\left\{\lambda\in(0,\infty):\ \int_{\mathcal{X}}\left[\frac{|f(x)|}
{\lambda}\right]^{r(x)}\,d\mu(x)\le1\right\}.
$$
As was pointed out in \cite[Remark 2.7(iv)]{yhyy21a}, for any $r(\cdot):\ \mathcal{X}\to(0,\infty)$ with
$0<\widetilde{r}_-\leq \widetilde{r}_+<\infty$, the variable Lebesgue space $L^{r(\cdot)}(\mathcal{X})$
is a quasi-Banach function space and hence a $\mathrm{BQBF}$ space. The variable exponent $r(\cdot)$ is said to be
\emph{locally log-H\"older continuous} if there exists a positive constant $c_{\rm{log}}$ such that,
for any $x,y\in\mathcal{X}$,
$$
|r(x)-r(y)|\le \frac{c_{\rm{log}}}{\log(e+1/d(x,y))}
$$
and that $r(\cdot)$ is said to satisfy the \emph{log-H\"older decay condition} with a basepoint $x_r\in\mathcal{X}$
if there exists a $r_\infty\in\mathbb{R}$ and a positive constant $c_\infty$ such that, for any $x\in\mathcal{X},$
\begin{align*}
|r(x)-r_\infty|\leq \frac{c_\infty}{\log(e+d(x,x_r))}.
\end{align*}
The variable exponent $r(\cdot)$ is said to be \emph{log-H\"older continuous} if $r(\cdot)$ satisfies both the locally
log-H\"older continuous condition and the log-H\"older decay condition. In the case when $X(\mathcal{X}):
=L^{r(\cdot)}(\mathcal{X})$, we denote $H_{X,\,L}(\mathcal{X})$, $H_{X,\,L,\,\rm{mol}}^{M,\,\epsilon}(\mathcal{X})$,
$H_{X,\,L,\,\rm{at}}^{M}(\mathcal{X})$, $H_{X,\,L,\,g}(\mathcal{X})$, and $H_{X,\,L,\,g^\ast_{\lambda}}(\mathcal{X})$,
respectively, by $H_{L}^{r(\cdot)}(\mathcal{X})$, $H_{L,\,\rm{mol}}^{r(\cdot),\,M,\,\epsilon}(\mathcal{X})$,
$H_{L,\,\rm{at}}^{r(\cdot),\,M}(\mathcal{X})$, $H_{L,\,g}^{r(\cdot)}(\mathcal{X})$, and $H_{L,\,g^\ast_{\lambda}}
^{r(\cdot)}(\mathcal{X})$. Applying Theorem \ref{thm-mc} with $X(\mathcal{X}):=L^{r(\cdot)}(\mathcal{X})$, we
have the following conclusion.

\begin{theorem}\label{vls1}
Let the operator $L$ be the same as in Theorem \ref{thm-mc} and $r:\ \mathcal{X}\to(0,\infty)$ be log-H\"older
continuous. Assume that $0<\widetilde{r}_-\leq\widetilde{r}_+<2$, $M\in(\frac{n}{2}[\frac{1}{\min\{1,\widetilde{r}_-\}}
-\frac{1}{2}],\infty)\cap\mathbb{N}$, and $\epsilon\in(\frac{n}{\min\{1,\widetilde{r}_-\}},\infty)$. Then the spaces
$H_{L}^{r(\cdot)}(\mathcal{X})$, $H_{L,\,\rm{mol}}^{r(\cdot),\,M,\,\epsilon}(\mathcal{X})$, and $H_{L,\,\rm{at}}
^{r(\cdot),\,M}(\mathcal{X})$ coincide with equivalent quasi-norms.
\end{theorem}

\begin{proof}
Let $p:=\widetilde{r}_-,$ $q_0:=2,$ and $s_0\in (0,\min\{1,\widetilde{r}_-\})$ satisfy both $M>\frac{n}{2}
(\frac{1}{s_0}-\frac{1}{2})$ and $\epsilon>\frac{n}{s_0}.$ By \cite[Theorem 2.7]{zsy16}, we find that
$X(\mathcal{X}):=L^{r(\cdot)}(\mathcal{X})$ satisfies Assumption \ref{vector1} for the aforementioned
$p.$ Meanwhile, using \cite[Lemma 2.9]{zsy16}, we conclude that
\begin{align*}
\left[\left(\left[L^{r(\cdot)}(\mathcal{X})\right]^{\frac{1}{s_0}}\right)'\right]^{\frac{1}{(2/s_0)'}}
=L^{(r(\cdot)/s_0)'/(2/s_0)'}(\mathcal{X}),
\end{align*}
where $\frac{1}{r(x)/s_0}+\frac{1}{(r(x)/s_0)'}=1$ for any $x\in\mathcal{X}.$ From this, the assumption
that $0<\widetilde{r}_+<2,$ and \cite[Lemma 2.5]{zsy16}, it follows that $X(\mathcal{X}):=L^{r(\cdot)}
(\mathcal{X})$ satisfies Assumption \ref{vector2} for the aforementioned $s_0$ and $q_0.$ Therefore,
all the assumptions of Theorem \ref{thm-mc} are satisfied, which further implies the desired conclusions
of the present theorem. This finishes the proof of Theorem \ref{vls1}.
\end{proof}

\begin{remark}
We point out that, when $\widetilde{r}_+\in(0,1]$, Theorem \ref{vls1} was obtained in \cite[Theorem 3.3
and Proposition 5.12]{yz18}. However, when $\widetilde{r}_+\in(1,2)$, Theorem \ref{vls1} is new.
\end{remark}

By Theorems \ref{thm-spec}, \ref{thm-g}, and \ref{sch} with $X(\mathcal{X}):=L^{r(\cdot)}(\mathcal{X})$, we
have the following conclusions; since their proofs are similar to that of Theorem \ref{vls1}, we omit the
details here.

\begin{theorem}\label{vls4}
Let $L$ be the same as in Theorem \ref{thm-spec} and $r:\ \mathcal{X}\to(0,\infty)$ be log-H\"older continuous.
Assume that $0<\widetilde{r}_-\leq\widetilde{r}_+<2$ and $s\in(\frac{n}{\min\{1,\widetilde{r}_-\}},\infty)$.
Let $\varphi\in C^\infty_{\mathrm{c}}(\mathbb{R})$ satisfy \eqref{620} and $m:\ [0,\infty)\to\mathbb{C}$ be a
bounded Borel function satisfying \eqref{2022930c}. Then there exists a positive constant $C$ such that, for
any $f\in H^{r(\cdot)}_{L}(\mathcal{X})$,
\begin{align*}
\|m(L)(f)\|_{H^{r(\cdot)}_{L}(\mathcal{X})}\leq C\|f\|_{H^{r(\cdot)}_{L}(\mathcal{X})}.
\end{align*}
\end{theorem}

\begin{theorem}\label{vls2}
Let $r:\ \mathcal{X}\to(0,\infty)$ be log-H\"older continuous.
\begin{itemize}
\item [$\mathrm{(i)}$]Let $L$ be the same as in Theorem \ref{thm-g-2}.
Assume that $0<\widetilde{r}_-\leq\widetilde{r}_+<2$ and $\lambda\in(\frac{2n}{\min\{1,\widetilde{r}_-\}},\infty)$.
Then both $g_L$ and $g_{\lambda,\,L}^\ast$ are bounded from
$H_{L}^{r(\cdot)}(\mathcal{X})$ to $L^{r(\cdot)}(\mathcal{X}).$
\item [$\mathrm{(ii)}$]Let $L$ be the same as in Theorem \ref{thm-g}.
Assume that $0<\widetilde{r}_-\leq\widetilde{r}_+<\infty$ and $\lambda\in(\frac{2n}{\min\{1,\widetilde{r}_-\}},\infty)$.
Then the spaces $H_{L}^{r(\cdot)}(\mathcal{X})$, $H_{L,\,g}^{r(\cdot)}(\mathcal{X}),$ and $H_{L,\,g^\ast_{
\lambda}}^{r(\cdot)}(\mathcal{X})$ coincide with equivalent quasi-norms.
\end{itemize}
\end{theorem}

\begin{theorem}\label{vls3}
Let $L$ be the same as in Theorem \ref{sch} and $r:\ \mathcal{X}\to(0,\infty)$ be log-H\"older continuous.
Assume that $0<\widetilde{r}_-\leq\widetilde{r}_+<2$ and $s\in(\frac{n}{\min\{1,\widetilde{r}_-\}}-\frac{n}{2},
\infty)$. Then there exists a positive constant $C$ such that, for any   $t\in\mathbb{R}$ and $f\in H^{r(\cdot)}_{L}(\mathcal{X})
$,
\begin{align*}
\left\|(I+L)^{-s}e^{itL}f\right\|_{H^{r(\cdot)}_{L}(\mathcal{X})}
\leq C(1+|t|)^s\|f\|_{H^{r(\cdot)}_{L}(\mathcal{X})}.
\end{align*}
\end{theorem}

\begin{remark}
To the best of our knowledge, Theorems \ref{vls4}, \ref{vls2}, and \ref{vls3} are completely new.
\end{remark}

\bigskip

\noindent Xiaosheng Lin, Dachun Yang (Corresponding author) and Wen Yuan

\medskip

\noindent Laboratory of Mathematics and Complex Systems (Ministry of Education of China),
School of Mathematical Sciences, Beijing Normal University, Beijing 100875,
The People's Republic of China

\smallskip

\smallskip

\noindent {\it E-mails}: \texttt{xslin@mail.bnu.edu.cn} (X. Lin)

\noindent\phantom{{\it E-mails:} }\texttt{dcyang@bnu.edu.cn} (D. Yang)

\noindent\phantom{{\it E-mails:} }\texttt{wenyuan@bnu.edu.cn} (W. Yuan)

\bigskip

\noindent Sibei Yang

\medskip

\noindent School of Mathematics and Statistics, Gansu Key Laboratory of Applied Mathematics
and Complex Systems, Lanzhou University, Lanzhou 730000, The People's Republic of China

\smallskip

\noindent{\it E-mail:} \texttt{yangsb@lzu.edu.cn}

\end{document}